\documentclass[12pt]{article}
\title{Continuum limits of random matrices \\ and the Brownian carousel}
\author{Benedek Valk\'o
\footnote{Department of Mathematics, University of Wisconsin - Madison, WI 53705, USA. valko@math.wisc.edu}
 \hskip 3em B\'alint Vir\'ag\footnote{Department of Mathematics, University of Toronto, ON M5S
2E4, Canada. balint@math.toronto.edu} 
}

\date{}

\usepackage{palatino}
\usepackage{units}
\usepackage{amsfonts}
\usepackage{graphicx}
\usepackage{bm}
\usepackage{amsmath}
\usepackage{amsthm}
\usepackage[round]{natbib}

\newcommand{\idiota}{n_0/n}

\theoremstyle{plain}
    \newtheorem{theorem}{Theorem}
    \newtheorem{fact}[theorem]{Fact}
    \newtheorem{lemma}[theorem]{Lemma}
    \newtheorem{proposition}[theorem]{Proposition}
    \newtheorem{corollary}[theorem]{Corollary}
    
      \newtheorem{step}{Step}

\theoremstyle{definition} 
    
    \newtheorem{remark}[theorem]{Remark}

\theoremstyle{remark} 

\newcommand{\RRVlong}{Ram{\'\i}rez, Rider, and Vir\'ag \citeyearpar{RRV}}

\newcommand\mnote[1]{} 

\newcommand\be{\begin{equation}}

\newcommand\ee{\end{equation}}

\newcommand{\comment}[1]{}
\newcommand{\eps}{\varepsilon}
\newcommand{\Z}{{\mathbb Z}}
\newcommand{\ZZ}{{\mathbb Z}}
\newcommand{\UU}{{\mathbb U}}
\newcommand{\R}{{\mathbb R}}
\newcommand{\CC}{{\mathbb C}}

\newcommand{\N}{{\mathbb N}}
\newcommand{\HH}{{\mathbb H} }

\newcommand{\ev}{{\rm   E}}
\newcommand{\pr}{\mbox{\rm P}}
\newcommand{\one}{{\mathbf 1}}

\newcommand{\dist}{\mbox{\rm dist}}

\newcommand{\Arg}{\mbox{\rm Arg}}

\newcommand{\Poi}{\mbox{\rm Poi}}

\newcommand{\sm}{{\raise0.3ex\hbox{$\scriptstyle \setminus$}}}
\newcommand{\lcirc}{{\raise-0.15ex\hbox{$\scriptscriptstyle \circ$}}}

\newcommand{\FF}{{\mathcal F}}

\newcommand{\cd}{\stackrel{d}{\Longrightarrow}}

\renewcommand{\O}{\mathcal{O}}

\newcommand{\gl}{\lambda}

\newcommand{\vf}{\tilde\varphi}

\newcommand{\ftp}[1]{\left\lfloor #1 \right\rfloor_{2\pi}}
\newcommand{\rtp}[1]{\left\langle #1 \right\rangle_{2\pi}}
\newcommand{\frtp}[1]{\left\{ #1 \right\}_{2\pi}}

\newcommand{\RR}{{\mathbb R}}
\newcommand{\mm}{\mathcal M}
\newcommand{\lstar}{{\raise-0.15ex\hbox{$\scriptstyle \ast$}}}

\newcommand{\PSL}{\operatorname{PSL}(2,\RR)}
\newcommand{\UPSL}{\operatorname{{UPSL}}(2,\RR)}
\newcommand{\cp}{\stackrel{P}{\longrightarrow}}
\newcommand{\Sineb}{\operatorname{Sine}_{\beta}}
\newcommand{\Airyb}{\operatorname{Airy}_{\beta}}
\newcommand{\ep}{\varepsilon}

\newcommand{\ash}{\operatorname{ash}}

\newcommand{\tphh}{\hat\varphi^{ \odot}}
\newcommand{\tph}{\varphi^{ \odot}}

\newcommand{\Dfi}{\Delta \varphi}
\newcommand{\Dal}{\Delta \alpha}
\renewcommand{\Re}{\operatorname{Re}}
\renewcommand{\Im}{\operatorname{Im}}

\newcommand{\ha}{\hat \alpha}

\newcommand{\nnn}{\mu_n^2/4}
\newcommand{\varupsl}[1]{ \mathbf #1}
\newcommand{\nn}{  \mu_n}

\begin{document}
\bibliographystyle{balint}
\maketitle
\begin{abstract}
We show that at any location away from the spectral edge, the
eigenvalues of the Gaussian unitary ensemble and its general
$\beta$ siblings converge to $\Sineb$, a translation invariant
point process. This process has a geometric description in term of
the Brownian carousel, a deterministic function of Brownian motion
in the hyperbolic plane.

The Brownian carousel, a description of the a continuum
limit of random matrices, provides a convenient way to
analyze the limiting point processes. We show that the gap
probability of $\Sineb$ is continuous in the gap size and
$\beta$, and compute its asymptotics for large gaps.
Moreover, the stochastic differential equation version of
the Brownian carousel exhibits a phase transition at
$\beta=2$.
\end{abstract}

\begin{figure}[ht]
\begin{center}
\includegraphics*[width=370pt]{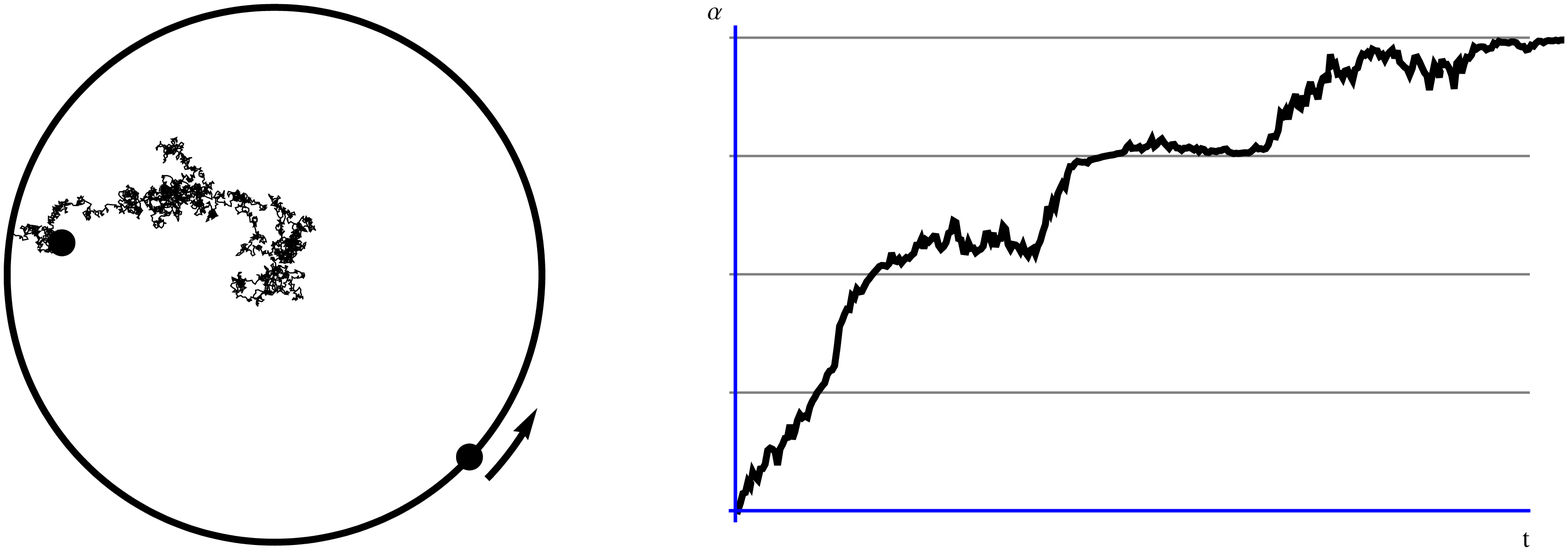}
\\ {The Brownian carousel and the winding angle $\alpha_\lambda$}
\end{center}
\end{figure}
\newpage
\tableofcontents
\newpage

\section{Introduction}

The Gaussian orthogonal and unitary ensembles are the most
fundamental objects of study in random matrix theory. In
the past decades, their eigenvalue distribution has shown
to be important in several areas of probability,
combinatorics, number theory, operator algebras, even
engineering (see \cite{Deift} for an overview). For
dimension $n$, the ordered eigenvalues $\lambda_1\le \ldots
\le \lambda_n \in \RR$ have joint density
\begin{equation}
\label{betadens}
 \frac{1}{Z_{n,
     \beta}}\;
e^{ - \beta \sum_{k=1}^n \lambda_k^2/4} \, \prod_{j < k} |
\lambda_j - \lambda_k |^{\beta},
\end{equation}
where $\beta=1,2$ for the Gaussian orthogonal and unitary
 ensembles, respectively. The above density makes sense for
any $\beta\ge 0$, and the point process is often called
Coulomb gas in Gaussian potential at inverse temperature
$\beta$. The goal of this paper is to study its
$n\to\infty$ point process limit away from the spectral
edge.

The limit is described via a special case of the  {\bf
hyperbolic carousel}. Let
\begin{itemize}
 \item $b$ be a path in the hyperbolic plane \vspace{-.5em}
 \item $z$ be a point on the boundary of the hyperbolic plane,
 and \vspace{-.5em}
 \item $f:\RR_+\to \RR_+$ be an integrable function.
\end{itemize}
To these three objects, the hyperbolic carousel associates a
multi-set of points on the real line defined via its counting
function $N(\gl)$ taking values in $\Z\cup\{-\infty,\infty\}$.
As time increases from $0$ to $\infty$, the boundary point $z$ is
rotated about the center $b(t)$ at angular speed $\lambda f(t)$.
$N(\gl)$ is defined as the integer-valued total winding number of the
point about the moving center of rotation.

The {\bf Brownian \bf carousel} is defined as the
hyperbolic carousel driven by hyperbolic Brownian motion
$b$. See Section \ref{s_carousel} for more details.

In order to study the $n\to\infty$ limit of \eqref{betadens} we
need to pick the center $\mu_n$ of the scaling window for each
$n$. Then the scaling factor follows the Wigner semicircle law.
Our main theorem gives necessary and sufficient conditions on
$\mu_n$ to get a bulk-type limit.

\begin{theorem}\label{t_main} For $\beta>0$, let $\Lambda_n$ denote the point process given
by \eqref{betadens}, and let $\mu_n$ be a sequence so that \
$n^{\nicefrac{1}{6}}(2\sqrt{n}-|\mu_n|)\to +\infty$. Then
\begin{equation}\nonumber
\sqrt{4n-\mu_n^2}\big(\Lambda_n-\mu_n\big) \Rightarrow \Sineb,
\end{equation}
where $ {\rm Sine}_\beta$ is the discrete point process given by
the Brownian carousel with parameters $f(t)=(\beta/4)e^{-\beta
t/4}$ and arbitrary $z$.
\end{theorem}

The convergence here is in law with respect to vague topology for
the counting measure of the point process. The limit and
convergence for the special values $\beta=1,2,4$ under  more
restrictive scaling conditions has been well-studied, see
\cite{mehta} or \cite{ForBook}. The Brownian carousel description is novel even in
these special cases.
We note that the ensemble (\ref{betadens}) may be
generalized by replacing the $\sum_k \gl_k^2$ in the
exponent by a similar sum involving a fixed function
$V$ of the eigenvalues. Assuming certain growth conditions
on $V$ the corresponding problem in the $\beta=2$
case can be treated using orthogonal polynomials and
Riemann-Hilbert methods, see e.g.~\cite{Deift}, \cite{DKM}.

Together with the following theorem, Theorem
\ref{t_main} gives a complete characterization of the possible
limiting processes for the ensembles \eqref{betadens}.
\begin{theorem}[\RRVlong]\label{t_rrv}
For $\beta>0$, let $\Lambda_n$ denote the point process given by
\eqref{betadens}, and  let $\mu_n$ be a sequence so that
$n^{\nicefrac{1}{6}}(2\sqrt{n}-\mu_n)\to a\in \RR$. Then
\begin{equation}\nonumber
n^{\nicefrac{1}{6}}(\Lambda_n-\mu_n) \Rightarrow \Airyb+\,a
\end{equation}
\end{theorem}
Here $\Airyb$ is defined as $-1$ times the point process of
eigenvalues of the stochastic Airy operator, see \RRVlong\,
for more details. A straightforward diagonalization
argument gives the following corollary, which is proved in
Section \ref{s_proof}.

\begin{corollary}\label{maincor}
As $a\to\infty$ we have
$2\sqrt{a}(\operatorname{Airy}_\beta+a)\Rightarrow \Sineb$.
\end{corollary}

The proof of Theorem 1 is based on the tridiagonal matrix
models introduced by \cite{Trotter} and \cite{DE}.
\cite{Sutton} and \cite{ES} present heuristics that the
operators given by the tridiagonal matrices have a limit
whose eigenvalues give the Sine and Airy processes. Theorem
\ref{t_rrv} shows that this is indeed the case at the
spectral edge.
%
The bulk case, however, is fundamentally different: there seems to
be no natural limiting operator with the spectrum given by the
Sine point process.
Rather than taking a limit of the operator itself, we
consider limits of discrete variants of the phase functions
in the Sturm-Liouville theory.  This connection is explored
further in Section \ref{s_sch}, where we describe how the
Sine point process appears as a universal limit for a large
class of one-dimensional Schr\"odinger operators.

The eigenvalue equation of a real tridiagonal matrix gives a
three-term linear recursion for the eigenvectors. This becomes a
two-term recursion for the ratios of consecutive entries, which
then evolves by linear fractional transformations fixing the real
line. So in our case, the evolution operators perform a
time-inhomogeneous random process in $\PSL$, the group of
orientation-preserving isometries of the hyperbolic plane. To get
the Brownian carousel, we regularize this evolution and take
limits. An important tool is Proposition \ref{p_turboEK} (based on
the results of \cite{SV}), which yields stochastic differential
equation limits of Markov processes with heavy local oscillations.

The Brownian carousel description gives a simple way to
analyze the limiting point process. The hyperbolic angle of
the rotating boundary point as measured from $b(t)$ follows
the \textbf{stochastic sine equation}, a coupled one-parameter
family of stochastic differential equations
\begin{equation}\label{e_ssefenetudja}
d\alpha_\gl=  \gl  f\,dt + \Re ((e^{-i\alpha_\gl}-1)dZ), \qquad
\alpha_\gl(0)=0,
\end{equation}
driven by a two-dimensional standard Brownian motion. For a
single $\lambda$, this reduces to the one-dimensional
stochastic differential equation
\begin{equation}\label{e sse3a}
d\alpha_\gl=   \gl f\,dt + 2 \sin(\alpha_\gl/2)dW,\qquad
\alpha_\gl(0)=0,
\end{equation}
which converges as $t\to\infty$ to an integer multiple
$\alpha_\lambda(\infty)$ of $2\pi$. A direct consequence of
the definition of $\Sineb$ is the following.
\begin{proposition}\label{sseprop}
The number of points $N(\gl)$ of the point process $\Sineb$ in
$[0,\lambda]$ has the same distribution as
$\alpha_\lambda(\infty)/(2\pi)$.
\end{proposition}
%

\begin{figure}\nonumber
\begin{center}
\includegraphics*[width=125pt]{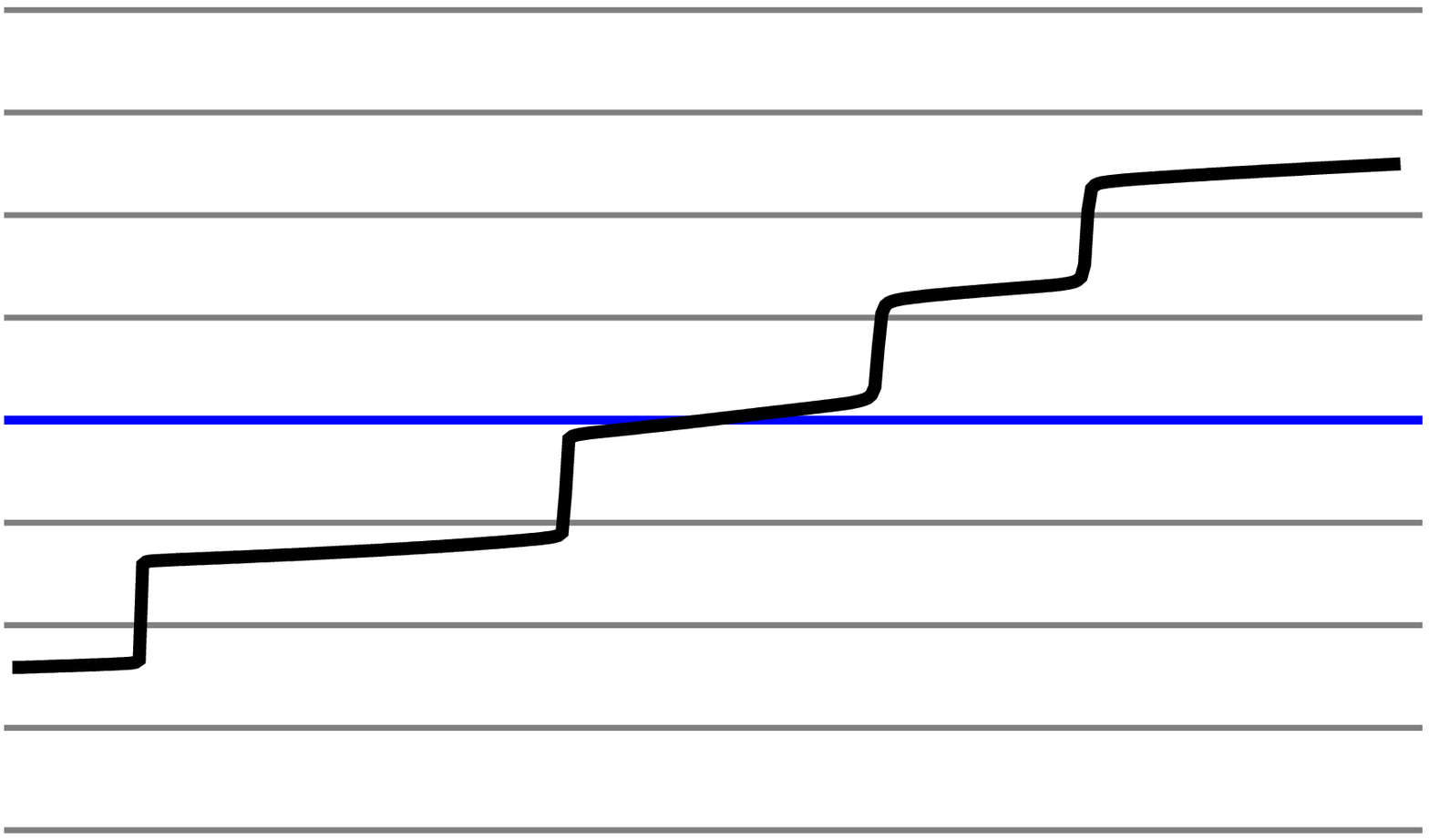} \phantom{MM}
\includegraphics*[width=125pt]{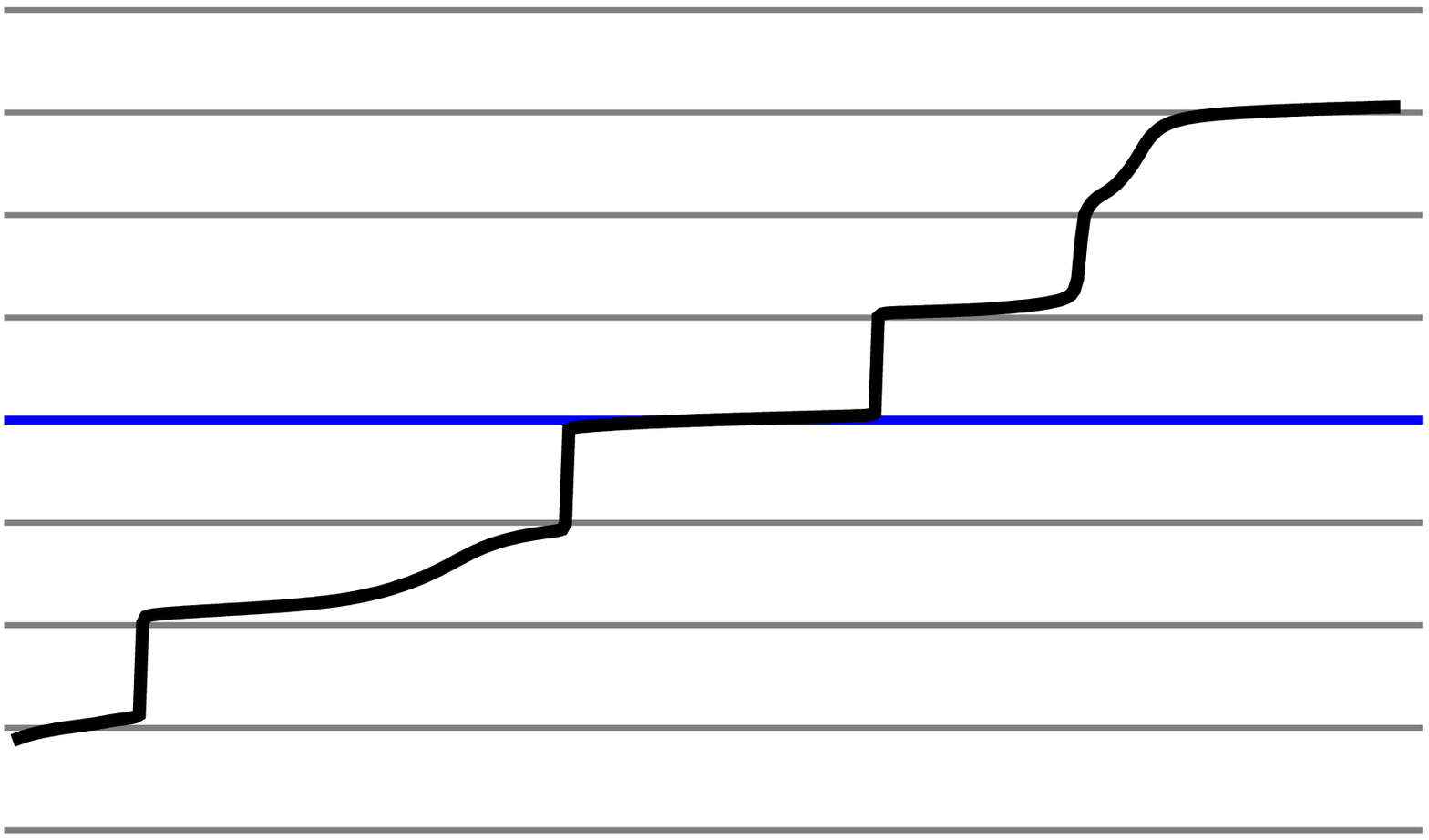} \phantom{MM}
\includegraphics*[width=125pt]{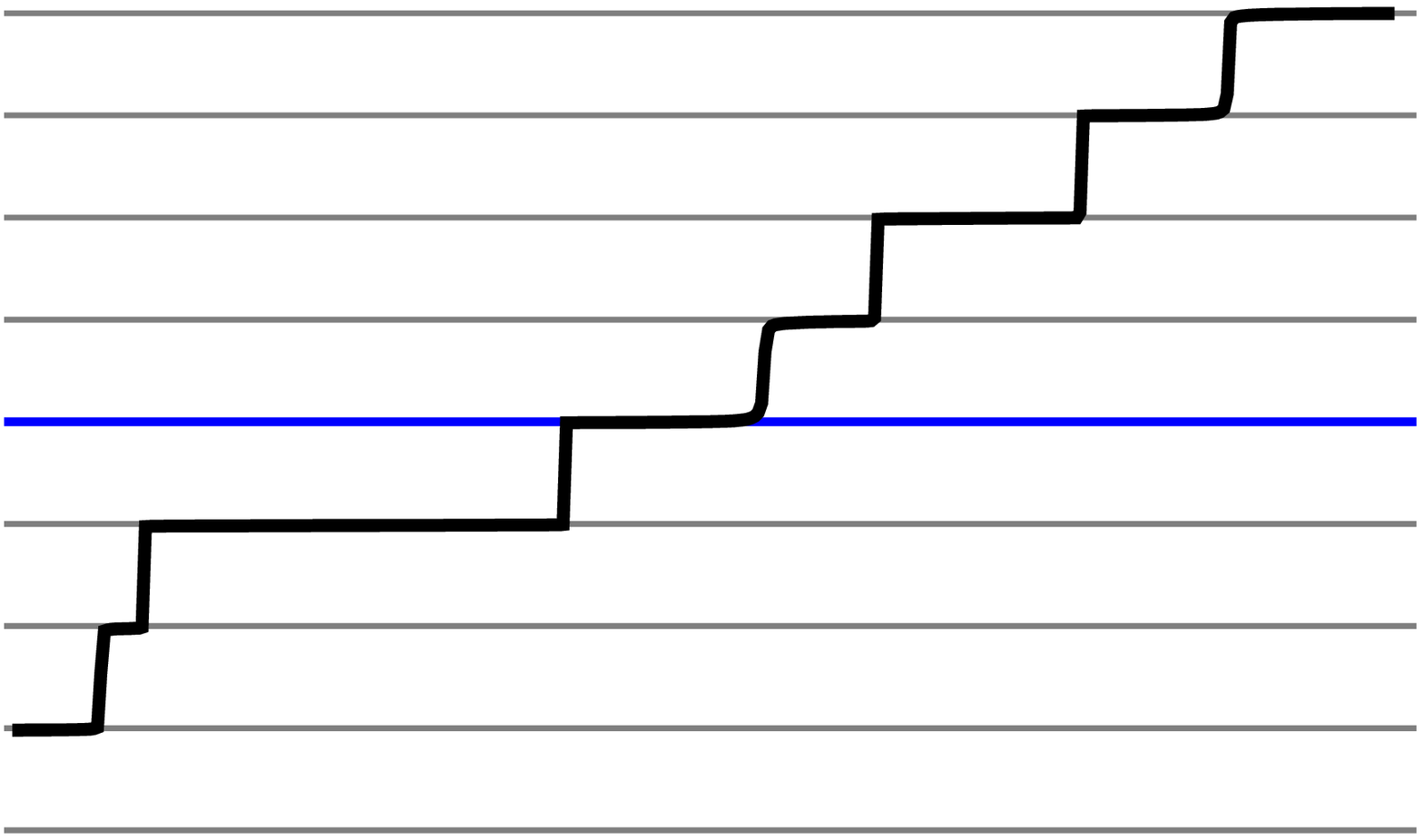}
\end{center}
\caption{The $\beta=1$ stochastic sine equation as a function of
$\lambda$ at three  times}
\end{figure}

Convergence to the solution of the coupled SDEs is the result
formally announced in the lecture by \cite{announce}. In
independent work, \cite{ks} present a related but different
description of the limit processes in the setting of circular
ensembles (see, e.g.\ \cite{ForBook}, Chapter 2 for discussion of
these models and \cite{k2} for further related results).

Proposition \ref{sseprop} allows us to analyze the point process
$\Sineb$,  for example to determine the asymptotics of large gap
probabilities. This has been predicted by \cite{Dy62} and proved
for the cases $\beta=2$ by \cite{Wi96} and for $\beta=1, 4$ by
\cite{JMMS80}; there, more refined asymptotics are presented; see
also \cite{DIZ96}.
\begin{theorem}
For $k\ge 0$ fixed and $\lambda \to \infty$, we have
$$
\pr(\mbox{\# of points in }[0,\lambda]\le k) =
\exp\big(-\lambda^2( \beta/64 +o(1))\big).
$$
\end{theorem}
This is shown  in Section \ref{s_gap} for the case of a
more general parameter  $f$. Several similar asymptotic
identities can be computed this way, and continuity
properties can be studied. For the $\Sineb$ processes we
have
\begin{proposition}\label{p_continuity}
The probability distribution of $N(\lambda)$ is a continuous
function of $\lambda$ and $\beta$.
\end{proposition}
In contrast, the stochastic sine equation exhibits a phase
transition at $\beta=2$.
\begin{theorem}\label{t_pt}
For any $\lambda>0$ we have a.s.
\begin{equation}\label{pt1} \mbox{for all }t
\mbox{ large enough }\alpha_\lambda(t)\ge \alpha_\lambda(\infty)
\end{equation}
 if and only if
$ \beta \le 2$. In particular, the probability of the event
\eqref{pt1} is not analytic at $\beta=2$ as a function of
$\beta$.
\end{theorem}

Deift (personal communication, 2007) asked whether this
phase transition also appears on the level of gap probabilities.
This question remains open.


\section{The Brownian carousel and the stochastic sine equation}\label{s_carousel}

\subsection{Definitions}

In the Poincar\'e disk model of the hyperbolic plane a boundary point can be described by an angle.
The {\bf Brownian carousel ODE} with parameters $f(t)$ and $z_0$
describes the evolution of the lifted angle $\gamma_\gl(t)$ with $e^{i
\gamma_\gl(0)}=z_0$ as it is rotated about the center $B(t)$ at angular
speed $\gl f(t)$. Here $B(t)$ is hyperbolic Brownian motion, that is
the strong solution of the SDE
\begin{equation}\nonumber
\qquad dB = \frac{(1-|B|^2)}{2}\,d\tilde Z
\end{equation}
driven by complex Brownian motion $\tilde Z$ with standard real and imaginary parts. The speed of $\gamma_\gl$, as measured in units of
boundary harmonic measure from $B$, is $\gl f/(2\pi)$. To change to an
angle measured from $0$, we need to divide by the Poisson kernel
$$
\Poi(e^{i\gamma_\gl},w)
  =\frac{1}{2\pi} \Re\,\frac{e^{i\gamma_\gl}+w}{e^{i\gamma}-w}
  =\frac{1}{2\pi} \frac{1-|w|^2}{|e^{i\gamma_\gl}-w|^2},
$$
which yields the ODE
\begin{equation}\label{e bcode}
\partial_t \gamma_\gl = \frac{\gl f}{2\pi\,\Poi(e^{i\gamma_\gl},B)}=\gl f \frac{|e^{i\gamma_\gl}-B|^2}{1-
|B|^2}.
\end{equation}
The most convenient way to define the winding number $N(\lambda)$
of $e^{i\gamma_\lambda}$ about the moving center of rotation
$B(t)$ is to follow the corresponding angle. Let  $\alpha_\gl(t)$
denote the hyperbolic angle determined by the points $z_0, B(t)$
and $e^{i \gamma_\gl(t)}$. As we will check at the end of this
section, It\^o's formula shows that $\alpha$ satisfies the {\bf
stochastic sine equation} \eqref{e_ssefenetudja}, i.e.\
\begin{equation}\label{e sse2}
d\alpha_\gl=   \gl f\,dt + \Re ((e^{-i\alpha_\gl}-1)dZ),
\qquad \alpha_\gl(0)=0,
\end{equation}
where $dZ$ is simply complex white noise with standard real
and imaginary parts. The name of the SDE comes from the
fact that the last term equals $2 \sin(\alpha_\gl/2) \Im
(e^{-i\alpha_\gl/2}dZ)$. Since $dW=\Im (
e^{-i\alpha_\gl/2}dZ)$ is 1-dimensional white noise, we get
the SDE \eqref{e sse3a} for the single $\lambda$ marginals.

Propostion \ref{l_sdeprop} of the next section shows that
\begin{equation}\label{Ndef}
\frac{1}{2\pi}\lim_{t\to \infty} \alpha_\lambda (t)
\end{equation}
exists for every $\lambda$ a.s.\ and for every
$\lambda_1<\lambda_2$ a.s.\ $N({\lambda_1})\le
N({\lambda_2})$. Thus  $N(\lambda)$ can be defined as the
unique random right-continuous function which agrees with
\eqref{Ndef} for every $\lambda$ a.s.
%
%


To deduce \eqref{e sse2}, let ${\mathcal T}(w,z)$ denote the
M\"obius automorphism of the unit disk taking $z_0$ to $1$ and
taking $w$ to $0$. It is given by the formula
\begin{equation}\label{e_mobi}
{\mathcal T}(w,z)=\frac{S(w,z)}{S(w,z_0)}, \qquad S(w,z) =
\frac{z-w}{1-\overline w z}.
\end{equation}
Then $\alpha$ is defined as the continuous solution of
\begin{equation}\label{e_car_sse}
\alpha(0)=0, \qquad e^{i\alpha(t)}={\mathcal
T}(B_t,e^{i\gamma(t)}).
\end{equation}
The stochastic sine equation \eqref{e sse2} follows from
taking logarithms and applying It\^o's formula. For the
driving Brownian motion we get the explicit expression
\begin{equation}\nonumber 
dZ = {2\partial_2{\mathcal T}(B,B)}dB = \frac{2}{1-|B|^2}\;
\frac{1-\overline B}{1-B}\ dB.
\end{equation}

\begin{remark} By It\^o's formula applied to the logarithm of \eqref{e_car_sse}, the noise
term in  \eqref{e sse2} can be interpreted as the infinitesimal
movement of the angle $\alpha$ under the difference of
transformations $d{\mathcal T} ={\mathcal T}(B+dB,\gamma){\mathcal T}(B, \gamma)^{-1}$. This
infinitesimal M\"obius transformation $d{\mathcal T}$ moves 0 to
${\mathcal T}(B,B+dB)=\partial_2 {\mathcal T}(B,B)dB$, a standard complex Brownian
motion increment. Such a transformation $d{\mathcal T}$ changes the angle of
any two points on the boundary by a Brownian increment with
standard deviation proportional to their distance.
 This gives a
more conceptual explanation of the noise term in \eqref{e sse2}.
\end{remark}

\subsection{Properties of the Brownian carousel}

Let  $L^1_*$ denote the set of absolutely integrable
functions of $\RR^+$ which tend to $0$  at  $+\infty$.
Given a hyperbolic Brownian motion and a boundary point
$z_0$, the Brownian carousel associates a random counting
function $N(\gl)$ to each $f\in L^1_*$. More generally, it
is fruitful to study how $N(\lambda)$ changes when the
parameter $f$ varies but the Brownian path remains fixed.
In this case $\lambda$ can be absorbed in the parameter $f$
so we will use the notation $N_f=N_f(1)$, and $\alpha_f$
for the case $\lambda=1$.

%
\renewcommand{\theenumi}{(\roman{enumi})}
\renewcommand{\labelenumi}{\theenumi}
\begin{proposition}[Properties of the Brownian carousel]\label{l_sdeprop}
 We have
\begin{enumerate}
\item \label{qqi}  $\alpha_f-\alpha_g$ has the same distribution
as $\alpha_{f-g}$,

\item \label{qi} $\alpha_f(t)$ is increasing in $f$,

\item \label{szelep} $\ftp{\alpha_f(t)}$ is nondecreasing in $t$ when  $f\ge 0$.
Here  $\ftp{x}=\max\left (2\pi \Z \cap (-\infty,x]\right)$.
 \item
\label{qii} $N_f=\frac{1}{2\pi}\lim_{t\to \infty}
\alpha_f(t)$ exists and is an integer a.s.,
 \item \label{qiib} $N_0=0$ and $N_f$
is increasing in $f$, \item \label{qiii} $\ev |N_f| \le
\frac{1}{2\pi}\|f\|_1, $
 \item \label{qiv}  $\ev N_f=\frac{1}{2\pi}\int_0^\infty f(x)\,dx$, and
\item \label{qv}  $N_f$ has exponential tails. For integers
$a,k>0$ we have
$$\pr(|N_f|\ge a k )\le 2 \left[\frac{\|f\|_1}{2\pi a}\right]^k.$$
\end{enumerate}
\end{proposition}
\begin{proof}
Claim \ref{qqi} holds because $\alpha_f-\alpha_g$ solves
\begin{equation}\nonumber
d\alpha_\gl=   \gl (f-g)\,dt + \Re ((e^{-i\alpha_\gl}-1)dZ^*),
\qquad \alpha_\gl(0)=0,
\end{equation}
with $dZ^*=e^{-i \alpha_g} dZ$. The standard coupling argument
shows that the solution of the stochastic sine equation is
monotone in the drift term, so we get \ref{qi}.

Now  assume that $f\ge g\equiv 0$. Then by the above
$\alpha_f(t)\ge \alpha_{g}(t)=0$. Claim \ref{szelep}
follows by repeating this argument for the process after
the hitting time of $2k\pi$.

Assume $f\ge 0$, and let $F(t)=\int_0^tf(s) \, ds$. Then
$\alpha_f-F$ is a continuous local martingale which is uniformly
bounded below by $-\|f\|_1$. Thus it a.s.\ converges to a random
limit. So $\alpha$ also converges, but it can only converge to a
location where the noise term vanishes; we get \ref{qii}, and
\ref{qiib} also follows from \ref{qi}. Now \ref{qiii} follows from
 \begin{eqnarray*}
 2\pi \ev N_f-\|f\|_1
 =\ev \alpha_f(\infty)-F(\infty)
 &\le& \lim_{t\to\infty} (\ev \alpha_f(t)-F(t))
 \\ &=& \alpha_f(0)-F(0)=0,
 \end{eqnarray*}
where the inequality is by Fatou's Lemma.
By \ref{szelep}  the function
$t\mapsto \ftp{\alpha_f(t)}$ is nondecreasing, hence the above
inequality implies that $\ftp{\alpha_f(t)}$ is uniformly
integrable, and so is $\alpha_f(t)$. Thus $\alpha_f-F$ is a
uniformly integrable martingale and so $\ev
\alpha_f(\infty)=F(\infty)$, as required for \ref{qiv}.

 For general $f\in L^1_*$, monotonicity \ref{qi} gives
\begin{equation}\label{e alphafpm}
\alpha_{-f^-}\le \alpha_{f} \le \alpha_{f^+},
\end{equation}
where $x^+=\max(x,0)$ and $x^-=(-x)^+$. Now $\alpha_{-f^-}$ has
the same distribution as $-\alpha_{f^-}$. By the previous argument
$\alpha_{f^+}$ and $\alpha_{f^-}$ are uniformly integrable. Hence
$\alpha_f-F$ is a uniformly integrable martingale, and \ref{qii},
\ref{qiv} follow. Claim  \ref{qiib} also follows via \ref{qi}. We
take positive  and negative part of \eqref{e alphafpm} to get $
\alpha_f^+ \le \alpha_{f^+} $, and $ \alpha_f^-\le
-\alpha_{-f^-}$, where the latter has the same distribution as
$\alpha_{f^-}$. Taking limits and expectations gives
$$
\ev |\alpha_f(\infty)| = \ev \alpha_f(\infty)^+ + \ev
\alpha_f(\infty)^- \le \ev \alpha_{f^+}(\infty) + \ev
\alpha_{f^-}(\infty) = \|f\|_1
$$
which gives \ref{qiii}.

Returning to $f\ge 0$, Markov's inequality implies that
$\pr(N_f\ge a)\le\frac{1}{2\pi}\|f\|_1/a$. Stopping the process at
time $\tau$ when and if $\alpha$ hits $2\pi ka$ we note that
$$
\pr \left(N_f\ge (k+1) a\,\big|\,N_f\ge k a,\mathcal
F_\tau\right)\le \frac{1}{2\pi} \|f_\tau\|_1/a \le
\frac{1}{2\pi}\|f\|_1/a,$$ where $f_\tau$ is $f$ shifted to the
left by $\tau$. It follows that for integer $k$ we have
$$
\pr(N_f\ge ka)\le \big[\frac{1}{2\pi}\|f\|_1/a\big]^k
$$
For general $f\in L^1_*$, we consider the positive and negative
parts separately and use monotonicity \ref{qiib} to get \ref{qv}.
\end{proof}

\begin{remark}\label{r_proc}
The previous lemma shows that for a fixed $f\in L^1_*$ the random
function $N(\gl)$ is a.s.~finite, integer valued, monotone
increasing with stationary increments. Thus $N(\gl)$ is the
counting function of a translation invariant point process. Since
$f\mapsto \alpha_f$ and $f\mapsto -\alpha_{-f}$ have the same
distribution, the distribution of the point process is symmetric
with respect to reflections.
\end{remark}

\begin{corollary}
For any $f\in L^1_*$ the point process defined by $N(\gl)$ is a.s.\ simple.
\end{corollary}

\begin{proof}
The tail estimate of the lemma implies that the probability that
there are two points or more in a fixed interval of length $\eps$ is at
most $c\eps^2$. Breaking the interval $[0,1]$ into pieces of
length $\eps$, and using translation invariance, we see that the
chance that there is a double point in $[0,1]$ is at most
$2c\eps$. Letting $\eps \to 0$ shows that a.s.~there are no double
points in $[0,1]$. The claim now follows from translation
invariance.
\end{proof}

Let $\mm$ denote the space of probability distributions on
$\ZZ$ with expectation. For  $\mu_1,\mu_2\in \mm$ let
$d(\mu_1,\mu_2)$ be the first Wasserstein distance,
i.e.~the infimum of $\ev |X_1-X_2|$ over all realizations
where the joint distribution of $(X_1,X_2)$ has marginals
$\mu_1$ and $\mu_2$. The topology induced by $d$ is
stronger than weak convergence of probability measures. Let
$\mathcal L(N_f)$ denote the distribution of $N_f$. The
following proposition is a stronger version of Proposition
\ref{p_continuity} in the introduction.

\begin{proposition}\label{l_12}
The map  $f\mapsto \mathcal L(N_f)$ is  Lipschitz-1 continuous in
$f$: for $f,g\in L^1_*$ we have $d(\mathcal L(N_f),\mathcal
L(N_g))\le \|f-g\|_1$.
\end{proposition}

\begin{proof}

Proposition \ref{l_sdeprop} gives that $N_g-N_f$ has the
same distribution as  $N_{g-f}$ which implies
\begin{equation*}d(\mathcal L(N_f),\mathcal
L(N_g))\le \ev |N_g-N_f|=\ev |N_{g-f}|\le \|g-f\|_1.\qedhere\end{equation*}
\end{proof}
%
%
%

\subsection{Large gap probabilities} \label{s_gap}

\begin{theorem}\label{t_ggap}
Let $f:\R^+\to \R^+$  satisfy $f(t)\le c/(1+t^2)$ for all $t$ and
$\int_0^\infty |df|<\infty$. Let $k\ge 0$. As $\lambda \to
\infty$, for the point process given by the Brownian carousel with
parameter $f$ we have
\begin{eqnarray} \label{e_ggap}
\pr(\mbox{\# of points in }[0,\lambda]\le k) =
\exp\big(-\lambda^2( \|f\|_2^2/8+o(1))\big).
\end{eqnarray}
\end{theorem}

\begin{lemma}\label{gtail} Let $Y$ be an adapted stochastic process with $|Y_t|<m$, and let
$X$ satisfy the SDE $dX=YdB$ where $B_t$ is a Brownian motion.
Then for each $a,t>0$ we have
$$
\pr(X(t)-X(0)\ge a) \le \exp\left(- a^2/(2 tm^2) \right).
$$
\end{lemma}
\begin{proof} We may assume $X(0)=0$.
Then $X_t=B_\tau$ where $\tau$ is the random time change
$\tau=\int_0^t Y^2(s)ds$. Since $\tau<m^2 t$ the inequality
now follows from
\begin{equation*}
\pr(B_r>a)\le \exp\left(- a^2/(2 r) \right).\qedhere
\end{equation*}
%
%
%
%
\end{proof}

\begin{proof}[Proof of Theorem \ref{t_ggap}] The event in
\eqref{e_ggap} is given in terms of the stochastic sine equation
as $\lim_{t\to\infty} \alpha_{\gl}(t)\le 2k\pi$. We will give
upper and lower bounds on its probability.

\bigskip

\noindent{\bf Upper bound.} By Proposition \ref{l_sdeprop}
\ref{szelep} it is enough to give an upper bound on the
probability that $\alpha$ stays less than $x=2(k+1)\pi$.
 For $0<s<t$ we have
\begin{eqnarray*}
\pr(\alpha(t)<x\,|\,\mathcal F_{s})  &=& \pr\left(-\int_s^t
2\sin(\alpha/2) dB \,> \,\lambda \int_s^t f dt - x
+\alpha(s)\,\Big|\,\mathcal F_{s}\right).
\end{eqnarray*}
We may drop the $\alpha(s)$ from the right hand side and use Lemma
\ref{gtail} with $Y=-2\sin(\alpha/2)$, $m=2$, $a=\lambda (\int_s^t f dt- x /\lambda)$ to get the upper bound
$$
\pr(\alpha(t)<x\,|\,\mathcal F_{s})\le \exp(-\lambda^2r(s,t)),\qquad
 r(s,t)=\frac{(\int_s^t f dt- x /\lambda)^2}{8(t-s)}.
$$
Then, by just requiring $\alpha(t)<x$ for times $\eps,
2\eps,\ldots \in[0,K]$ the probability that $\alpha$ stays less
than $x=2(k+1)\pi$ is bounded above by
$$
 \ev
\prod_{k=0}^{K/\eps} \pr(\alpha((k+1)\eps)<x\,\big|\,\mathcal
F_{k\eps})
 \le \exp\Big\{-\gl^2 \sum_{k=0}^{K/\eps} r(\eps k, \eps
k+\eps) \Big\}.
$$
A choice of $\eps$ so that  $ x /\lambda=o(\eps)$ as
$\lambda\to\infty$ yields
 the asymptotic Riemann sum 
$$
\sum_{k=0}^{K/\eps} r(\eps k, \eps k+\eps) =\frac{1}{8}
\int_0^K f^2(t)dt + o(1).
$$
Letting $K\to \infty$ provides the desired upper bound.

\smallskip

\noindent{\bf Lower bound.}
Consider the solution $\tilde \alpha(t)$ of \eqref{e sse3a} with
the same driving Brownian motion, but with initial condition
$\tilde \alpha(0)=\pi$. Then $\tilde \alpha\ge \alpha$. For
$\eps<\pi/4$, let $A_{s}$ be the event  that $\tilde\alpha(t)\in
(0,\pi + \eps]$ for  $t\in[0,s]$. Then
 $$
\pr(\alpha(\infty)<2\pi) \ge \pr(A_s)\sup_{y\in (0,\pi+\eps)}\pr
(\tilde \alpha(\infty)<2 \pi \,|\,\tilde \alpha(s)=y) .
 $$
 The sup
is bounded below via Markov's inequality by
$$
1- \frac{ \pi+\eps+ \gl \int_s^\infty
 f(t)dt}{2 \pi} \ge 1/4,
$$
where  the last inequality holds if  $s$ is set to be a large
constant multiple of $\lambda$. The event $A_s$ is equivalent to
$R=\log \tan(\tilde \alpha/4)$ staying in the interval
$I=(-\infty,\log\tan((\pi + \eps)/4)]$ where the evolution of $R$
is given by It\^o's formula as
 \begin{equation}\nonumber
 dR = \frac{\lambda}{2}\, f \cosh R\, dt + \frac{1}{2}\tanh R\, dt + dB, \quad R(0)=0.
 \end{equation}
Let $I^*=[-\eps,\log\tan((\pi + \eps)/4)]$, and consider a process
$R^*$ so that (i) the noise terms of $R$ and $R^*$ are the same
and (ii)
 the drift term
of $R^*$ at every time is greater than the spatial maximum over
$I^*$ of the drift term of $R$.  Let  $A^*_{s}$ denote the event
that for $t\in [0,s]$  we have  $R^*_t \in I^*$. On this event
$R^*\ge R$, and therefore $A_{s}$ also holds. With an appropriate
choice of $c$ we may set
$$q(t)=(1/2+c\eps) \lambda f(t)+c\eps, \qquad \qquad
dR^*=q(t) dt +dB.
$$
Let $A^*_{s}$ also denote the corresponding set of paths.
Girsanov's theorem gives
\begin{equation}\label{girs}
\pr(R^*\in A^*_{s})=\ev \left[\one( B\in
A^*_{s})\exp\left(-\frac{1}{2}\int_0^s q(t)^2 dt +\int_0^s
q(t)dB(t)\right)\right].
\end{equation}
Integration by parts transforms the second integral:
$$
q(s) B(s)- \int_0^s B(t) dq(t)\ge -c\eps \lambda
\big(f(s)+\int_0^\infty |df|\big)-c\eps\ge -c'\eps(1+\lambda)
$$
on the event $B\in A^*_s$. Here we also used that $f$ is bounded.
The probability of this event, i.e.\ that Brownian motion stays in
an interval of width $c\eps$, is  at least $\exp(-c' s/\eps^2)$.
In summary, \eqref{girs} is bounded below by
$$
\exp\left(-c\eps (1+\lambda) -c\,
s/\eps^2-(1/8+c\eps)\lambda^2\|f\|_2^2\right).
$$
The choice $s=c\lambda$, $\eps=\lambda^{-\nicefrac{1}{3}}$ gives
the desired lower bound.
\end{proof}

\subsection{A phase transition at $\beta=2$}

The goal of this section is to prove Theorem \ref{t_pt} that at
$\beta=2$ there is a phase transition in the behavior of the
stochastic sine equation.

As $\alpha$ converges to an integer multiple of $2\pi$ and
it can never go below an integer multiple of $2\pi$ that it
has passed (Proposition \ref{l_sdeprop} \ref{szelep}),
eventually  it must converge either from above or from
below. Theorem \ref{t_pt} says that $\alpha$ converges from
above with probability $1$ if and only if $\beta \le 2$.
Otherwise, it converges from below with positive
probability.
\begin{proof}
\emph{Case $\beta\le 2$.} It suffices to prove that if
$\alpha_\gl(t_0)\in (2\pi k-y,2\pi k)$ with $0<y<\pi$ then $\alpha_\gl(t)$ leaves
this interval a.s.~in finite time. As $\alpha_\gl(t+t_0)$ also
evolves according to the stochastic sine equation with $\gl'=\gl
e^{-\beta t_0/4}$, we may set $t_0=0$ and we are also free
to set $k=1$.
Let $\mathcal B$ denote the event that the process $\alpha_\gl$
started at $\alpha_\gl(0)=x$ in $(2\pi-y,2\pi)$ will stay in this
interval forever. It suffices to show that $\mathcal B$ has zero
probability.

Consider $R=\log \tan(\alpha_\gl/4)$ and set $y$ so
that $\log \tan((2\pi-y)/4)=1$. While
$\alpha_\gl\in(0,2\pi)$,  It\^o's formula gives the evolution of the process $R$:
 \begin{equation}\label{dR} R(0)=r_0>1, \qquad dR = q(R,t)dt+ dB, \qquad q(r,t)=\frac{ \lambda}{2}
\,\cosh r \,e^{-\beta t/4} + \frac{1}{2}\tanh r
 \end{equation}
Then $\mathcal{B}$ is the event that $R(t)\in(1,\infty)$ for all
$t$. On $\mathcal{B}$ we have
\begin{equation}\nonumber
q(R,t)>\frac12 \tanh R\ge \frac12-e^{-2R}\ge 1/4.
\end{equation}
which gives  $R(t)\ge t/4-B(t)+r_0$ from \eqref{dR}. Set
\begin{equation}\label{e_QQQ}
Q(t)=\int_0^t \frac{\tanh(R)-1}{2}ds,
\end{equation}
by the previous inequality, on the event $\mathcal B$ we have
$$
Q(t)\ge -\int_0^t e^{-t/2+2B(s)-2r_0}ds>-\int_0^\infty
e^{-t/2+2B(s)-2r_0}ds=-M.
$$
where the  integral $M$ is  a.s.~finite. Let
\begin{equation}\label{e_LLL}
L(t)=R(t)-t/2-B(t)-Q(t).
\end{equation}
Then on $\mathcal B$ we have
\begin{eqnarray}\label{e_LLL1}
L'(t)&=&\frac{ \lambda}{2}
\cosh(L(t)+t/2+B(t)+Q(t))e^{-\beta t/4}\\
&\ge& \frac{ \lambda}{4}
\exp\left[L(t)+B(t)+t(1/2-\beta/4)-M\right].\label{23}
\end{eqnarray}
The equation follows from It\^o's formula and the inequality uses
$\cosh r\ge e^r/2$. Multiplying \eqref{23}  by $e^{-L}$ and
integrating we get that on the event $\mathcal B$
\begin{equation}\nonumber
e^{-L(0)}-e^{-L(t)}\ge C \int_0^{t}
\exp\left[B(s)+s(1/2-\beta/4)\right] ds.
\end{equation}
with a random  $0<C<\infty$. As the exponent is
a Brownian motion with nonnegative drift,  the limit of the
integral on the right is a.s.~infinite, thus the probability of
$\mathcal B$ is 0.

\smallskip

\noindent \emph{Case $\beta< 2$.} It suffices to prove that for a large $t_0$ if  $\alpha_\gl(t_0)\in (2\pi -\eps,2\pi)$ then $\alpha_\gl(t)$ stays in the slightly larger interval $(2\pi -\delta,2\pi)$ with positive probability. Choosing the values of $\eps$ and $\delta$ appropriately it suffices to show that if $R(0)>2$ and $\gl$ is small enough then the event $\mathcal B$ that  $R\in(1,\infty)$ for $t\ge 1$ has positive probability.

Recall the definition of $Q$ and $L$ from (\ref{e_QQQ}) and (\ref{e_LLL}). On the event $\mathcal B$ we have
$$
1/4\le \tanh(R) \le 1/2, \quad\textup{and}\quad
-t/4 \le Q(t)\le 0.
$$
Using this with (\ref{e_LLL1}) and the fact that for $r$
nonnegative $ \cosh r\le e^{r}$ we get
\begin{eqnarray}\label{fst}
L'(t)&\le&  \frac{
\lambda}{2}\exp\left[L(t)+B(t)+t(1/2-\beta/4)\right].
\end{eqnarray}
From (\ref{fst}) we get
\begin{equation*}
e^{-L(0)}-e^{-L(t)}\le \frac{\gl}{2} \int_0^{t} \exp\left[B(s)+s(1/2-\beta/4)\right] ds.
\end{equation*}
Let $M^*$ denote the above integral for  $t=\infty$. Then
$M^*$ is almost surely finite. Moreover, $L(t)$ and thus
$R(t)$ remain finite if
\begin{equation}\label{integral}
M^* < 2 e^{-L(0)}/\gl.
\end{equation}
From (\ref{e_LLL1}) we get $L'(t)>0$ and $L(t)> L(0)=2$ which gives
$$R(t)>2+B(t)+t/2+Q(t)\ge2+B(t)+t/4.$$
So $R(t)$ stays above 1 if
\begin{equation}\label{BM event}
B(t)\ge -t/4-1 \qquad \mbox{ for all }t.
\end{equation}
This has positive probability, so the conditional
distribution of and $M^*$ given \eqref{BM event} is
supported on finite numbers. This means that the
intersection of the events \eqref{BM event} and
\eqref{integral} holds with positive probability for a
sufficiently small choice of $\lambda$, and it implies
$\mathcal B$.
\end{proof}

\section{Breakdown of the proof of Theorem 1}\label{s_proof}

The goal of this section is to divide the proof of the main
theorem into independent pieces, which in turn will be proved in
the later sections. The proof presented here also serves as an
outline of the later sections.

\begin{proof}[Proof of Theorem \ref{t_main}]
Fix $\beta>0$, and consider the $n\times n$ random tridiagonal
matrix
\begin{equation}
M(n)=\frac{1}{\sqrt{\beta}}\left(
  \begin{array}{cccc}
    \mathcal N_0 & \chi_{(n-1)\beta} &   &   \\
    \chi_{(n-1)\beta} & \mathcal N_{1} & \chi_{(n-2)\beta} &  \\
     & \chi_{(n-2)\beta} & \mathcal N_{2} & \ddots \\
     &  & \ddots & \ddots
  \end{array}
\right)\label{e_originalmatrix}.
\end{equation}
where the $\chi_{j\,\beta}$ and $\mathcal N_j$ entries are
independent, $\mathcal N_j$ has normal distribution with
mean 0 and variance 2, and $\chi_{j\beta}$ has chi
distribution with $j\beta $ degrees of freedom. (For
integer values of its parameter, $\chi_d$ is the length of
a $d$-dimensional vector with independent standard normal
entries.) We let $\Lambda_n$ be the multi-set of
eigenvalues of this matrix, which by \cite{DE} has the
desired distribution \eqref{betadens}.

First, we may assume that $\mu_n\ge0$; indeed, for a tridiagonal
matrix, changing the sign of all diagonal elements changes the
spectrum to its negative. In our case the diagonal elements have
symmetric distributions, and by Remark \ref{r_proc} the limiting
$\Sineb$ process is also symmetric.

We set
\begin{equation}
n_0=n_0(n)=n- \nnn-\frac12.\nonumber
\end{equation}
The assumption $n^{\nicefrac{1}{6}}(2\sqrt{n}-|\mu_n|)\to
\infty$ implies
$$
n_0^{-1} \nn^{\nicefrac{2}{3}}\to 0, \qquad
\frac{4n-\mu_n^2}{4 n_0} \to 1.
$$
So it suffices to show that
\begin{equation}
 \textup{if $n_0^{-1}  \nn^{\nicefrac{2}{3}}\to 0$ as
$n\to \infty,$\quad then\quad }2 n_0^{\nicefrac{1}{2}}(\Lambda_n-\nn) \Rightarrow
\Sineb,\label{e_limit}
\end{equation}
an equivalent version of the claim which makes computations nicer. Recall that the
counting function $N(\gl)$ of a set of points in $\RR$ is the
number of points in $(0,\lambda]$ for $\lambda\ge 0$ or negative
the number of points in $(\lambda,0]$ for $\lambda <0$.

Denote the counting function of the random multiset $ 2
n_0^{\nicefrac{1}{2}}(\Lambda_n-\nn)$ by
$N_n(\gl)$, and that of $\Sineb$ by $N(\gl)$. Claim
(\ref{e_limit}) follows if for every  $d\ge 1$ and $(\gl_1,
\gl_2, \ldots, \gl_d)\in \RR^d$ we have
\begin{equation}\nonumber
\big(N_n(\gl_1),N_n(\gl_2),\ldots, N_n(\gl_d)\big)\cd
\big(N(\gl_1),N(\gl_2),\ldots, N(\gl_d)\big).
\end{equation}
The proof of this consists of several steps, these are  verified
in detail in the subsequent sections with the help of the
Appendix.

Consider the one-parameter family of SDEs defining the $\Sineb$
process:
\begin{equation}\label{e_ssesde}
d\tilde \alpha_\gl=   \gl\frac{\beta}{4} e^{-\beta t/4} dt
+ \Re ((e^{-i\tilde \alpha_\gl}-1)dZ),
\end{equation}
where $Z$ is complex Brownian motion on $[0,\infty)$ with
standard real and imaginary parts. The time-change $t\to
-\frac2\beta \log(1-t)$ transforms \eqref{e_ssesde} to
\begin{equation}
 2 \sqrt{\beta(1-t)} \,d\alpha_\lambda = \lambda
\beta^{\nicefrac{1}{2}} dt+ 2\sqrt{2} \Re(
(e^{-i\alpha_\lambda}-1)dW)\label{e_SDE_alpha},
\end{equation}
where $W_t$ is complex Brownian motion for  $t\in[0,1)$
with standard real and imaginary parts. Proposition
\ref{l_sdeprop} of Section \ref{s_carousel} shows that the
counting function $N(\lambda)$ of the process $\Sineb$ can
be represented as the right-continuous version of
$(2\pi)^{-1}\lim_{t\to\infty} \tilde \alpha_{\gl}(t)$, a
limit which exists for every $\lambda\in \RR$ a.s. This
gives
 \begin{step}\label{st_sde} For every $\lambda \in \RR$, a.s.
we have  $2\pi N(\gl)=\lim_{\eps \to 0^{+}} \alpha_{\gl}
(1-\eps)$.
\end{step}
The eigenvalue equation for a tridiagonal matrix gives a
three-term recursion for the eigenvector entries. This can
be solved for any value of $\lambda$, but the boundary
condition given by the last equation is only satisfied for
eigenvalues.

The ratios of consecutive eigenvector entries $r_{\ell,\lambda}$
evolve via transformations of the form $r\mapsto a-b/r$, $b>0$.
These transformations are isometries of the Poincare\'e half plane
model of the hyperbolic plane. The hyperbolic framework is
introduced in Section \ref{s_hyp} for the study of these
recursions. In particular,  $r_{\ell,\lambda}$ moves on the
boundary of the hyperbolic plane which can be represented as a
circle, eigenvalues can be counted by tracking the winding number of $r_{\ell,\gl}$ as a function of $\gl$.
The rough phase function $\hat \varphi_{\ell, \lambda}$
(introduced in Section \ref{s_model}) transforms $r_{\ell,\gl}$
to an angle through
$2\arctan(r_{\ell,\lambda})$. Taking always the appropriate inverse of $\tan$
we get a continuous function of $\lambda$ taking values in
$\RR$, the universal cover of the circle.

Our goal is to take limits of the evolution of $\hat
\varphi_{\ell,\lambda}$. Since it has fast oscillations  first it needs to be regularized.
In order to remove the oscillations, we  follow a shifted
version of the hyperbolic angle of $r_{\ell,\lambda}$ around the
fixed point of a simplified version of the transformation
$r\mapsto a-b/r$.
As we will see later, the important part of the
evolution takes place  in the interval $0\le \ell \le \lfloor n_0
\rfloor$, which is exactly when this transformation is a hyperbolic rotation.

The precise regularization is done in Sections \ref{s_svp};
there we introduce the (regularized) phase function
$\varphi_{\ell,\gl}$ and target phase function
$\tph_{\ell,\gl}$ with parameters $0\le \ell \le \lfloor
n_0 \rfloor$ and $\gl\in \RR$. These correspond to solving
the eigenvalue equations   starting from the two ends, $1$
and $n$. As $n\to\infty$, these two parts will require
completely different treatment, so it is natural to break
the evolution into two parts this way. Proposition
\ref{p_fi}  shows how we can count the eigenvalues using
the zeroes of these phase functions mod $2\pi$; this is a
discrete analogue of the Sturm-Liouville oscillation
theory.

Let $\#A$ denote the number of
elements of $A$.

\begin{step}\label{st_phi}
For $\ell=1,2,\ldots ,\lfloor n_0\rfloor$, the function
$\varphi_{\ell,\gl}$ is monotone increasing, and is independent of
$\tph_{\ell,\gl}$. For any $\lambda<\lambda'$ and $1\le \ell\le \lfloor n_0\rfloor $ almost surely we
have
\begin{equation}\label{e_Ndif}
N_{n}(\lambda')-N_{n}(\lambda)=\# \left(
(\varphi_{\ell,\lambda}-\tph_{\ell,\lambda},\varphi_{\ell,\lambda'}-\tph_{\ell,\lambda'}]
\cap 2\pi \Z    \right).
\end{equation}
\end{step}

Since $\hat \varphi$ counts all eigenvalues below a certain
level, its regularized version $\varphi$ will encode all
the fluctuations in the number of such eigenvalues. So the
continuum limit of $\varphi$ is expected to have large
oscillations as its time-parameter converges to $\infty$.
In order to deal with this problem, we introduce the  {\bf
relative phase function}
$\alpha_{\ell,\gl}=\varphi_{\ell,\gl}-\varphi_{\ell,0}$.
This is related to the number of eigenvalues in an
interval, so it is expected that its scaling limit will
have nice behavior at $+\infty$.  Note that
$\alpha_{\ell,\gl}$ has the same sign as $\gl$  by Step
\ref{st_phi}. Let
\begin{equation}\nonumber
m_1=\lfloor n_0 (1-\eps)  \rfloor, 
\quad m_2=\lfloor n- \nnn- \kappa\, (\nn^{\nicefrac{2}{3}}
\vee 1)  \rfloor,
\end{equation}
where the constants $\eps, \kappa>0$ will be specified later in a
way that the chain of inequalities $0\le m_1\le m_2$ holds.

Next we will
describe the limiting behavior of $\varphi_{\ell,\gl}$ and
$\alpha_{\ell,\gl}$ when $\ell$ is in the intervals $[0,m_1]$ and
$[m_1,m_2]$, respectively. This is the content of the next three
steps.
Section \ref{s_sstep} studies the behavior of the relative phase function on  $[0,m_1]$.
In Corollary \ref{cor_sdelim}  we will
prove that $\alpha_{\ell,\gl}$ converges
to the SDE (\ref{e_SDE_alpha})  in this stretch.
\begin{step}\label{st_elso}
For every $0<\eps \le1$
\begin{equation}
\alpha_{m_1,\gl}\cd \alpha_{\gl} (1-\eps), \qquad \textup{as $n\to
\infty$}
\end{equation}\nonumber
in the sense of finite dimensional distributions for $\gl$.
\end{step}
Proposition \ref{l_middlestretch} of Section \ref{s_middle}
shows that $\alpha_{\ell,\gl}$ does not change much in the
second stretch, if it is already close to 0 mod $2\pi$ at the beginning of the stretch.
\begin{step}\label{st_masodik}
 There exists constants $c_0, c_1$ depending
only on $\bar \gl$ and $\beta$ such that if $\kappa=\kappa_n>c_0$,
$\gl\le
|\bar \gl|$ 
then
\begin{equation}\label{e_masodik} \ev \left[\left|(
\alpha_{m_1,\gl}-\alpha_{m_2,\gl})\right| \wedge 1\right] \le c_1
(\ev \,\dist(\alpha_{m_1,\gl},2\pi
\Z)+\eps^{\nicefrac{1}{2}}+n_0^{\nicefrac{-1}{2}}
( \nn^{\nicefrac{1}{3}}\vee 1)+\kappa^{-1}),
\end{equation}
\end{step}
In Proposition \ref{p_unif} of Section \ref{s_unif} we show that
$\varphi_{m_2, 0}$ becomes uniform mod $2\pi$.
\begin{step}\label{st_harmadik}
If $\kappa\to \infty$ and $n_0^{-1} \kappa
(\nn^{\nicefrac{2}{3}}\vee 1)\to 0$ then
\begin{equation*}
\frtp{\varphi_{m_2,0}}\cd \textup{Uniform}[0,2\pi],
\end{equation*}
where $\frtp{x}=\min_{k\in \ZZ, k\le x} (x-2\pi k)$.
\end{step}
Finally, in Lemma \ref{l_short} of Section \ref{s_end} we show
that nothing interesting happens after $m_2$.
\begin{step}\label{st_end}
For every fixed $\kappa>0$ and  $\gl\in \RR$
\begin{equation}\nonumber
\left| \tph_{m_2,\gl}-\tph_{m_2,0} \right|\cp 0, \qquad \textup{as
$n\to \infty$}.
\end{equation}
\end{step}
In a metric space, if $\lim_{n\to\infty}x_{n,k}=x_k$ for
every $k$ and also $\lim_{k\to\infty} x_k=x$ then we can
find a subsequence $n(k)\to\infty$ for which
$\lim_{n\to\infty} x_{n,n(k)}=x$. This simple fact,
together with the previous steps,  allows us to choose
sequences $\eps=\eps_n\to 0$, $\kappa=\kappa_n\to \infty$
in a way that the following limits hold simultaneously:
\begin{eqnarray}
(\alpha_{m_1,\gl_i},i=1,\ldots,d)&\cd& 2\pi
(N(\lambda_i),i=1,\ldots, d), \label{ittajel}
\\
\frtp{\varphi_{m_2,0}}&\cd& \textup{Uniform}[0,2\pi] \label{e_xyz}
\\ \left| \tph_{m_2,\gl_i}-\tph_{m_2,0} \right|&\cp&  0, \qquad i=1,\ldots, d
\label{nostep6}
\end{eqnarray}
Since $\dist(\cdot,2\pi \Z)$ is a bounded continuous
function, (\ref{ittajel}) implies that the right hand side
of (\ref{e_masodik}) vanishes in the limit and so
\begin{eqnarray}
\alpha_{m_1,\gl_i}-\alpha_{m_2,\gl_i} &\cp& 0 ,\qquad i=1,\ldots,
d\label{ittajel2}.
\end{eqnarray}
By (\ref{ittajel}) and (\ref{ittajel2}) the completion of
the proof only requires the following last step. Let
$\rtp{x}$ denote the element of $2\pi \Z$ in
$[x-\pi,x+\pi)$.
\begin{step} For $i=1,\ldots,d$ and $\lambda=\lambda_i$ we have
$ \lim_{n\to\infty}\pr\left(2\pi N_{n}(\gl)= \rtp{\alpha_{m_2,\gl}}
\right)= 1.$\label{e_vege}
\end{step}
We conclude by the proof of Step \ref{e_vege}. We suggest
skipping it at the first reading, as it is the most
technical part of this outline. We include it here because
it uses too much of the notation and assumptions of the
preceding discussion.

We will assume $\gl>0$, the other case follows similarly.
Then $0\le \rtp{\alpha_{m_2,\lambda}}\in 2\pi \ZZ$, for any
$x\in \RR$ we have
\begin{equation}\nonumber
\rtp{\alpha_{m_2,\lambda}}=2\pi\,\#\big(\left(x,x+\rtp{\alpha_{m_2,\lambda}}\right]\cap
2\pi \Z\big)
\end{equation}
Using this with
$x=\varphi_{m_2,\gl}-\tph_{m_2,\gl}-\rtp{\alpha_{m_2,\lambda}}$ we
get
\begin{equation}
\rtp{\alpha_{m_2,\lambda}}
=2\pi\,\#\Big(\left(\varphi_{m_2,0}-\tph_{m_2,\gl}+\alpha_{m_2,\gl}-\rtp{\alpha_{m_2,\gl}},
\varphi_{m_2,\gl}-\tph_{m_2,\gl} \right]\cap 2\pi
\Z\Big).\label{e_asd}
\end{equation}
The symmetric difference between  the intervals in \eqref{e_asd}
and \eqref{e_Ndif} is an interval $J$ with endpoints
$\varphi_{m_2,0}-\tph_{m_2,0}$ and
$\varphi_{m_2,0}-\tph_{m_2,\gl}+\alpha_{m_2,\gl}-\rtp{\alpha_{m_2,\gl}}$.
So it suffices to show
 \be\label{e_vege1_}
\lim_{n\to\infty}\pr\left((J \cap 2\pi \ZZ)=\emptyset
\right)= 1.\ee We will show that the length of $J$
converges to 0 while one of its endpoints becomes uniformly
distributed mod $2\pi$. First,
$$
|J|\le \left|\alpha_{m_2,\gl}-
\rtp{\alpha_{m_2,\gl}}\right| +\left| \tph_{m_2,\gl}
-\tph_{m_2,0} \right|\cp 0,
$$
where the convergence of the first term follows from
(\ref{ittajel}, \ref{ittajel2}) as $\alpha_{m_2,\lambda}$
converges to an element of $2\pi \ZZ$; the convergence of the
second term is  \eqref{nostep6}. Also, since $\varphi_{m_2,0}$ and
$\tph_{m_2,0}$ are independent, from (\ref{e_xyz}) we have
$\frtp{\varphi_{m_2,0}-\tph_{m_2,0}}\cd
\operatorname{Uniform}[0,2\pi]$. Equation (\ref{e_vege1_}), Step
\ref{e_vege} and the theorem follows.
\end{proof}

\begin{proof}[Proof of Corollary \ref{maincor}]
Note that weak convergence of point processes is metrizable. Let
$a_i \to \infty$. For every $i$, we can find $n_i>i$ so that the
point process
$$\Lambda^*_{i}=2\sqrt{a_i}\,n_i^{\nicefrac{1}{6}}(\Lambda_{n_i}-2\sqrt{n_i}+a_i
{n_i}^{\nicefrac{-1}{6}})$$ is $1/i$-close to $2
\sqrt{a_i}(\Airyb+a_i)$ by Theorem \ref{t_rrv}. By Theorem
\ref{t_main}, $\Lambda^*_{i}$ converges to $\Sineb$.
\end{proof}

\section{The hyperbolic description of the phase evolution}

\subsection{The hyperbolic point of view}\label{s_hyp}

The eigenvector equation for a tridiagonal matrix gives a
three-term recursion in which each step is of the form
$u_{\ell+1}=b u_{\ell}-a u_{\ell-1}$, in our case with $a > 0$.
Let $\PSL$ denote the group of linear fractional transformations
preserving the upper half plane $\HH$ and its orientation. Then
$r_{\ell}=u_{\ell+1}/u_{\ell}$ evolves by elements of $\PSL$ of
the form $r\mapsto b-a/r$.

We will think of $\HH$ as the Poincar\'e half-plane model for the
hyperbolic plane; it is equivalent to the Poincar\'e disk model
$\UU$ via the bijection
$$\varupsl{U}:\bar \HH \to \bar
\UU, \qquad z\mapsto \frac{i-z}{i+z},
$$
which is also a bijection of the boundaries. Thus $\PSL$ acts
naturally on $\bar \UU$, the closed unit disk. As $r$ moves on the boundary $\partial
\HH\equiv \RR\cup \{\infty\}$, its image under $\UU$ will move
along $\partial \UU$.

 In order to follow the number of times this
image circles $\UU$, we would like to extend the action of $\PSL$
from $\partial \UU$ to its universal cover, $\RR'\equiv \RR$,
where we use prime to distinguish this from $\partial \HH$.
This action is uniquely determined up to shifts by $2\pi$, but
here we have a choice. For each choice, we get an element of a
larger group $\UPSL$ defined via its action on $\RR'$. $\UPSL$
still acts on $\bar \HH$ and $\bar \UU$ just like $\PSL$, and for
$\varupsl{T}\in \UPSL$ the three actions are denoted by
$$
 \bar \HH \to \bar \HH: z\mapsto z.\varupsl{T},
 \qquad \bar \UU \to \bar \UU:z\mapsto z \lcirc \varupsl{T},\qquad
 \R'\to \R': z\mapsto z \lstar \varupsl{T}.
$$
We note in passing that the topological group $\UPSL$ is
the universal cover of the hyperbolic motion group $\PSL$,
and $\PSL$ is a quotient of $\UPSL$ by the infinite cyclic
normal subgroup generated by the $2\pi$-shift on $\R'$. For
every $\varupsl{T}\in\UPSL$ the function $x\mapsto x\lstar \varupsl{T}$ is
strictly increasing, analytic and quasiperiodic, i.e.
$(x+2\pi)\lstar \varupsl{T}=x\lstar \varupsl{T}+2\pi$.

Given an element $\varupsl{T}\in \UPSL$, $x,y\in \RR'$, we define the
angular shift
$$
\ash_{\R'}(\varupsl{T},x,y)=(y\lstar \varupsl{T}-x \lstar \varupsl{T})-(y-x)
$$
i.e.\ the amount the signed distance of $x,y$ changed over the
transformation $\varupsl{T}$. This only
depends on the image of $\varupsl{T}$ in $\PSL$ and the images
$v=e^{ix},w=e^{iy}\in
\partial \UU$ of $x,y$ under the covering map.
This allows us to define $\ash(\varupsl{T},v,w)$; more concretely,
$$
 \ash(\varupsl{T},v,w)=
 \ash_{\R'}(\varupsl{T},x,y) =
 \Arg_{[0,2\pi)}(w\lcirc \varupsl{T}/v\lcirc \varupsl{T})
  -\Arg_{[0,2\pi)}(w/v),
 $$
where the last equality has self-evident notation and is
straightforward to check. Note  also that the above formula
defines $\ash(\varupsl{T},v,w)$ for $\varupsl{T}\in  \PSL$, $v,w\in \partial \UU$ as
well. For explicit computations, we will rely on the following
fact, whose proof is given in Appendix \ref{a_ash}.

\begin{fact}[Angular Shift Identity]\label{f_ashid}
Let $\varupsl{T}\in \PSL$ be a M\"obius transformation  and $v,w\in
\partial \UU$; let $\sigma=0\lcirc \varupsl{T}^{-1}$. Then
\begin{equation}\label{f_ashidentity}
\ash(\varupsl{T},v,w)=2\Arg\left(\frac{(w-\sigma)v}{w(v-\sigma)}\right)=2
\Arg\left(\frac{1-\sigma \bar w}{1-\sigma \bar v}   \right).
\end{equation}
\end{fact}

Next, we specify generators for $\UPSL$. Let $\varupsl{Q}(\alpha)$ denote
the rotation by $\alpha$ in $\UU$ about $0$, more precisely, the
shift by $\alpha$ on $\RR'$:
 \begin{equation}\varphi\lstar \varupsl{Q}(\alpha)=\varphi+\alpha \label{d_Q1}
 \end{equation}
For $a,b\in\RR$ let $\varupsl{A}(a,b)$ be the affine map $z\mapsto a(z+b)$
in $\HH$. If $a>0$ then this is in $\PSL$, it fixes the $\infty$
in $\partial \HH$ and $-1$ in $\partial \UU$. We specify the
action of $\varupsl{A}$ on $\RR'$ by making it fix $\pi\in \RR'$. Then we
have
\begin{equation}\label{e_lift}
 \varphi\lstar \varupsl{A}(a,b)=\varphi+\ash(\varupsl{A}(a,b),-1,e^{i \varphi}).
\end{equation}
The following lemma estimates the angular shift. The proof
is given in Appendix \ref{a_ash}.
\begin{lemma}\label{l_ash}
Suppose  that for a $\varupsl{T}\in \UPSL$ we have $(i+z).\varupsl{T}=i$ with
$|z|\le \nicefrac13$. Then
 \begin{equation} \label{e_ash}
 \begin{array}{rcl}
 \ash(\varupsl{T},v,w)&=&\Re\left[(\bar w -\bar v)\left(-z-\frac{i(2+\bar v + \bar w)}{4}\,z^2\right)
 \right]+\ep_3\\[5pt]
  &=&-\Re\left[(\bar w -\bar v) z\right]+\ep_2\\
  &=&\ep_1,
  \end{array}
  \end{equation}
where for $d=1,2,3$ and an absolute constant $c$ we have
\begin{eqnarray}
|\ep_d| \le c |w-v| |z|^d\le 2c |z|^d,\label{e_asherr}
  \end{eqnarray}
If $v=-1$ then the previous bounds hold even in the case
$|z|>\nicefrac13$.
\end{lemma}

\subsection{Phase evolution equations}\label{s_model}

The eigenvalue equation of a tridiagonal matrix can be
solved recursively. The goal of this section is to analyze
this recursion in terms of phase functions.

We conjugate the matrix $M=M(n)$ in
\eqref{e_originalmatrix} by a diagonal matrix $D$ with
$$
D_{ii}=D(n)_{ii}=\prod_{\ell=1}^{i}\frac{\chi_{(n-\ell)\beta}}{\sqrt{\beta}\,
s_{\ell}},\qquad \mbox{where}\qquad
s_j=\sqrt{n-j-\nicefrac12}. $$ We get the tridiagonal
matrix $M^D=D^{-1}MD$ given by
\begin{equation}
\frac{1}{\sqrt{\beta}}\left(
  \begin{array}{cccc}
    \mathcal N_0& \frac{\chi_{(n-1)\beta}^2}{  s_{1}\sqrt{\beta}} &   &   \\
      s_{1}\sqrt{\beta} & \mathcal N_{1} & \frac{\chi_{(n-2)\beta}^2}{  s_{2}\sqrt{\beta}} &  \\
     &   s_{2}\sqrt{\beta} & \mathcal N_{2} & \ddots \\
     &  & \ddots & \ddots
  \end{array}
\right)=\left(
  \begin{array}{cccc}
      X_0 &   s_0+  Y_0 &   &   \\
      s_{1} &   X_{1} &   s_{1}+  Y_{1} &  \\
     &   s_{2} &   X_{2} & \ddots \\
     &  & \ddots & \ddots
\end{array}
\right).\label{e_matrix}
\end{equation}
Then $M^D$ and $M$ have the same eigenvalues, but $M^D$ has the
property that the eigenvalue equations are independent. (A similar
conjugation appears in \cite{ES}.) The moments of the independent
random variables
\begin{eqnarray}
  X_j=\frac{\mathcal N_{j}}{\sqrt{\beta}}, \qquad
Y_j=\frac{\chi_{(n-j-1)\beta}^2}{\beta s_{j+1}}-s_{j},\quad 0\le j\le n-1 \nonumber
\end{eqnarray}
are explicitly computable via the moment generating functions for
the $\Gamma$ distribution.

  Our proof is valid for any choice of independent real-valued random
variables $X_j,Y_j$ satisfying the following asymptotic
moment conditions. $X_j$ and $Y_j$ may also depend on $n$,
in which case the implicit error terms are assumed to be
uniform in $n$. 
\begin{equation}
\begin{array}{c|c|c|c}
  \mbox{moment} & 1^{\textup{st}} & 2^{\textup{nd}} & 4^{\textup{th}} \\ \hline
   & \O((n-j)^{\nicefrac{-3}{2}}) & 2/\beta+\O((n-j)^{-1}) & \O(1)
\end{array}\label{e_mom}
\end{equation}
Let $u_\ell=u_{\ell,\Lambda}$ ($1\le \ell \le n$) be a
non-trivial solution of the first $n-1$ components of the eigenvalue equation
with a given spectral parameter $\Lambda$, i.e.~
  \begin{equation} \label{e_vrec}
 \begin{array}{c}
    s_{\ell}u_{\ell}+  X_{\ell}\,
  u_{\ell+1}+(  Y_{\ell}+  s_{\ell})
 u_{\ell+2}= \Lambda
  u_{\ell+1},\qquad 0\le \ell \le n-2\\
  u_0=0, \qquad u_1=1.
  \end{array}
  \ee
Then
with $ r_{\ell}=r_{\ell,\Lambda}=u_{\ell+1}/u_{\ell}$ we have
  \be\label{e_rrec}
  r_{\ell+1}
  =\left(-\frac{1}{r_{\ell}}+\frac{\Lambda}{  s_{\ell}}
  -\frac{  X_{\ell}}{  s_{\ell}}\right)\left(1+\frac{  Y_{\ell}}{  s_{\ell}}\right)^{-1}\\[4pt]
  ,\qquad
   0\le \ell \le n-2
    \ee
  This also holds for
$r_\ell=0$ or  $r_\ell=\infty$; in fact, the initial value of
the recursion is $r_0=\infty$. If we set $Y_{n-1}=0$ and define  $r_n$ via the $\ell=n-1$ case of  \eqref{e_rrec}, then
$\Lambda$ is an eigenvalue if and only if $r_n=0$.

%

We will use the point of view and notation introduced in Section
\ref{s_hyp}. Namely, $r$ takes values in $\partial \HH=\RR\cup
\{\infty\}$, the boundary of the hyperbolic plane. Moreover, the
evolution of $r$ can be lifted to the universal cover of $\partial
\HH$. The extra information there allows us to count eigenvalues,
as the following proposition shows. The proposition also
summarizes the evolution of $r$ and its lifting $\hat \varphi \in
\RR'$. We note that
this is just a discrete analogue of
the Sturm-Liouville oscillation theory suitable for our
purposes; such analogues are available in the literature.
Although we state this proposition in our setting, a
trivial modification holds for the eigenvalues of general
tridiagonal matrices with positive off-diagonal terms.

%

%


\begin{proposition}[Wild phase function]\label{p_wild}
There exist functions $\hat \varphi, \tphh: \{0,1,\ldots,
n\}\times \R\to\R$ satisfying the following:
\begin{enumerate}
\item $r_{\ell,\Lambda}.\varupsl{U}=e^{i \hat \varphi_{\ell,\Lambda}}$,

\item $\hat \varphi_{0,\Lambda}=\pi$, $\tphh_{n,\Lambda}=0$.

\item For each $0<\ell\le n$, $\hat \varphi_{\ell,\Lambda}$
is analytic and strictly increasing in $\Lambda$. For $0\le \ell
<n$, $\tphh_{\ell,\Lambda}$ is analytic and strictly decreasing in
$\Lambda$.

\item  For any \, $0\le \ell \le n$, $\Lambda$ is an
eigenvalue of $M$ if and only if $\hat
\varphi_{\ell,\Lambda}-\tphh_{\ell,\Lambda}\in 2\pi \ZZ$.
\end{enumerate}
\end{proposition}

\begin{proof}

We consider the following elements of the universal cover $\UPSL$
of the hyperbolic motion group $\PSL$:
\begin{eqnarray}\label{d_QW}
 \varupsl{Q}(\pi),\qquad
 \varupsl{W}_j&=&
               \varupsl{A}((1+Y_j/s_j)^{-1},-X_j/s_j)  \qquad  0\le j\le n-1
\end{eqnarray}
where $Q$  corresponds to a rotation in the model $\UU$, and
$A$ corresponds to an affine map in the model $\HH$,
as defined in (\ref{d_Q1}-\ref{e_lift}).
With this
notation, the evolution \eqref{e_rrec} of $r$  becomes
\begin{equation}\label{e 1step}
\begin{array}{rcl}
\varupsl{R}_{\ell,\Lambda}&=&  \varupsl{Q}(\pi)\,
\varupsl{A}(1,\Lambda/  s_{\ell})\, \varupsl{W}_{ \ell}, \\
 r_{\ell+1}&=& r_{\ell}. \varupsl{R}_{\ell,\Lambda},\end{array}
\end{equation}
for $0\le \ell \le n-1$, and $\Lambda$ is an eigenvalue if and
only if $\infty.\varupsl{R}_{0,\Lambda}\cdots \varupsl{R}_{n-1,\Lambda}=0$.
Multiplying this by $(\varupsl{R}_{\ell,\Lambda}\cdots \varupsl{R}_{n-1,\Lambda})^{-1}$ for some $0\le \ell\le n$ and then moving
to the universal cover $\RR'$ of $\partial \HH$ gives the
equivalent characterization $\hat
\varphi_{\ell,\Lambda}=\tphh_{\ell,\Lambda}$ mod $2\pi$, where
\begin{equation}\label{e_1stepfi}
\hat \varphi_{\ell,\Lambda}=
 \pi\lstar \varupsl{R}_{0,\Lambda}\cdots \varupsl{R}_{\ell-1,\Lambda}
 ,\qquad
 \tphh_{\ell,\Lambda}=0\lstar
 \varupsl{R}_{n-1,\Lambda}^{-1}\cdots \varupsl{R}_{\ell,\Lambda}^{-1},
\end{equation}
which is exactly { (iv)}. Claims { (i)-(ii)} follow from
the definition.

As $\varphi_{0,\Lambda}=\pi$, one readily checks that
$\varphi_{1,\Lambda}$ is  strictly increasing. Since
$(\varphi,\Lambda)\mapsto \varphi\lstar \varupsl{R}_{\ell,\Lambda}$
are nondecreasing analytic  functions in both parameters
and so are their compositions, the statement of claim {
(iii)}  for $\hat \varphi_{\ell,\Lambda}$ now follows. The
same proof works for $\tphh$.
\end{proof}

Motivated by part (iv) of the proposition, we call $\tphh$
the {\bf target phase function}.

\subsection{Slowly varying phase evolution for a scaling window}

\label{s_svp}

For scaling, we set
$$
s(\tau)=s^{(n)}(\tau)=\sqrt{1-\tau-\nicefrac{1}{2n}}
$$
so that we have $  s_{\ell}=s(\ell/n)\sqrt{n}$. Making $s$
depend on $n$ via the $\nicefrac{1}{2n}$ term helps make
 the upcoming formulas exact rather than only asymptotic.

The phase function $\hat \varphi_\ell$ introduced in the
previous section exhibits fast oscillations in $\ell$. In
this section we will extract a slowly moving component of
the phase evolution whose limiting behavior can be
identified. The oscillations of $\hat \varphi_\ell$ are
caused by the macroscopic  term $\varupsl{Q}(\pi)\,
\varupsl{A}(1,\Lambda/  s_{\ell})$ of the evolution
operator  $\varupsl{R}_{\ell,\Lambda}$. The recursion
(\ref{e 1step}) has different behavior depending on whether
this macroscopic
 part is a rotation or not. As we will see later, the
continuum limit process comes from the stretch
$0\leq \ell <n_0$
where it is a
rotation (this is because the corresponding eigenvectors will be
localized there). The  eigenvalues
$\Lambda$ of interest will be near the scaling window $\mu_n$, so
we define the main part of the evolution operator as the macroscopic part of $\varupsl{R}_{\ell,\mu_n}$, that is
\begin{equation}\label{e_Gamma}
\varupsl{J}_{\ell}=\varupsl{Q}(\pi)\, \varupsl{A}(1,\mu_n/
s_{\ell})=\varupsl{Q}(\pi)\varupsl{A}\Big(1,\frac{\mu_n}{\sqrt{n}s(\ell/n)}\Big).
\end{equation}
This is a rotation if it  has a
fixed point $\rho_\ell$ in the open upper half plane $\HH$, the fixed point equation  $\rho_\ell.\varupsl{J}_{\ell}=\rho_\ell$ turns into
\begin{equation}
\rho_\ell^2-2\,\frac{\nn/\sqrt{4n}}{s(\ell/n) }
\rho_\ell+1=0\label{e_crho}.
\end{equation}
Note that $\nn/\sqrt{ 4n}$ is the relative location of
the scaling window in the Wigner semicircle supported on $[-1,1]$.
Since $s(\tau)$ is decreasing, we have that $\rho_\ell\in \HH\,$ for
$\tau<\idiota$, where $s(\idiota)=\nn/\sqrt{ 4n}$. This explains the
choice of the parameter $n_0$.

Thus $\rho_\ell=\rho(n_0/n,\ell/n)$, where $\rho(\tau_1,\tau_2)$
is the solution in the closed upper half plane of
\begin{equation}
\rho^2-2\,\frac{s(\tau_1)}{s(\tau_2) } \rho+1=0,\qquad  \mbox{
i.e. }\qquad
\rho(\tau_1,\tau_2)=\frac{s(\tau_1)}{s(\tau_2)}+i\sqrt{1-\frac{s(\tau_1)^2}{s(\tau_2)^2}}.
\label{e_crho2}
\end{equation}
More specifically,
\be\label{e_rhorho}
\rho_\ell=\sqrt{\frac{ \nnn}{ \nnn+n_0-\ell}}+i
\sqrt{\frac{n_0-\ell}{ \nnn+n_0-\ell}}.
\ee

Because of our choice of scaling window and the density in the
Wigner semicircle law it is natural to choose the scaling
(\ref{e_limit}) by setting
\begin{equation}
\Lambda=\nn+\frac{\gl}{2\sqrt{n_0}}\label{e_scalingnu}.
\end{equation}
We recycle the notation $u_{\ell,\gl}, r_{\ell,\gl},
\hat\varphi_{\ell,\gl}, \tphh_{\ell,\gl}$ for the quantities
$u_{\ell,\Lambda}, r_{\ell,\Lambda}, \hat\varphi_{\ell,\Lambda},
\tphh_{\ell,\Lambda}$. We separate $\varupsl{J}_\ell$  from the evolution operator $\varupsl{R}$ to get:
\begin{equation}\label{e_evoevo}
\varupsl{R}_{\ell,\lambda}= \varupsl{J}_{\ell}
\varupsl{L}_{ \ell,\gl} \varupsl{W}_{ \ell},\qquad
\varupsl{L}_{ \ell,\gl}\;=\;
\varupsl{A}\Big(1,\frac{\gl}{2s(\ell/n)\sqrt{n_0n}}\Big).
\end{equation}
Note that $\varupsl{L}_{ \ell,\gl}$ and $\varupsl{W}_{
\ell}$ become infinitesimal in the $n\to\infty$ limit while
$\varupsl{J}_{\ell}$ does not.
%
$\varupsl{J}_{\ell}$  is a hyperbolic rotation,
differentiating $z\mapsto z.\varupsl{J}_{\ell}$ at $z=\rho_\ell$\ shows
the angle to be $-2\Arg(\rho_\ell)\in [-\pi,0]$. Let
$$\varupsl{T}_\ell=\varupsl{A}(\Im(\rho_\ell)^{-1},-\Re(\rho_\ell))$$
correspond to the affine  map sending $\rho_\ell\in \HH$ to
$i\in \HH$, then we may write
 \be\nonumber
 \varupsl{J}_{\ell} = \varupsl{Q}(-2\Arg(\rho_\ell))^{\varupsl{T}_\ell^{-1}},
 \ee
 where
$A^B=B^{-1} A B$. Rather than following $\hat \varphi$ itself, it
will be more convenient to follow a version which is shifted so
that the fixed point $\rho_\ell$ of the rough evolution is shifted
to $i$. Moreover, in order to follow a slowly changing angle, we
remove the cumulative effect of the macroscopic rotations
$\varupsl{J}_{\ell}$. Essentially, we study the ``difference'' between
the phase evolution of the random recursion and the version with
the noise and $\lambda$ terms removed. The quantity to follow is
\begin{equation}
\varphi_{\ell,\lambda}=\hat \varphi_{\ell,\lambda} \lstar
\varupsl{T}_\ell \varupsl{Q}_{\ell-1}, \label{d_tilder}
\end{equation}
where
\begin{equation}\nonumber
\varupsl{Q}_{\ell}=\varupsl{Q}(2\Arg(\rho_0))\ldots
\varupsl{Q}(2\Arg(\rho_{\ell})),\qquad -1\le \ell\le n_0.
\end{equation}
Acting on $\UU$,  $\varupsl{Q}_\ell$ is simply a rotation about 0,
more precisely a multiplication by
\begin{equation}\label{d_eta}
\eta_\ell=\rho_0^{2} \rho_1^{2}\ldots\rho_\ell^{2}.
\end{equation}
From (\ref{e_1stepfi}) and (\ref{d_tilder}) we get that $\varphi$
evolves by the one-step operator
\begin{eqnarray}
(\varupsl{T}_\ell \varupsl{Q}_{\ell-1})^{-1}
\varupsl{R}_{\ell,\lambda}
(\varupsl{T}_{\ell+1}\varupsl{Q}_{\ell})&=&
(\varupsl{T}_\ell^{-1} \varupsl{L}_{ \ell} \varupsl{W}_{
\ell}
\varupsl{T}_{\ell+1})^{\varupsl{Q}_{\ell}}:=(\varupsl{S}_{\ell,\gl})^{\varupsl{Q}_{\ell}}.
\nonumber
\end{eqnarray}
We keep this ``conjugated'' notation because
$\varupsl{S}_{\ell,\lambda}$ corresponds to an affine transformation.

For $\ell\le n_0$  we define the corresponding target phase
function
 \be\label{e_target}
\tph_{\ell,\gl}=\tphh_{\ell,\gl}\lstar \varupsl{T}_\ell \varupsl{Q}_{\ell-1}.
\ee The following summarizes our findings and translates
the results of Proposition \ref{p_wild} to this setting.
Here and in the sequel we use the difference notation
$\Delta x_\ell=x_{\ell+1}-x_\ell$.

\begin{proposition}[Slowly varying phase function]\label{p_fi}
The functions $\varphi, \tph:\{0,1,\ldots,\lfloor
n_0\rfloor \}\times \RR\to\RR$ satisfy the following for
every $0<\ell\le n_0$:
\begin{enumerate}
\item $\varphi_{0,\lambda}=\pi$ \item $ \varphi_{\ell,\lambda}$
and $-\tph_{\ell,\gl}$ are analytic and strictly increasing
in $\lambda$, and are also independent. \item With
$\varupsl{S}_{\ell,\gl}=\varupsl{T}_\ell^{-1} \varupsl{L}_{
\ell} \varupsl{W}_{ \ell} \varupsl{T}_{\ell+1}$, we have
$\Delta
\varphi_{\ell,\gl}=\ash(\varupsl{S}_{\ell,\gl},-1,e^{i\varphi_{\ell,\gl}}\bar
\eta_{\ell})$. \label{e_phievo}
\item 
$ \hat \varphi_{\ell,\gl}=\varphi_{\ell,\gl}\lstar
\varupsl{Q}_{\ell-1}^{-1} \varupsl{T}_\ell^{-1}$.
  \item  For any $\lambda<\lambda'$ we have a.s.\ $N_{n,\lambda'}-N_{n,\lambda}=\# \left(
(\varphi_{\ell,\lambda}-\tph_{\ell,\lambda},\varphi_{\ell,\lambda'}-\tph_{\ell,\lambda'}]
\cap 2\pi \Z\right)$.
\end{enumerate}
\end{proposition}

The form
 \be\nonumber
 \varupsl{S}_{\ell,\lambda} =
 ( \varupsl{L}_{ \ell,\gl})^{\varupsl{T}_\ell} \varupsl{S}_{\ell,0}
 \ee
breaks $S$ into a deterministic $\lambda$-dependent part
and a random part that does not depend on $\lambda$. Let
$\varphi_{\ell,\lambda}^*=\varphi_{\ell,\lambda}\lstar (
\varupsl{L}_{ \ell,\gl})^{\varupsl{T}_\ell
\varupsl{Q}_\ell} $ be the intermediate phase between these
two steps. Note that $\varphi^*_{\ell,0}=\varphi_{\ell,0}$,
and
 \begin{equation}\label{e_Ldef}
( \varupsl{L}_{
\ell,\gl})^{\varupsl{T}_\ell}=\varupsl{A}\Big(1,\frac{\lambda
}{2\sqrt{n_0n}\sqrt{s(\idiota)^2-s(\ell/n)^2}}\Big)=
\varupsl{A}\Big(1,\frac{\lambda
}{2\sqrt{n_0(n_0-\ell)}}\Big)
 \end{equation} The
relative phase functions
 \be
 \alpha_{\ell,\gl}
 =
 \varphi_{\ell,\gl}-\varphi_{\ell,0},\qquad
 \alpha^*_{\ell,\gl}
 =
 \varphi^*_{\ell,\gl}-\varphi^*_{\ell,0}\nonumber
 \ee
are the main tools for counting eigenvalues in intervals.

\begin{proposition}[Relative phase function]\label{p_alpha}
The function $\alpha:\{0,1,\ldots,\lfloor n_0\rfloor \}\times
\RR\to\RR$ satisfies
\begin{enumerate}

\item $\alpha_{0,\gl}=0$, $\alpha_{\ell,0}=0$ and for each
$\ell>0$, $\alpha_{\ell,\gl}$ is analytic and strictly
increasing in $\gl$.

\item \label{alphaevo}
$\Delta
\alpha_{\ell,\gl}\;=\;\ash((\varupsl{L}_{\gl,n-\ell})^{\varupsl{T}_\ell},-1,e^{i \varphi_{\ell,\gl}}
\bar \eta_\ell)+\ash(\varupsl{S}_{\ell,0},e^{i \varphi_{\ell,\gl}^*} \bar
\eta_\ell,e^{i \varphi_{\ell,0}} \bar \eta_\ell)$
$$
\phantom{MMI}=\;\ash(( \varupsl{L}_{
\ell,\gl})^{\varupsl{T}_\ell},-1,e^{i \varphi_{\ell,\gl}}
\bar \eta_\ell)+\ash(\varupsl{S}_{\ell,0},e^{i
\varphi_{\ell,\gl}^*} \bar \eta_\ell,e^{i
\varphi_{\ell,\gl}} \bar
\eta_\ell)+\ash(\varupsl{S}_{\ell,0},e^{i
\varphi_{\ell,\gl}} \bar \eta_\ell,e^{i \varphi_{\ell,0}}
\bar \eta_\ell)
$$

\item \label{szelep1} For each
$\ell$ and $\lambda\ge 0$ we have
$\ftp{\alpha_{\ell,\lambda}}\le
\ftp{\alpha^*_{\ell+1,\lambda}}=
\ftp{\alpha_{\ell+1,\lambda}}$.
\end{enumerate}
\end{proposition}

\begin{proof}
(i)-(ii) are direct consequences of Proposition \ref{p_fi}. To
check (iii), we note
\begin{eqnarray*}
\alpha_{\ell,\lambda}&=&\varphi_{\ell,\gl}-\varphi_{\ell,0}\\
\alpha_{\ell,\lambda}^*&=&\varphi_{\ell,\gl}\lstar( \varupsl{L}_{ \ell,\gl})^{\varupsl{T}_\ell \varupsl{Q}_{\ell}}-\varphi_{\ell,0}\\
\alpha_{\ell+1,\lambda}&=&\varphi_{\ell,\gl}\lstar(
\varupsl{L}_{ \ell,\gl})^{\varupsl{T}_\ell
\varupsl{Q}_{\ell}}
(\varupsl{S}_{\ell,0})^{\varupsl{Q}_{\ell}}-\varphi_{\ell,0}\lstar
(\varupsl{S}_{\ell,0})^{\varupsl{Q}_{\ell}}.
\end{eqnarray*} Since
the map $ \varupsl{L}_{ \ell,\lambda}$ and its conjugates
are monotone in $\lambda$, we get
$\alpha_{\ell,\lambda}\le\alpha^*_{\ell,\lambda}$. Since
$(\varupsl{S}_{\ell,0})^{\varupsl{Q}_{\ell}}$ is the
lifting of a M\"obius transformation, it is monotone and
$2\pi$-quasiperiodic, whence
$\ftp{\alpha^*_{\ell,\lambda}}=
\ftp{\alpha_{\ell+1,\lambda}}$.
\end{proof}
\begin{remark}[Translation to the original matrix]
\label{r_oldmatrix}  Let $w_\ell$ denote the solution of the
discrete eigenvalue equation for the {\it original matrix}
(\ref{e_originalmatrix}). It is given in terms of the diagonal
matrix $D$ defined in the beginning of the section and the
solution $u$ of recursion (\ref{e_vrec}) as $w_\ell=(D u)_{\ell}$.
The ratios of the consecutive entries of this vector are
\[
p_\ell:=\frac{w_{\ell+1}}{w_\ell}=\frac{(Du)_{\ell+1}}{(Du)_\ell}=\frac{u_{\ell+1}}{u_\ell}
\, \frac{D_{\ell+1,\ell+1}}{D_{\ell,\ell}}=r_\ell\, \frac{
\chi_{(n-\ell-1)\beta}}{\sqrt{\beta}\,   s_{\ell-1}}.
\]
If $\ell\le n_0$ then we may further rewrite this using
$z_\ell$ as
\[
p_\ell=\frac{ \chi_{(n-\ell-1)\beta}}{\sqrt{\beta}\,
s_{\ell+1}}\, \left((z_\ell.\varupsl{U}^{-1}) \bar
\eta_{\ell-1} \Im(\rho_\ell)+\Re(\rho_\ell)\right).
\]
\end{remark}

\subsection{The discrete carousel}\label{ss_discrete}

Corollary \ref{cor_sdelim} in Section \ref{s_sdelim} shows that
the appropriate limit of the relative phase function
$\alpha_{\ell,\gl}$ is the stochastic sine equation.  In this
section we bring the discrete evolution equations in the form that
it becomes clear that their limit should be the Brownian carousel.

By \eqref{e 1step} the evolution of  $r_\ell$ is governed by a certain discrete process $\hat{
\varupsl{G}}_{\ell,\gl}$ in the hyperbolic automorphism group $\UPSL$:
$$r_\ell=r_0.\hat{\varupsl{G}}_{\ell,\lambda}=r_0.\varupsl{R}_{0,\lambda}\cdots \varupsl{R}_{\ell-1,\lambda}.$$
This process has rough jumps,
but it is a smooth function of the parameter $\lambda$. It is
therefore natural to expect that the evolution of the automorphism
$\hat{\varupsl{G}}_{\ell,\gl}^{\phantom{-1}} \hat{\varupsl{G}}_{\ell,0}^{-1}$ will have a
continuous scaling limit. In the following, we will rewrite this
expression in a form indicating the desired scaling limit.

By \eqref{d_tilder} the evolution operator  $ \varupsl{G}_{\ell,\lambda}$ of
$\varphi$ satisfies $\varupsl{G}_{\ell,\lambda}=\hat{
\varupsl{G}}_{\ell,\lambda}\varupsl{T}_{\ell}\varupsl{Q}_{\ell-1}$, and therefore
$ \varupsl{G}_{\ell,\gl}^{\phantom{-1}}  \varupsl{G}_{\ell,0}^{-1}=\hat{\varupsl{G}}_{\ell,\gl}^{\phantom{-1}} \hat{\varupsl{G}}_{\ell,0}^{-1}$. The evolution of $\varphi_{\ell,\gl}$ is given
by \be\label{e_fifi} \varphi_{\ell,\gl}=\varphi_{0,\gl}\lstar
\varupsl{G}_{\ell,\lambda}=\pi \lstar  \varupsl{G}_{\ell,\lambda}, \ee where with $A^B=B^{-1} A B$ we have
\begin{eqnarray}
\varupsl{G}_{\ell,\gl}&=&\varupsl{Y}_{0,\gl} \varupsl{X}_0\; \varupsl{Y}_{1,\gl} \varupsl{X}_1 \;\cdots
\;\varupsl{Y}_{\ell-1,\gl} \varupsl{X}_{\ell-1} \nonumber
 \\
 &=&
 \varupsl{Y}_{0,\gl} \;\varupsl{Y}_{1,\gl}^{\varupsl{G}_{1}^{-1}} \;\varupsl{Y}_{2,\gl}^{\varupsl{G}_{2}^{-1}
}\cdots \; \varupsl{Y}_{\ell-1,\gl}^{\varupsl{G}_{\ell-1}^{-1} }\; \varupsl{G}_{\ell}
 \nonumber
 \\
\varupsl{G}_\ell\; = \;\varupsl{G}_{\ell,0}&=&\varupsl{X}_0 \varupsl{X}_1 \cdots \varupsl{X}_{\ell-1},
\label{e_fievo_}
\end{eqnarray}
and we used the temporary notation
$\varupsl{Y}_{\ell,\gl}=(( \varupsl{L}_{
\ell,\gl})^{\varupsl{T}_{\ell}})^{\varupsl{Q}_\ell}$,
$\varupsl{X}_\ell=
(\varupsl{S}_{\ell,0})^{\varupsl{Q}_{\ell}}$. By
definition,
\begin{equation}\label{e_alphaevo_} \alpha_{\ell,\gl}=\varphi_{\ell,\gl}\lstar \varupsl{Q}(-\varphi_{\ell,0})
= \pi \lstar \varupsl{G}_{\ell,\lambda}\; \varupsl{Q}(-\varphi_{\ell,0}).
\end{equation}
We introduce the notation
\begin{eqnarray}\label{e_disc_BC}
\gamma_{\ell,\gl}&:=&\pi\lstar \varupsl{G}_{\ell,\lambda} \varupsl{G}_{\ell}^{-1} =
 \pi\lstar
 \varupsl{Y}_{0,\gl} \;\varupsl{Y}_{1,\gl}^{\varupsl{G}_{1}^{-1}} \;\varupsl{Y}_{2,\gl}^{\varupsl{G}_{2}^{-1}
 }\cdots \; \varupsl{Y}_{\ell-1,\gl}^{\varupsl{G}_{\ell-1}^{-1} }
\\
B_{\ell}&:=&0\lcirc \varupsl{G}_{\ell}^{-1} \in \UU.
\label{e_disc_BM}
\end{eqnarray}
With $\mathcal T$ denoting the M\"obius transformation
defined in (\ref{e_mobi}), we claim that
\[
\mathcal T(B_\ell,z)=z\lcirc   \varupsl{G}_{\ell} \, \varupsl{Q}(-\varphi_{\ell,0})
\]
with the choice of $z_0=-1$. This follows from the fact that
$\mathcal T(B_\ell,B_\ell)=0$ by definition and $\mathcal T(B_\ell,-1)=1$ by
(\ref{e_fifi}, \ref{e_fievo_}). Hence (\ref{e_alphaevo_}) becomes
\begin{equation*}
e^{i\alpha_{\ell,\gl}}=\mathcal T(B_\ell,e^{i \gamma_{\ell,\gl}}).
\end{equation*}
which is the same form as equation (\ref{e_car_sse})
relating the stochastic sine equation to the Brownian
carousel ODE.

\begin{remark}[Heuristics]
Note that $\varupsl{X}_\ell$ is approximately an infinitesimal noise element
in $\PSL$. $\varupsl{X}_\ell$ acting on $\UU$ moves $0$  infinitesimally in
a random direction.  This direction is not necessarily isotropic,
but the conjugation by the macroscopic rotation $\varupsl{Q}_\ell$ makes the
composition of consecutive $\varupsl{X}_\ell$'s move $0$ to an approximately
isotropic random direction. Thus $B_\ell$ in \eqref{e_disc_BM}
approximates hyperbolic Brownian motion started at $0$ run at a
time-dependent speed. Similarly, the $\varupsl{Y}_{\ell,\lambda}$ are
infinitesimal parallel translations, but because of the
conjugation by the macroscopic rotations $\varupsl{Q}_\ell$, their
composition approximates rotation about $0$. Thus
$\gamma_{\ell,\lambda}$ in \eqref{e_disc_BC} approximately evolves
by rotations about $B_\ell$. This is exactly how the Brownian
carousel evolves, giving a conceptual explanation of our results.
This suggests an alternative way to prove our results via the
Brownian carousel ODE \eqref{e bcode}.
\end{remark}

\section{The stochastic sine equation as a limit}\label{s_SSE}

This section describes the stochastic differential equation limit
of the phase function on the first stretch $[0,n_0(1-\eps)]$. In
the limit, this stretch completely determines the eigenvalue
behavior; this will be proved in Section \ref{s_uneventful}.

\subsection{Single-step asymptotics}\label{s_sstep}

Let $\mathcal F_\ell$ denote the $\sigma$-field generated
by the random variables $  X_0,$ $  X_{1},$ $\ldots,$ $
X_{\ell-1}$, and $  Y_{0},  Y_{1},\ldots, Y_{\ell-1}$. Let
$\ev_\ell[\,\cdot\,]$ denote conditional expectation with
respect to  $\FF_\ell$. By definition, the random variables
$\hat \varphi_{\ell,\gl}, \varphi_\ell, \alpha_\ell$ are
measurable with respect to $\FF_\ell$. Moreover, for fixed
$\gl$, both $\hat \varphi_{\ell,\gl}$ and
$\varphi_{\ell,\gl}$ are Markov chains adapted to $\mathcal
F_\ell$.

Throughout this and the subsequent sections we assume that $|\gl|$
is bounded by a constant $\bar \gl$. By default the notation  $\O(x)$ will
refer to a deterministic quantity whose absolute value is bounded
by $c |x|$, where  $c$ depends only on $\beta$ and $\bar \gl$. As $\ell$ varies $k$ will denote
$n_0-\ell$.

This section presents the asymptotics for the moments of step
$\Delta \varphi_{\ell,\gl}:=\varphi_{\ell+1,\gl}-\varphi_{\ell,\gl}$. 
Recall from Section \ref{s_svp} that $\ell$ moves on the interval
$[0,n_0]$. 
The continuum limit of $\varphi_{\ell,\gl}$ will live on the time interval $[0,1]$
so we introduce
$$ t=\frac \ell
{n_0 }\in[0,1].$$ We also introduce the rescaling  of $s(t)^2$ on
this stretch:
\begin{equation}\label{e_s_hat}
\hat s(t)^2 = \frac{s(t\,\idiota)^2-s(\idiota)^2}\idiota
\end{equation}
with $\hat s\ge 0$. This actually simplifies to
\begin{equation}
\hat s(t)=\sqrt{1-t}=\sqrt{{k}/{n_0}}
\end{equation}
in our case. As we will see later,
the scaling limit of the evolution of the relative phase
function will depend on $s$ and the scaling parameters
through $\hat s$. The fact that this function only depends
on $t$ explains why the point process
limits do not depend on the choice of the scaling
window in Theorem \ref{t_main}.
In  Section \ref{s_sch} we provide a more detailed
discussion and further implications. We will keep
the notation $\hat s$ (instead of writing $\sqrt{1-t}$) to
facilitate the treatment of a more general model discussed there.

%
%

Proposition \ref{p_fi} \ref{e_phievo} 
expresses the difference $\Delta
\varphi_{\ell,\gl}:=\varphi_{\ell+1,\gl}-\varphi_{\ell,\gl}$
via the angular shift of $\varupsl{S}_{\ell,\gl}$ and $(
\varupsl{L}_{ \ell,\gl})^{\varupsl{T}_{\ell}}$. Lemma
\ref{l_ash}, in turn, writes the angular shift in terms of
the pre-image of $i=\sqrt{-1}$. In the present case
\begin{eqnarray}
Z_{\ell,\gl}=i. \varupsl{S}_{\ell,\gl}^{-1}-i
 \;=\;
 i. \varupsl{T}_{\ell+1}^{-1}( \varupsl{L}_{ \ell,\gl}
 \varupsl{W}_{ \ell})^{-1}\varupsl{T}_\ell-i
 =
 \label{e_Zgl}
 v_{\ell,\lambda}+V_\ell,
\end{eqnarray}
where
\begin{equation}\label{e_vV}
 v_{\ell,\lambda}=
 -\frac{\lambda }{2n_0\hat s(t)}
 +
 \frac{\rho_{\ell+1}-\rho_\ell}{\Im \rho_\ell},\qquad V_\ell=
 \frac
 {  X_{\ell} +{\rho_{\ell+1}}  Y_{\ell}}
 {\sqrt{n_0}\,\hat s(t)}.
\end{equation}
The random variable $V_{\ell}$ is measurable with respect to
$\FF_{\ell+1}$,  but independent of $\FF_{\ell}$.


By Taylor expansion  we have the following estimates for the
deterministic part of $Z_{\ell,\gl}$:
\begin{equation} v_{\ell,\lambda}=
\frac{v_\lambda(t)}{n_0}+\O(k_{\phantom 0}^{-2}),\qquad
v_\lambda(t)= - \frac{\gl}{2\hat s(t)}+\frac{
 \frac{d}{dt}\rho(t)}{\Im \rho(t)},\qquad |v_{\lambda}(t)|\le c \frac{n_0}{k},\label{e_vbecs}
\end{equation}
where we abbreviate  $\rho(t)=\rho(n_0/n,tn_0/n)=\rho_\ell$, see
\eqref{e_crho2}. The behavior of the random term is governed by
\begin{equation}
\begin{array}{rclrcl}
 \ev V_{\ell} &=& \O(n-\ell)^{\nicefrac{-3}{2}} k_{\phantom 0}^{\nicefrac{-1}{2}} \qquad\qquad&
 \ev |V_\ell^2|
 &=&
\frac1{n_0}p(t)+\O(n-\ell)^{-1}k_{\phantom 0}^{-1} \\[.2em]
 \ev V_\ell^2
&=&\frac1{n_0}q(t)+\O(n-\ell)^{\nicefrac{-1}{2}}k_{\phantom
0}^{\nicefrac{-3}{2}}, &\ev |V_{\ell}|^d &=& \O(k_{\phantom
0}^{-d/2}),\quad d=3,4
\end{array}\label{e_Vmom}
\end{equation}
where
\begin{equation}
p(t)=\frac{4 }{\beta \hat s ^2}=\frac{4 n_0}{\beta k} ,\qquad q(t)=\frac{2 (1+\rho^2)}{
\beta \hat s ^2}.
\label{e_pqdef}
\end{equation}
Here the error terms come from the moment asymptotics
(\ref{e_mom}), the size of $\hat s$, and from the bounds
$$
\rho_{\ell+1}-\rho_{\ell}=\O(n-\ell)^{\nicefrac{-1}{2}}k_{\phantom
0}^{\nicefrac{-1}{2}},\qquad
\frac{d}{dt}\rho-\frac{\rho_{\ell+1}-\rho_\ell}{n_0}=\O(n-\ell)^{\nicefrac{-1}{2}}k_{\phantom
0}^{\nicefrac{-3}{2}}.
$$

\begin{proposition}[Single-step asymptotics for $\varphi_{\ell,\gl}$]
\label{p_sstepfi} For $\ell\le n_0$ with $t=\ell/n_0$ and $k=n_0-\ell$ we have
\begin{eqnarray}\label{e_onestepfi}
  \ev\left[\Dfi_{\ell,\gl} \big| \varphi_{\ell,\gl}=x \right]
  &=&\frac{1}{n_0} b_\gl(t)
+\frac1{n_0} \mbox{osc}_1+ \O(k_{\phantom
0}^{\nicefrac{-3}{2}}) = \O(k_{\phantom 0}^{-1}),
 \\[.4em]
 \label{e_onestepfi2}
  \ev\left[ \Dfi_{\ell,\gl} \Dfi_{\ell,\gl'}\big|\varphi_{\ell,\gl}=x, \varphi_{\ell,\gl'}=y \right]
  &=&
\frac1{n_0} a(t,x,y)
  +\frac1{n_0}\mbox{osc}_2+ \O(k_{\phantom 0}^{\nicefrac{-3}{2}}),
 \\[.4em]
\nonumber
 \ev_\ell\left|\Dfi_{\ell,\gl}\right|^d
 &=&\O(k_{\phantom 0}^{-d/2}), \qquad d=2,3,
\end{eqnarray}
where
\begin{eqnarray}
b_\gl= \frac{\lambda}{2\hat s}-\frac{\Re
\frac{d}{dt}\rho}{\Im \rho}+\frac{\Im (\rho^2)}{2\beta \hat s
^2},\qquad
a=\frac{2}{\beta \hat s ^2}  \Re\left[e^{i (y-x)}\right]+
\frac{3+\Re
\rho^2}{\beta \hat s ^2}.
\label{e_drift_zaj}
\end{eqnarray}
The oscillatory terms are
\begin{eqnarray}\label{e_oneerr1}
 \mbox{osc}_1&=&\Re \left((-v_{\lambda}-i q /2)
    e^{-i\,x} \eta_\ell\right)
+ \Re \left(i e^{-2 i\,x}
 \eta_\ell^2\,q\right)/4,
 \\[.4em]
 \mbox{osc}_2&=&
 p\Re \left(
  e^{-i\,x}\eta_\ell+
  e^{-i\,y}\eta_\ell   \right)/2
  +   \Re \left(q (e^{-i\,x} \eta_\ell+e^{-i\,y}\eta_\ell +
e^{-i\,(x+y)}\eta_\ell^2  ) \right)/2.
\nonumber
\end{eqnarray}

\end{proposition}

\begin{proof}
By Proposition \ref{p_fi} \ref{e_phievo} the difference
 $\Delta \varphi_{\ell,\gl}$ can be written
\begin{eqnarray}
\Delta \varphi_{\ell,\lambda}&=&\ash(\varupsl{S}_{\ell,\gl},-1,z\bar
\eta)\nonumber\\
&=&\Re\left[-(1+\bar z\eta)Z-\frac{i(1 + \bar z\eta)^2}{4}\,Z^2
 \right]+\O(Z^3)\label{e_ashfifi1}\\
 &=& -\Re Z+\frac{\Im
 Z^2}4+\eta\mbox{ terms}+\O(Z^3).\nonumber
\end{eqnarray}
where we used
 $Z=Z_{\ell,\lambda}$, $\eta=\eta_\ell$ and $z=\exp(i \varphi_{\ell,\gl})$.
The estimate (\ref{e_ashfifi1}) is from the quadratic
expansion \eqref{e_ash} of the angular shift in Lemma
\ref{l_ash}. Note that since the second argument of $\ash$
is $-1$, we do not need an upper bound on $|Z|$. We take
expectations, the error term becomes
\[
\O(\ev|Z|^3)=\O(|v_{\ell,\gl}|^3+\ev
|V_\ell|^3)=\O(k_{\phantom 0}^{\nicefrac{-3}{2}}).
\]
By (\ref{e_Zgl}, \ref{e_Vmom}) we may replace the $\ev Z$ ,
$\ev |Z|^2$ and $\ev Z^2$ terms by $v_\lambda(t)$ , $p(t)$
and $q(t)$ while picking up an error term of
$\O(k_{\phantom 0}^{-2})$. Significant contributions come
only from the non-random terms $v_{\ell,\lambda}$ of $Z$
and the expectation of $V_\ell^2$. We are then left with
oscillatory terms with $\eta$, error terms, and the main
term
\begin{eqnarray*}
-\Re v_{\ell,\lambda} +\Im \ev V_\ell^2/4 &=&
\frac{1}{n_0}\left(-\Re v_\lambda+\Im q \right)
+\O(k_{\phantom 0}^{\nicefrac{-3}{2}})\\
&=&\frac{1}{n_0}\left[ \frac{\lambda}{2\hat s}-\frac{\Re
\rho'}{\Im \rho} +\frac{\Im (\rho^2)}{2\beta \hat s ^2}
\right]+\O(k_{\phantom 0}^{\nicefrac{-3}{2}}).
 \end{eqnarray*}
 The error terms
come from the moment bounds (\ref{e_mom}) of $X$, $Y$ and
from the discrete approximation of the derivative $\Re
\rho'$; their exact order is readily computed. This gives
(\ref{e_onestepfi}) with (\ref{e_oneerr1}).  The
$\O(k_{\phantom 0}^{-1})$ bound comes from evaluating
 the continuous functions in the main and oscillatory  terms at $t=\ell/n_0$.

For $\ev_\ell\left[ \Delta\varphi_{\lambda,\ell}
\Delta\varphi_{\lambda',\ell}\right]$ one uses the linear
approximation of the angular shift to get
$$
\Delta \varphi_{\ell,\lambda}=\Re[-(1+\bar z\eta)Z_{\ell, \lambda}]+\O(Z_{\ell, \lambda}^2)
$$
and similarly for $\lambda'$. After multiplying the two estimates and taking expectations,
only the noise terms in $Z_{\ell,\lambda}, Z_{\ell,\lambda'}$ contribute. Namely, with $V=V_\ell$, we have
\begin{eqnarray}
 \ev_\ell \,[\Delta \varphi_{\ell,\gl}\Delta
 \varphi_{\ell,\gl'}]&=&\nonumber
 \frac14 \;\ev  \left[ (1+\eta\bar z) V+ (1+\bar \eta z)
 \bar V \right] \nonumber \left[ (1+\eta\bar z' )V+ (1+\bar
 \eta  z') \bar V \right]+\O(k_{\phantom 0}^{\nicefrac{-3}{2}})
 \\
 &=&\frac12\; \Re(1+\bar zz')\ev |V|^2 +\frac12\;
 \Re \ev V^2 +\eta \textup{ terms }+\O(k_{\phantom 0}^{\nicefrac{-3}{2}})
 \nonumber
\end{eqnarray}
where we used $z=\exp(i \varphi_{\ell,\gl})$ and $z'=\exp(i \varphi_{\ell,\gl'})$.
Formula (\ref{e_onestepfi2}) now follows from the asymptotics of $\ev |V|^2, \ev V^2$. The last
claim follows from the third moment asymptotics of $X,Y$.
\end{proof}

\subsection{Continuum limit of the phase evolution}\label{s_sdelim}
\newcommand{\lip}{\stackrel{P}{\longrightarrow}}
\newcommand{\IO}{\mathcal I\,}

The goal of this section is to show that the first stretch
of the phase evolution converges in law to the solution of the
SDE (\ref{e_SDE_alpha}). Typically, the phase evolves in an oscillatory manner,
so we have to take advantage of  averaging. Our main
tool will be the following proposition, based on \cite{SV}
and \cite{EthierKurtz}, which allows for averaging of the
discrete evolutions.


\begin{proposition}\label{p_turboEK}
Fix $T>0$, and for each $n\ge 1$ consider a Markov chain
$$
(X^n_\ell\in \RR^d,\, \ell =1\ldots  \lfloor nT
\rfloor).
$$
Let $Y^n_\ell(x)$ be distributed as the increment
$X^n_{\ell+1}-x$ given $X^n_\ell=x$. We define
$$
b^n(t,x)= n E[ Y_{\lfloor nt\rfloor}^n(x)],\qquad
a^n(t,x)=nE[ Y_{\lfloor nt\rfloor}^n(x) Y_{\lfloor
nt\rfloor}^n(x)^\textup{T}].
$$
Suppose that as $n\to \infty $ we have
\begin{eqnarray}
|a^n(t,x)-a^n(t,y)|+|b^n(t,x)-b^n(t,y)|&\le&
c|x-y|+o(1)\label{e lip}\\
 \sup_{x,\ell}  E[|Y^n_\ell(x)|^3] &\le& cn^{\nicefrac{-3}{2}} \label{e 3m},
\end{eqnarray}
and that there are functions $a,b$ from $\R\times[0,T]$ to
$\R^{d^2}, \R^d$ respectively with bounded first and second
derivatives so that
\begin{eqnarray}
\sup_{x,t} \Big|\int_0^t a^n(s,x)\,ds-\int_0^t a(s,x)\,ds
\Big|+\sup_{x,t} \Big|\int_0^t b^n(s,x)\,ds-\int_0^t b(s,x)\,ds
\Big| &\to& 0. \label{e
fia}
\end{eqnarray}
Assume also that the initial conditions converge weakly:
$$
X_0^n\cd X_0.
$$
Then $(X^n_{\lfloor n t\rfloor}, 0 \le t \le T)$ converges in law
to the unique solution of the SDE
$$
dX = b \,dt + a\, dB, \qquad X(0)=X_0.$$
\end{proposition}

We will prove this in Appendix \ref{a_sdetools}. The next
lemma provides the averaging conditions for the above
proposition. Recall that   $\Dfi_{\ell,\gl}=\varphi_{\gl,
\ell+1}-\varphi_{\ell,\gl}$.
%

\begin{lemma}\label{l_porgosdrift}
Fix $\gl, \gl'$ and $\eps>0$. Then for any $\ell_1\le n_0
(1-\eps)$
\begin{eqnarray}\label{bsum}
\frac1{n_0}\sum_{\ell=0}^{\ell_1-1}
\ev\left[\Dfi_{\ell,\gl}\,|\,\varphi_{\ell,\gl}=x
\right]&=&\frac1{n_0}\sum_{\ell=0}^{\ell_1-1} b_\gl(t) +\O( \nn
n_0^{\nicefrac{-3}{2}}+n_0^{\nicefrac{-1}{2}})\\
\frac1{n_0}\sum_{\ell=0}^{\ell_1-1}
\ev\left[\Dfi_{\ell,\gl}
\Dfi_{\ell,\gl'}\,|\,\varphi_{\ell,\gl}=x ,\,
\varphi_{\ell,\gl'}=y \right]
&=&\frac1{n_0}\sum_{\ell=0}^{\ell_1-1}
a(t,x,y)+
\O( \nn
n_0^{\nicefrac{-3}{2}}+n_0^{\nicefrac{-1}{2}}) \nonumber
\end{eqnarray}
where $t=\ell/n_0$, the functions $b_\gl, a$ are defined in
(\ref{e_drift_zaj}), and the implicit constants in $\O$ depend
only on  $\eps,\beta, \bar \lambda$.
\end{lemma}

\begin{proof}
Summing (\ref{e_onestepfi}) we get \eqref{bsum} with a
preliminary error term
\begin{eqnarray*}
\frac1{n_0}\sum_{\ell=0}^{\ell_1-1}  \Re (e_{1,\ell}\,
\eta_{\ell})+\frac1{n_0}\sum_{\ell=0}^{\ell_1-1}  \Re
(e_{2,\ell} \,\eta_{\ell}^2)+\O(k_1^{\nicefrac{-1}{2}}),
\end{eqnarray*}
where the first two terms will be denoted
$\zeta_1,\zeta_2$. Here
$$e_{1,\ell}=(-v_\lambda(t)-i q(t)/2) e^{-i  x},\qquad e_{2,\ell}=i q(t) e^{- 2i
x}/4$$ and $k_1=n_0-\ell_1>cn_0$, where for this proof $c$
denotes varying constants depending on $\eps$. Using the
fact that $v_\gl, q$ and their first derivatives are
continuous on $[0,1-\eps]$
 we get
$$|e_{i,\ell}|<c,\qquad |e_{i,\ell}-e_{i,\ell+1}|<c
n_0^{-1}.$$ Thus by the oscillatory sum Lemma
\ref{l_blackboxnew},
\begin{eqnarray*}
|\zeta_1|\le  c\sum_{\ell=0}^{\ell_1-1} (\nn/\sqrt{k}+1)
n_0^{-2} +c(\nn/\sqrt{k_1}+1) n_0^{-1}\le
 c ( \nn n_0^{\nicefrac{-3}{2}}+n_0^{-1}).
\end{eqnarray*}
Similarly, if we apply the same estimate for the second sum, we get
\begin{equation*}
|\zeta_2|\le c \sum_{\ell=0}^{\ell_1-1}
(\nn/\sqrt{k}+\sqrt{n_0}/\nn) n_0^{-2} +c
(\nn/\sqrt{k_1}+\sqrt{n_0}/\nn) n_0^{-1}\le c
( \nn
n_0^{\nicefrac{-3}{2}}+ \nn^{-1}
n_0^{\nicefrac{-1}{2}}).
\end{equation*}
We could also estimate $\zeta_2$ by taking absolute value
in each term. Using (\ref{e_Vmom}) together with (\ref{e_crho}) we get
\[
|q(t)|=\frac{\mu_n}{\beta \hat s(t)^2 \sqrt{\mu_n^2/4+k-1/2}}\le C \mu_n k^{\nicefrac{-3}{2}}n_0^{-1}
\]
which leads to
\begin{eqnarray*}
|\zeta_2|\le c \sum_{\ell=0}^{\ell_1-1}
 \nn k^{\nicefrac{-3}{2}} \le c  \nn
n_0^{\nicefrac{-1}{2}}.
\end{eqnarray*}
Using this bound for $ \nn\le 1$ and the previous one for
$ \nn>1$ we get the desired estimate \eqref{bsum}. The
asymptotics of the second sum follow similarly.
\end{proof}

We are now ready to state and prove the continuum limit theorem.

\newcommand{\bfi}{\underline{\varphi}}
\newcommand{\bal}{\bm{\alpha}}
\newcommand{\bvf}{\bm{\vf}}

\begin{theorem}[Continuum limit of the phase function]\label{thm_sdelim}
Suppose that $n_0/n\to 1/(1+\nu)$ with $\nu \in [0,\infty]$. Then the continuous function $\rho(t)=\rho(n_0/n,tn_0/n)$ (see (\ref{e_crho})) converges to a limit for which we use the same
notation.
%
Let $B$ and $Z$ be a real and a complex Brownian motion,
and for each $\lambda\in \R$ consider the strong solution of
\begin{eqnarray}\label{e_sdelimfi}
  d \varphi_{\lambda}
  &=&
\left[ \frac{\lambda}{2\hat s}-\frac{\Re \rho'}{\Im \rho}
+\frac{\Im (\rho^2)}{2\beta \hat s ^2} \right]dt +
\frac{\sqrt2\Re(e^{-i\varphi_\lambda}dZ)}{\sqrt{\beta}\,\hat s}
  +\frac{\sqrt{3+\Re \rho^2}}{\sqrt{\beta}\,\hat s}dB,
  \\[.4em]
  \varphi_{\lambda}(0)&=&\pi. \nonumber
\end{eqnarray}
Then we have
$$
 \varphi_{\gl, \lfloor n_0 t\rfloor }\cd \varphi_\lambda(t),\qquad \textup{as  $n\to \infty$},
$$
where the convergence is in the sense of finite dimensional
distributions for $\lambda$ and in path-space $D[0,1)$ for
$t$.
\end{theorem}

\begin{remark} From (\ref{e_crho}) we get that the limit of $\rho(n_0/n, tn_0/n)$ as $n_0/n\to 1/(1+\nu)$ is
$$\rho(t)=\frac{\nu }{\nu +1-t}+i {\frac{\sqrt{(1-t) (2 \nu +1-t)}}{\nu +1-t}}$$
with $\rho(t)=1$ if $\nu=\infty$.
Thus equation (\ref{e_sdelimfi}) can be written as
 \begin{equation}\label{sdelim2}
 \sqrt{1-t}\;
 d\varphi_\lambda
 =
 \frac{\lambda }{2}\;
 dt+\sqrt{\frac2 \beta}\,\Re( e^{-i\varphi_\lambda}dZ_t)+
 \left(\frac{1}\beta-\frac12\right)
 \frac{\sqrt{\nu}}{\nu+1-t}
 dt+
 \sqrt{\frac{2(2\nu+1-t)}{{\beta(\nu+1-t)}}}
dB,
 \end{equation}
where the last two terms are 0 and $2
\beta^{\nicefrac{-1}{2}} dB$, respectively when
$\nu=\infty$.
\end{remark}

\begin{proof}[Proof of Theorem \ref{thm_sdelim}]
It suffices to show that for any finite sequence
$(\lambda_1,\ldots, \lambda_d)$ and for any $T<1$ the
following holds on the time interval $[0,T]$,
\begin{equation}\nonumber
(\varphi_{ \lfloor n_0 t\rfloor,\gl_1},\ldots, \varphi_{ \lfloor n_0
t\rfloor,\gl_d}) \cd (\varphi_{\lambda_1}(t),\ldots,
\varphi_{\lambda_d}(t)).
\end{equation}
We will use Proposition \ref{p_turboEK}. For $x\in \R^d$ let
\begin{eqnarray*}
\begin{array}{rclrcl}
\bfi_{\ell}&=&(\varphi_{\ell,\gl_1},\ldots, \varphi_{
\ell,\gl_d}),&\qquad \Delta
\bfi_{\ell}&=&\bfi_{\ell+1}-\bfi_{\ell},
\\[6pt]
b_{\ell}(x)&=&n_0 \ev\left[ \Delta \bfi_\ell \big| \bfi_\ell=x\right],&\qquad
a_\ell (x)&=&n_0 \ev\left[ (\Delta \bfi_\ell)
(\Delta \bfi_\ell )^T \big| \bfi_\ell=x \right].
\end{array}
\end{eqnarray*}
Recall the estimates (\ref{e_onestepfi}) and
(\ref{e_onestepfi2}). Since
$ \nn^2/(4n_0)\to \nu$, the functions $b_\gl, a$
defined in \eqref{e_drift_zaj} converge uniformly on
$[0,T]$ to $\hat b_\gl,\hat a$ which are also
defined by \eqref{e_drift_zaj} but in terms of the limit
of $\rho$ (recall that $\hat s$ is just $\sqrt{1-t}$).

Using this with Lemma \ref{l_porgosdrift} we get that
\begin{eqnarray}
\sup_{x\in R^d,t\le T} |\int_0^t n_0 b_{\lfloor n_0
s\rfloor }(x)ds-\int_0^t \tilde b(x,s) ds| &\to& 0,
\nonumber
 \\
\sup_{x\in R^d,t\le T} |\int_0^t n_0 a_{\lfloor n_0
s\rfloor }(x) ds-\int_0^t \tilde a(x,s) ds| &\to& 0,
 \label{e_fitty2}
\end{eqnarray}
where
\begin{eqnarray}\nonumber
\tilde b(x,t)&=& \left(\hat b_{\gl_1}(t),\ldots,\hat
b_{\gl_d}(t) \right) , \qquad \Big(\tilde a(x,t) \Big)_{j,k} =\hat
a(t,x_j,x_k)
.\nonumber
\end{eqnarray}
This means that condition (\ref{e fia})
in Proposition \ref{p_turboEK}  is satisfied. Because of
(\ref{e_drift_zaj}) and the moment
bounds we can see that (\ref{e lip}) and (\ref{e 3m}) are
also satisfied, thus $(\varphi_{\gl, \lfloor n_0
t\rfloor},\ldots, \varphi_{\gl_d, \lfloor n_0 t\rfloor})$
converges weakly to the SDE corresponding to $\tilde b(x,t),
\tilde a(x,t)$. The only thing left is to identify the
limiting SDE from the functions $\tilde b(x,t), \tilde a(x,t)$.
This follows easily, by observing that if $Z$ is a complex
Gaussian with independent standard real and imaginary parts
and $\omega_1,\omega_2\in \CC$ then
\begin{equation*}
    \ev \Re(\omega_1 Z)\Re(\omega_2 Z)=\ev(\omega_1
Z+\overline{ \omega_1 Z})(\omega_2 Z+\overline{ \omega_2
Z})/4 =({\omega_1 \overline \omega_2}+{\omega_2 \overline
\omega_1} )/2= \Re(\omega_1\overline{\omega}_2). \qedhere
\end{equation*}
\end{proof}
Theorem \ref{thm_sdelim} leads to the following corollary.
\begin{corollary}\label{cor_sdelim}
Let $W_t$ be complex Brownian motion with standard real and
imaginary parts and consider the strong solution of the  following one-parameter
family of SDEs
\begin{equation}\nonumber
\sqrt{1-t}\; d\alpha_\lambda  = \frac{\lambda }{2}\;
 dt+ \sqrt{2/\beta}\;\Re( (e^{-i\alpha_\lambda}-1)d W_t).
 \end{equation}
Then
\be
 \alpha_{ \lfloor n_0 t\rfloor, \gl}\cd \alpha_\lambda(t),\qquad \textup{as  $n\to \infty$}\label{e_alphalim}
\ee
where the convergence is in the sense of finite dimensional
distributions for $\lambda$ and in path-space $D[0,1)$ for
$t$.
\end{corollary}

\begin{proof}
If $ \nn^2/(4n_0)$ converges to a finite or infinite value as
$n\to\infty$ then the statement follows immediately with $W_t=e^{i
\varphi_{0}(t)}Z_t$. This implies that  for any subsequence of $n$
we can choose a further subsequence along which (\ref{e_alphalim})
holds, a characterization of convergence.
\end{proof}

\subsection{Why are the limits in different windows the
same? Universality and non-universality}\label{s_sch}

This subsection is meant to explain why the continuum limit
of the relative phase function does not depend on the
choice of the scaling window $\mu_n$. In order to do that,
we will discuss a more general model where this is not
necessarily true.

The discussion of this section is not an integral part of
the proof of the main theorem; the goal is to provide some
additional insight for the results.

\paragraph{A more general  model.} The following is a generalization of the model
\eqref{e_matrix}. Consider random tridiagonal $n\times n$
matrices with diagonal elements $X_0,X_1,X_2,\dots$ and
off-diagonal elements $s_1, s_2, s_3,\dots$ and $s_0+Y_0,
s_1+Y_1, s_2+Y_2\dots$, see (\ref{e_matrix}). The random
variables $X_i, Y_i$ are independent  with mean
approximately zero, variance approximately $2/\beta$ and a
bounded fourth moment. The deterministic numbers $s_\ell$
depend on $n$ and are approximately $\sqrt{n} s(\ell/n),$
where $s(t)$ is a nonnegative, sufficiently smooth
decreasing function on [0,1].  In the case of
$\beta$-ensembles we have $s(t)=\sqrt{1-t}$.
%
%
%

We will try to understand the  point process limit of the eigenvalues of these tridiagonal matrices near $\mu_n$ where the scaling parameter $\mu_n$ will be in the interval  $[s(1) \sqrt{n},s(0) \sqrt{n}]$. It turns out that in this more general setup the arguments of the previous two sections follow through essentially without change. This is the main reason why we expressed everything  in terms of $\rho, s, \hat s$, instead of using the sometimes more simple explicit values.

We first have to identify the  scaling around $\mu_n$ so
that we have a nontrivial limit, for this we consider
equation (\ref{e_crho}). We define $n_0\in[0,n]$ as the
unique value for which
\begin{equation}
s(n_0/n)={\mu_n}/{\sqrt{4n}}.\label{e_n0_gen}
\end{equation}
Then for $\ell\in [0,n_0]$ the complex number $\rho_\ell\in \HH$ defined through (\ref{e_crho}) (see also (\ref{e_crho2})) is of unit length
and our scaling around $\mu_n$ will be given by (\ref{e_scalingnu}):
\[
\Lambda=\mu_n+\lambda/(2\sqrt{n_0}).
\]
\begin{figure}
\begin{center}
\includegraphics*[width=300pt]{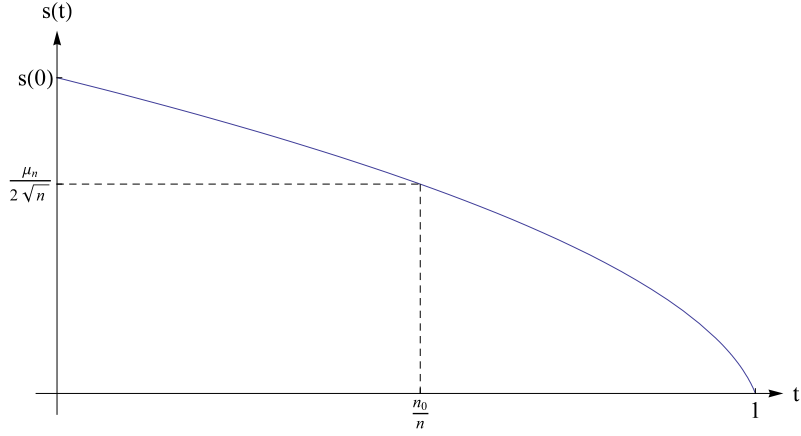}\phantom{MMMM}\end{center}
\caption{The definition of the scaling parameter $n_0$}
\end{figure}
The subsequent computations, i.e.~the introduction of the
slowly varying phase function, the single-step asymptotics
and the continuum limit of the phase evolution can be
carried out the same way as we have done in Subsections
\ref{s_model} and \ref{s_svp} and Section \ref{s_SSE}. The
assumption $n_0^{-1} \mu_n^{2/3}\to 0$ as $n\to \infty$
will ensure that the arising error terms are negligible.

Thus, according to Theorem \ref{thm_sdelim} if $n_0/n$ converges to a constant as $n\to \infty$ then $\varphi_{\lfloor n_0 t\rfloor,\lambda}\cd \varphi_\lambda (t)$ where  $\varphi_\gl(t)$ is the solution of (\ref{e_sdelimfi}). This gives the following limit for the relative phase function $\alpha_{\lfloor n_0 t\rfloor,\lambda}$:
\begin{equation}
\hat s \,d\alpha_\gl=\frac{\gl}{2} dt+\sqrt{2/\beta} \Re\left((e^{-i\alpha_\gl}-1) dW_t  \right), \qquad \alpha_\gl(0)=0.\label{e_alphagen}
\end{equation}
Here $\hat s$ is given by the $n\to \infty$ limit of
 \begin{equation}\label{trafo}
 \hat
s_n(t)=\sqrt{\frac{s(n_0/n \, t)^2-s(n_0/n)^2}{n_0/n}}.
 \end{equation}
There are various ways of interpreting equation (\ref{e_alphagen}),
perhaps the most intuitive is via the Brownian carousel
which is already apparent in the discrete evolution, see
Section \ref{ss_discrete}.

The fact that the limiting equation depends on $s$ through
the function $\hat s$ explains two phenomena. First,  in
the $\beta$-ensemble case $s(t)=\sqrt{1-t}$ and thus
(\ref{trafo}) gives that $\hat s(t)=\sqrt{1-t}$
\emph{regardless} of the value of the limit of $n_0/n$.
This shows why all limits will be governed by the same
stochastic differential equation, regardless on the choice
of the scaling parameter. However, in the more general
model, non-universality holds; the limiting stochastic
differential equation \eqref{e_alphagen} depends not only
on the limit of $s$ but also on the scaling window.

Second, consider a general $s$, with $s'(0)<0$, and choose $\mu_n$
so that
\begin{equation}
     (2\sqrt{n} s(0)-\mu_n)n^{1/6}\to \infty, \quad \textup{and} \quad
(2\sqrt{n} s(0)-\mu_n)n^{-1/2}\to 0 .\label{e_scaling_gen}
\end{equation}
This means that scaling parameter $\mu_n$ is close, but not too close to the edge of the spectrum $2s(0)\sqrt{n}$.
Since $\mu_n/\sqrt{4n}\to s(0)$, by (\ref{e_n0_gen}) we have $n_0/n\to 0$ and
\[
n_0=s'(0)^{-1} \sqrt{n} (\sqrt{n} s(0)-\mu_n/2)+o(n_0).
\]
Thus $n_0^{-1} \mu_n^{2/3}\to 0$, so we can apply our previous results.  From (\ref{trafo}) we get
\[
\hat s(t)=\lim \hat s_n(t)=c \sqrt{1-t},\qquad c=|(s^2)'(0)|^{1/2}
\]
which means that  the limiting sde (\ref{e_alphagen}) is the same as in the $\beta$-ensemble case, after a linear rescaling with a new parameter $\tilde \beta=\beta c^2$.
This means that even for a general choice of $s$ the point process limit of the eigenvalues in the scaling regime (\ref{e_scaling_gen}) is given universally by  the $\Sineb$ process.

A similar statement of universality holds for a class of 1-dimensional discrete random
Schr\"odinger operators with tridiagonal matrix representation. Consider the
a symmetric tridiagonal matrix with diagonal
and off-diagonal terms
$$
\tilde X_0,\,\tilde X_1,\,\ldots \qquad \tilde Z_0/2+s_0,\,\tilde
Z_1/2+s_1,\,\ldots\qquad
$$
where $\tilde X_i,\tilde Z_i$ are independent random variables with mean
approximately zero, variance approximately $\sigma^2$ and a bounded eighth moment. The deterministic numbers $s_\ell$ depend on $n$ and are approximately $\sqrt{n} s(\ell/n)$, where $s(t)$ is again a nonnegative, sufficiently smooth
decreasing function on [0,1].
 This gives the matrix representation of   a 1-dimensional discrete random
Schr\"odinger operator.

The analyze its spectrum, we first conjugate it with a diagonal matrix to transform it into a form similar to (\ref{e_matrix}). Choosing an appropriate diagonal matrix we can transform any off-diagonal pair $(\tilde Z_\ell/2+s_\ell,\tilde Z_\ell/2+s_\ell)$ into $(a_\ell,(\tilde Y_\ell/2+s_\ell)^2/a_\ell)$ with any nonzero $a_\ell$, while  the diagonal elements stay the same.
a simple computation shows that if we set
$a_\ell=\sqrt{n} \,s(\ell/n+1/(2n))\simeq \sqrt{n} \,s((\ell+1)/n)$ then the off-diagonal entries above the diagonal will have mean approximately equal to $\sqrt{n} s(\ell/n)$, variance approximately equal to $\sigma^2$ and a bounded fourth moment.
 Thus the previous results may be applied with $\beta=2/\sigma^2$. In particular near the edge of the spectrum (but not very near: see (\ref{e_scaling_gen})) the point process limit of the eigenvalues will be given universally by the $\Sineb$ process.

We would like to note that the point process limit near the edge of the spectrum (i.e.~when $(\mu_n-2\sqrt{n} s(0))n^{1/6}$ converges to a finite constant) one gets the $\Airyb$ process (see  \RRVlong,  Section 5). This allows us to complete the proof in the general case with arguments analogous to the following section.
To avoid excessive technicalities, we
chose to focus on the beta ensemble case in this paper. We plan to treat the more general
case in detail in a future work.

\section{Asymptotics in the uneventful stretch}
\label{s_uneventful}

Section \ref{s_SSE} describes the stochastic differential equation
limit of the phase function on the first stretch
$[0,n_0(1-\eps)]$. Here we show that in the limit, this stretch
completely determines the eigenvalue behavior.

\subsection{The uneventful middle stretch}\label{s_middle}

 The middle stretch is the discrete
time interval  $[m_1,m_2]$ with
\begin{equation}
m_1=\lfloor n_0 (1-\eps)
\rfloor,\qquad m_2=\lfloor n- \nnn- \kappa\,
(\nn^{\nicefrac{2}{3}}\vee 1)\rfloor\label{e_middle}
\end{equation}
  for
$\eps\in(0,1)$ and $\kappa>0$. The goal of this section is to
prove that if $\alpha_{\ell,\gl}$ is close to an integer
multiple of $2\pi$ after time $m_1$ then it changes little
up to time $m_2$. More precisely, we have
\begin{proposition}\label{l_middlestretch}
There exists a constant $c=c(\bar \gl,\beta)$ so that with
$y=n_0^{\nicefrac{-1}{2}} ( \nn^{\nicefrac{1}{3}}\vee 1)$ we have
\begin{equation}\label{e_masodik_}
\ev \big[\left|
\alpha_{\ell_2,\gl}-\alpha_{\ell_1,\gl}\right|\wedge 1 \big|
\mathcal F_{\ell_1} \big] \le c (\dist(\alpha_{\ell_1,\gl},2\pi
\Z)
 +\sqrt{\eps}+y+\kappa^{-1}),
\end{equation}
for all $\kappa>0, \eps\in(0,1)$, $\gl\le |\bar \gl|$ and $m_1\le
\ell_1\le \ell_2\le m_2$.
\end{proposition}

The first step is to estimate
$\Dal_{\ell,\gl}=\alpha_{\ell+1,\gl}-\alpha_{\ell,\gl}$
using the angular shift Lemma \ref{l_ash} with
$z=Z_{\ell,\gl}$ defined in \eqref{e_Zgl}. For the finer
asymptotics of the lemma, the condition $|z|<\nicefrac13$
is needed. For this, we truncate the original random
variables $X_\ell, Y_\ell$. For $m_1\le \ell \le m_2$,
introduce the random variables $\tilde X_\ell$, $\tilde
Y_\ell$ which agree with $X_\ell,Y_\ell$ on the event
\begin{equation}
|X_{\ell}|, |Y_{\ell}|\le \frac1{10} \sqrt{n_0}\, \hat
s(\ell/n_0),\label{e_cutoff}
\end{equation}
and are zero otherwise; this event depends on $n$ via $\hat s$. By
Markov's inequality and  the fourth moment assumption
(\ref{e_mom}) for $X_\ell,Y_\ell$, this event has probability at
least $1-c(n_0-\ell)^{-2}$. Summing this for $\ell\le m_2$ shows
that the total probability that the truncation has an effect is at
most $c\kappa^{-1}$. This can be absorbed in the error term
$\kappa^{-1}$ in \eqref{e_masodik_}, so it suffices to prove
Proposition \ref{l_middlestretch} for the truncated random
variables.

To keep the notation under control, we will drop the tildes and
instead modify the assumptions on $X_\ell,Y_\ell$. Namely,
denoting  $k=n_0-\ell$ we assume the bounds \eqref{e_cutoff} and
the modified moment asymptotics
\begin{eqnarray}
\begin{array}{c|c|c|c}
  \mbox{moment} & 1^{\textup{st}} & 2^{\textup{nd}} & 4^{\textup{th}} \\ \hline
   & \O(k_{\phantom 0}^{\nicefrac{-3}{2}}) & 2/\beta+\O(k_{\phantom 0}^{-1}) & \O(1)
\end{array}\nonumber
\end{eqnarray}
which follow from the original ones (\ref{e_mom}) and our choice
of truncation. With $p,q$ defined in \eqref{e_pqdef}, this changes
the moment asymptotics of $V_{\ell}$ (\ref{e_vV}) the following way:
\begin{equation}
\begin{array}{c|c|c|c}
    \ev V_{\ell} &\ev V_\ell^2&  \ev |V_\ell^2|& \ev |V_{\ell}|^4 \\ \hline
    \O (k_{\phantom 0}^{-2})&\frac1{n_0}q(t)+\O(k_{\phantom 0}^{-2})& \frac1{n_0}p(t)+\O (k_{\phantom 0}^{-2})&\O(k_{\phantom 0}^{-2})
\end{array}\label{e_Vmom2}\\
\end{equation}

\begin{proposition}[Single-step asymptotics for $\alpha_{\ell,\gl}$]\label{p_sstepa}
There exists $k^*=k^*(\beta,\bar\lambda)$ so that for every
$\ell\le n_0-k^*$ and $|\lambda|<\bar \lambda$ we have the
following.
\begin{eqnarray}
\ev_\ell \left[\Dal_{\ell,\gl}\right] \nonumber
 &=&\phantom{+}
\frac1{n_0} \Re\Big[(
e^{-i \varphi_{\ell,\gl}}-e^{-i \varphi_{\ell,0}})
\eta_{\ell}
 (-v_\gl-i q/2)
 \Big]\\
&&  \label{e_ea0}
 +
 \frac1{n_0} \Re\left[
 \frac{i q}{4}
(
e^{-2i \varphi_{\ell,\gl}}-e^{-2i \varphi_{\ell,0}})
\eta_{\ell}^2
 \right]+
 \O( n_0^{\nicefrac{-1}{2}} k_{\phantom 0}^{\nicefrac{-1}{2}} +k_{\phantom 0}^{\nicefrac{-3}{2}}\ha_{\ell,\gl})
 \\
 &=&
 \O( k_{\phantom 0}^{-1}  \ha_{\ell,\gl}+ n_0^{\nicefrac{-1}{2}}  k_{\phantom 0}^{\nicefrac{-1}{2}})
 \label{e_ea1}
 \\
 \ev_\ell \left[ (\Dal_{\ell,\gl})^2\right]
 &=&
 \O( k_{\phantom 0}^{-1}    \ha_{\ell,\gl}+ n_0^{-1} k_{\phantom 0}^{-1})
 \label{e_ea2}
 \\
 \ev_\ell | \Dfi_{\ell,\gl} \Dal_{\ell,\gl}|
 &=&
 \O(  \ha_{\ell,\gl} k_{\phantom 0}^{-1}).
 \label{e_extraalfa}
\end{eqnarray}
The functions $v_\gl=v_{\gl}(\ell/n_0)$, $q=q(\ell/n_0)$
are defined in (\ref{e_vbecs}, \ref{e_pqdef}), and
$\ha_{\ell,\gl}$ denotes the distance between
$\alpha_{\ell,\gl}$ and the set $2\pi \ZZ$.
\end{proposition}
\begin{proof}
By choosing a large enough $k^*\ge 1$ we can assume that for $\ell\le n-k^*$
\[
\frac{\bar \gl}{k^2}\le \frac1{10}, \qquad |v_{\ell,\gl}|\le \frac1{10}
\]
which together with (\ref{e_cutoff}) guarantees that the
random variable defined in (\ref{e_Zgl}) satisfies $
|Z_{\ell,\gl}|\le \nicefrac13$. The proof of the
proposition relies on the evolution rule, Proposition
\ref{p_alpha}
 \ref{alphaevo},
$$\Delta \alpha_{\ell,\gl}=
\ash(( \varupsl{L}_{ \ell,\gl})^{\varupsl{T}_\ell},-1,e^{i
\varphi_{\ell,\gl}} \bar
\eta_\ell)+\ash(\varupsl{S}_{\ell,0},e^{i
\varphi_{\ell,\gl}^*} \bar \eta_\ell,e^{i
\varphi_{\ell,\gl}} \bar
\eta_\ell)+\ash(\varupsl{S}_{\ell,0},e^{i
\varphi_{\ell,\gl}} \bar \eta_\ell,e^{i \varphi_{\ell,0}}
\bar \eta_\ell)
$$
whose terms we denote $\zeta_1,\zeta_2,\zeta_3$. First we show that
$\zeta_1,\zeta_2$ are small. By the definition \eqref{e_Ldef} of
$L$ we have
 $$\left|i.\left(( \varupsl{L}_{ \ell,\gl})^{\varupsl{T}_\ell}\right)^{-1}-i\right|= \left|\frac1{n_0}\,
\frac{\gl}{2 \hat s}\right|=\left|\frac{\gl}{\sqrt{k
n_0}}\right|\le \frac{1}{10}.
 $$
This estimate with the third bound of  Lemma \ref{l_ash}  gives
\begin{equation}\label{e_zizi}
\zeta_1=\varphi^*_{\ell,\gl}-\varphi_{\ell,\gl}=\O(n_0^{\nicefrac{-1}{2}}
k_{\phantom 0}^{\nicefrac{-1}{2}}).
\end{equation}
Applying again the third bound of Lemma \ref{l_ash} with
$|Z_{\ell,0}|\le\nicefrac13$ and (\ref{e_zizi}) gives
\begin{eqnarray*}
\zeta_2&=&\O(\varphi^*_{\ell,\gl}-\varphi_{\ell,\gl})=\O(n_0^{\nicefrac{-1}{2}}
k_{\phantom 0}^{\nicefrac{-1}{2}}).
\end{eqnarray*}
For $\zeta_3$ we  use the first estimate of Lemma \ref{l_ash} and note
that  in our case $|v-w|$ equals
\begin{equation}
|e^{i \varphi_{\ell,\gl}}-e^{i \varphi_{\ell,0}}|=|e^{i \alpha_{\ell,\gl}}-1|\le \ha_{\ell,\gl}.\label{e_triv1}
\end{equation}
Thus  with $Z=Z_{\ell,0}$ we have
\begin{equation}\label{e_ash_alal1}
\zeta_3=-\Re\left[ (e^{-i \varphi_{\ell,\gl}}-e^{-i \varphi_{\ell,0}})\eta_{\ell}(
Z+i Z^2/2)
 \right]+\Re\left[ i (e^{-2i \varphi_{\ell,\gl}}-e^{-2i \varphi_{\ell,0}})\eta_{\ell} Z^2/4\right]
 +\O(\ha_{\ell,\gl} Z^3).\end{equation}
Since $Z$ is independent of $\FF_\ell\,$ and
$\ha_{\ell,\gl}\in \FF_\ell\,$, the error term becomes $\O(
\ha_{\ell,\gl}k_{\phantom 0}^{\nicefrac{-3}{2}})$ after
taking conditional expectation. The definition \eqref{e_Zgl} of
$Z_{\ell,\lambda}$ and the moment bounds (\ref{e_Vmom2})
imply that replacing $\ev Z$ and $\ev Z^2$ by
$v_{\gl}(\ell/n_0)$ and $q(\ell/n_0)$ gives error terms of
order $\O(k_{\phantom 0}^{-2})$. Because of (\ref{e_triv1})
and
\begin{equation}
|e^{-2i \varphi_{\ell,\gl}}-e^{2i \varphi_{\ell,0}}|=|e^{2 i \alpha_{\ell,\gl}}-1|\le 2 \ha_{\ell,\gl}\label{e_triv2}
\end{equation}
we get (\ref{e_ea0}). Using the explicit form of $v_\gl$
and $q$ and (\ref{e_triv1}, \ref{e_triv2}) again, we obtain
(\ref{e_ea1}). The other estimates
follow similarly from the first-order
version of (\ref{e_ash_alal1}) and Proposition
\ref{p_sstepfi}.
\end{proof}

The following lemma relies on the careful use of single-step
bounds and oscillatory sum estimates. We postpone the proof
till Section \ref{a_tight}.

\begin{lemma}\label{l_tightestim1} Recall the definition of  $m_1, m_2$ from (\ref{e_middle}).
There exist $c_0,c_1$
depending on $\bar \lambda, \beta$ so that  with
$y=n_0^{\nicefrac{-1}{2}} ( \nn^{\nicefrac{1}{3}}\vee1)$ we
have
 \begin{eqnarray*}
|\ev [\alpha_{\ell_2,\gl}-\alpha_{\ell_1,\gl}|\mathcal
F_{\ell_1}]|&\le& c_1 (y+\sqrt{\eps})+\frac{\ev[
\ha_{\ell_2-1}|\mathcal
F_{\ell_1}]}{2}+\sum_{\ell=\ell_1}^{\ell_2-2} b_\ell
\ev[\ha_\ell|\mathcal F_{\ell_1}]
\\
0\;\le\; b_\ell &\le & c_1 \left(k_{\phantom
0}^{\nicefrac{-3}{2}}+ \nn k_{\phantom
0}^{\nicefrac{-5}{2}}+k_{\phantom 0}^{\nicefrac{-3}{2}}
 \nn \mathbf{1}_{k\ge  \nnn}\right),
\end{eqnarray*}
whenever $\kappa>c_0$, $|\lambda|<\bar \lambda$, and $m_1\le
\ell_1\le \ell_2\le m_2$. Here $k=n_0-\ell$.
\end{lemma}

The last ingredient needed for the proof of Proposition
\ref{l_middlestretch} is the following deterministic Gronwall-type
estimate.

\begin{lemma}[Gronwall estimate]\label{l_turbogw}
Suppose that for positive numbers $x_\ell, b_\ell, c$, integers
$\ell_1<\ell_2$  and $\ell=\ell_1+1,\ell_1+2,\ldots,\ell_2$ we
have
\begin{equation}\label{e_gr1}
x_{\ell}\le \frac{x_{\ell-1}}{2}+c+\sum_{j=\ell_1}^{\ell-1} b_j
x_j.
\end{equation}
Then
\[
x_{\ell_2}\le 2\,(x_{\ell_1}+c) \exp\left(3
\sum_{j=\ell_1}^{\ell_2-1} b_j\right).
\]
\end{lemma}
\begin{proof}
We can assume $\ell_1=0$. Let $y_\ell=x_{\ell}-x_{\ell-1}/2$, so
that we have
\begin{equation}\label{e xiny}
x_\ell = y_\ell+\frac{y_{\ell-1}}{2}+ \frac{y_{\ell-2}}{4}+\ldots+
\frac{y_{1}}{2^{\ell-1}}+\frac{x_0}{2^\ell}.
\end{equation}
Then \eqref{e_gr1} and the positivity of $x_0$ and $b_j$ gives
\begin{equation}\label{e yell}
y_\ell \le c+x_0s +
\sum_{j=1}^{\ell-1}b_j\left(y_\ell+\frac{y_{\ell-1}}{2}+
\frac{y_{\ell-2}}{4}+\ldots+ \frac{y_{1}}{2^{\ell-1}}\right),
\end{equation}
where $s=b_0+\ldots + b_{\ell_2-1}$. Taking positive parts in
\eqref{e yell}, and then summation by parts yields
\begin{equation}\label{e yellp}
y_\ell^+ \le ( c+x_0s) + \sum_{j=1}^{\ell-1}\tilde b_j y_j^+
\end{equation}
with $\tilde b_j=b_j+b_{j-1}/2+\ldots +
b_{\ell_2-1}/2^{\ell_2-j-1}$. Let $\ell_3$ be so that $1\le
\ell_3\le \ell_2$.
 We now multiply the inequality \eqref{e yellp} by
$\tilde b_\ell(1+\tilde b_{\ell+1})\cdots(1+\tilde
b_{{\ell_3}-1})$ and sum it for $1 \le \ell \le \ell_3-1$. We add
\eqref{e yellp} again with $\ell=\ell_3$. After cancellations, we
get
$$
y_{\ell_3}^+ \le (c+x_0s)\prod_{j=1}^{\ell_3-1} (1+\tilde b_j) \le
(c+x_0s) \exp\left(\sum_{j=1}^{\ell_2-1} \tilde b_j\right) \le
(c+x_0s) e^{2s}.
$$
Applying this inequality for all the $y$ terms in \eqref{e xiny}
with $\ell=\ell_2$  we get
\begin{equation*}
x_{\ell_2}\le 2(c+x_0s)e^2s+x_0\le
2(x_0+c)e^{3s}.\qedhere \end{equation*}
\end{proof}

\begin{proof}
[Proof of Proposition \ref{l_middlestretch}] For this proof, let
$a=\alpha_{\ell_1,\gl}$, and define
$a_\diamondsuit,a^\diamondsuit\in 2\pi \Z$ so that
$[a_\diamondsuit,a^\diamondsuit)$ is an interval of length $2\pi$
containing $a$. We condition on the $\sigma$-field $\mathcal
F_{\ell_1}$, so in this proof $\ev$ denotes the corresponding
conditional expectation. We also drop the index $\lambda$ from
$\alpha$. We will show that there exists $c_0$ so that if
$\kappa>c_0$, then with the quantifiers of the proposition
\begin{eqnarray}\label{e_tightness1}
\ev |\alpha_{\ell_2}-a_\diamondsuit |& \le& c_1 (
(a-a_\diamondsuit) + \sqrt{\eps}+y),
\\
\ev |\alpha_{\ell_2}-a^\diamondsuit| &\le& c_1 (
(a^\diamondsuit-a)+ \sqrt{\eps}+y)\label{e_tightness3}.
\end{eqnarray}
The claim of the proposition follows from this by an application
of the triangle inequality to the stronger bound. The additional
condition $\kappa>c_0$ is treated via the error term $1/\kappa$.
%

Lemma \ref{l_tightestim1}  provides the bound
 \begin{eqnarray*}
|\ev \alpha_\ell- a_\diamondsuit|&\le& (a-a_\diamondsuit)+c
(y+\sqrt{\eps})+\ev \ha_{\ell-1}/2+\sum_{j=\ell_1}^{\ell-2} b_j
\ev \ha_j.
 \end{eqnarray*}
Note that $\alpha$ never goes below an integer multiple of $2\pi$
that it passes (Proposition \ref{p_alpha} \ref{szelep1}), so $\alpha_{\ell}\ge
a_\diamondsuit$ for all $\ell\ge \ell_1$. This means that for
$\ell\ge \ell_1$ we have $\ha_\ell\le \alpha_\ell-a_\diamondsuit$
and with $x_\ell=\ev  |\alpha_\ell-a_\diamondsuit|$ we have the
bound
\begin{eqnarray}\label{Gprep}
x_\ell&\le& (a-a_\diamondsuit)+c
(y+\sqrt{\eps})+x_{\ell-1}/2+\sum_{j=\ell_1}^{\ell-2} b_j x_j.
\end{eqnarray}
According to  Lemma \ref{l_tightestim1} we can bound the
sum of the coefficients $b_\ell$ as
\[
\sum_{\ell=\ell_1}^{\ell_2-2} b_\ell\le c
(k_2^{\nicefrac{-1}{2}}+ \nn
k_2^{\nicefrac{-3}{2}}+1)\le c'
\]
which means that  (\ref{e_tightness1}) follows via the
Gronwall-type estimate of Lemma \ref{l_turbogw}.

Next, we consider the first time $T\ge \ell_1$ so that
$\alpha_{T}-a^\diamondsuit\ge 0$. Proposition \ref{p_alpha} \ref{alphaevo}
breaks one step of the evolution of  $\alpha$ into two
parts, from $\alpha_\ell$ to $\alpha_{\ell+1}^*$ and from
$\alpha_{\ell+1}^*$ to $\alpha_{\ell+1}$. It shows that
$\alpha$ can only pass an integer multiple of $2\pi$ in the
first part.  Since the first part is non-random, even the
time $T-1$ (and not just $T$) is a stopping time adapted to
our filtration. The overshoot can be estimated in two
steps. By \eqref{e_zizi}, and the fact that $k>c_0$ we have
\begin{equation}\label{e_tatu}
\ev \left[(\alpha_{T}^*-a^\diamondsuit) \one{(T\le \ell_2)}
\right] \le  c n_0^{\nicefrac{-1}{2}}.
\end{equation}
By the expected increment bound \eqref{e_ea1} and the
strong Markov property applied at $T-1$ we have
\begin{equation}\label{e_tatu2}
\ev \left[(\alpha_{T}-\alpha^*_{T}) \one{(T\le \ell_2)}
\right] \le  c n_0^{\nicefrac{-1}{2}}.
\end{equation}
This gives
\begin{eqnarray}
\ev (\alpha_{\ell_2}-a^\diamondsuit)^{+}
  \nonumber
\nonumber
 &=&
\ev \left[ \one{(T\le \ell_2)}\ev\left[\alpha_{\ell_2}-a^\diamondsuit\big | \mathcal F_T  \right]\right]\\
 \nonumber
 &\le&
 c_1 (\ev \left[(\alpha_{T}-a^\diamondsuit) \one{(T\le \ell_2)} \right]\ + \sqrt{\eps}+ y)\\
&\le&c_1' (\sqrt{\eps}+y), \label{e_tightness2}
\end{eqnarray}
where the first inequality uses (\ref{e_tightness1}) and the
strong Markov property, and the second uses (\ref{e_tatu},
\ref{e_tatu2}). To prove (\ref{e_tightness3}) first note that
Lemma \ref{l_tightestim1}  also gives
\[
|\ev \alpha_{\ell}-a^\diamondsuit | \le (a^\diamondsuit-a)+c
(y+\sqrt{\eps})+\ev \ha_{\ell-1}/2+\sum_{j=\ell_1}^{\ell-2} b_j
\ev \ha_j.
\]
Then by (\ref{e_tightness2}) and the identity $|a|=-a+2
a^+$ we get
\begin{eqnarray*}
\ev |\alpha_{\ell}-a^\diamondsuit |&\le& |\ev \alpha_{\ell}-a^\diamondsuit| +2 \ev (\alpha_{\ell}-a^\diamondsuit)^+\\
&\le& (a^\diamondsuit-a)+c (y+\sqrt{\eps})+\ev
\ha_{\ell-1}/2+\sum_{j=\ell_1}^{\ell-2} b_j \ev \ha_j.
\end{eqnarray*}
Since $\ha_{\ell}\le |\alpha_\ell-a^\diamondsuit|$, the inequality
\eqref{Gprep} follows with $x_\ell=\ev
|\alpha_\ell-a^\diamondsuit|$, and the Gronwall-type estimate in
Lemma \ref{l_turbogw} implies (\ref{e_tightness3}).
\end{proof}

\subsection{Bounds for oscillations in the middle stretch}\label{a_tight}

This section presents the proof of Lemma
\ref{l_tightestim1}, isolated as the most technical
ingredient of the proof in the previous section. We start
with a bound on the mixed differences.

\begin{lemma} \label{l_diszno}
There exists an absolute constant $c$ so that for $\ell\le n-k^*$
(with $k^*$ as in Proposition \ref{p_sstepa}) we have
 $$| \ev_\ell [\Delta e^{i \varphi_{\ell,\gl}}-\Delta e^{i \varphi_{\ell,0}}]|\le
ck_{\phantom 0}^{-1}
 \ha_{\ell}+c n_0^{\nicefrac{-1}{2}} k_{\phantom 0}^{\nicefrac{-1}{2}}$$
and the same inequality holds with $e^{2i\varphi}$ replacing $e^{i\varphi}$  on the
left-hand side.
\end{lemma}

\begin{proof} The left-hand side equals
\begin{eqnarray*}
&&\left| e^{i \varphi_{\ell,0}}\,\ev_\ell \left[(e^{i \alpha_{\ell,\gl}}-1)(e^{i \Delta \varphi_{\ell,\gl}}-1)+ (e^{i \Delta \alpha_{\ell,\gl}}-1)(e^{i \Delta \varphi_{ \ell,0}}-1)+(e^{i \Delta \alpha_{\ell,\gl}}-1)\right]\right|\\
&&\hskip15mm\le |e^{i \alpha_{\ell,\gl}}-1| |\ev_\ell [ e^{i
\Delta \varphi_{\ell,\gl}}-1]|+\ev_\ell |\Delta \alpha_{\ell,\gl}
\Delta \varphi_{ \ell,0}|+|\ev_\ell [e^{i \Delta
\alpha_{\ell,\gl}}-1]|.
\end{eqnarray*}
The statement now follows from (\ref{e_onestepfi}), Proposition
\ref{p_sstepa}, the bounds (\ref{e_triv1}, \ref{e_triv2}) and the bound
$$
|\ev [e^{i X}-1]|\le |\ev [e^{i X}-iX-1]|+|\ev( iX)|\le \ev |e^{i
X}-iX-1|+|\ev X|\le\ev X^2+|\ev X|.
$$
 The inequality
involving $e^{2i\varphi}$ can be proved the same way.
\end{proof}

Now we are ready to prove Lemma \ref{l_tightestim1}.

\begin{proof}[Proof of Lemma \ref{l_tightestim1}]  We will
drop $\gl$ in $\alpha_{\ell,\gl}$, and condition on the
$\sigma$-field $\mathcal F_{\ell_1}$. Let $\ev$ denote the conditional
expectation with respect to this $\sigma$-field and let $x_\ell=\ev \hat \alpha_\ell$.
We have
\begin{eqnarray}\label{alfadiff}
\left| \ev
\left[\alpha_{\ell_2}-\alpha_{\ell_1}\right]\right|&\le&
\Big| \sum_{\ell=\ell_1}^{\ell_2-1} \ev\left[ \ev \left(
\,\Delta \alpha_\ell \big| \mathcal F_\ell \right)\right]\Big|.
\end{eqnarray}
Let
\begin{eqnarray*}
g_{1,\ell}=\frac1{n_0}(-v_\gl-iq/2)\ev (
e^{-i \varphi_{\ell,\gl}}-e^{-i \varphi_{\ell,0}}), \quad g_{2,\ell}=\frac1{n_0}\,\frac{iq}{4}\ev(
e^{-2i \varphi_{\ell,\gl}}-e^{-2i \varphi_{\ell,0}}).
\end{eqnarray*}
By the single-step asymptotics  (\ref{e_ea0})
the right-hand side of \eqref{alfadiff} can be bounded
by
\begin{eqnarray}\nonumber
\big|\hspace{-.18em} \sum_{\ell=\ell_1}^{\ell_2-1}  \Re
(g_{1,\ell}\, \eta_{\ell})\big|+ \big|\hspace{-.18em}
\sum_{\ell=\ell_1}^{\ell^*} \Re (g_{2,\ell}\,
\eta_{\ell}^2)\big|+\big|\hspace{-.28em}
\sum_{\ell=\ell^*+1}^{\ell_2-1} \Re (g_{2,\ell}\,
\eta_{\ell}^2)\big| + c \sum_{\ell=\ell_1}^{\ell_2-1} k_{\phantom
0}^{\nicefrac{-3}{2}}x_\ell +c \sum_{\ell=\ell_1}^{\ell_2-1}
n_0^{\nicefrac{-1}{2}} k_{\phantom
0}^{\nicefrac{-1}{2}},\nonumber
 \end{eqnarray}
with the usual notation $k=n_0-\ell$.
We call the terms $\zeta_1$, $\zeta_2$, $\zeta_3$,
$\zeta_4$, $\zeta_5$.
Note that $\zeta_2$, $\zeta_3$ come from a single sum cut
in two parts at $\ell^*=n_0-\lfloor  \nnn \rfloor$, and one
part may be empty. Clearly, we have $\zeta_5\le c
\sqrt{\eps}$, and $\zeta_4$ is already in the desired form.
By (\ref{e_triv1}) and the bounds (\ref{e_vbecs},
\ref{e_pqdef}) on $q,v$ we have
\begin{eqnarray*}
|g_{1,\ell}|&\le& \frac{c}{n_0}|v_\gl+iq/2| x_\ell \le
\frac{c}{k}\,x_\ell.
\end{eqnarray*}
Lemma \ref{l_diszno} with $t_+=(\ell+1)/n_0$ gives
\begin{eqnarray*}
|\Delta g_{1,\ell}|&\le&
\frac{c}{n_0}\Big(\,|v_\gl(t_+)+\frac{iq(t_+)}{2}| \, |\ev
[\Delta e^{i \varphi_{\ell,\gl}}-\Delta e^{i \varphi_{\ell,0}} ]|+
\left(|\Delta_\ell v_\gl|+|\Delta_\ell q|\right) \ev|
e^{-i \varphi_{\ell,\gl}}-e^{-i \varphi_{\ell,0}}
|\Big)
\\
&\le& c k_{\phantom 0}^{-2}x_\ell +c n_0^{\nicefrac{-1}{2}}
k_{\phantom 0}^{\nicefrac{-3}{2}},
\end{eqnarray*}
where we used the notation $\Delta_\ell
f=f((\ell+1)/n_0)-f(\ell/n_0).$ The oscillatory sum Lemma
\ref{l_blackboxnew} gives
\begin{eqnarray*}
\zeta_1 &\le& c( \nn
(n_0-\ell_2)^{\nicefrac{-3}{2}}+(n_0-\ell_2)^{-1})
x_{\ell_2-1}+ c \sum_{\ell=\ell_1}^{\ell_2-2} (x_\ell
k_{\phantom 0}^{-2} +n_0^{\nicefrac{-1}{2}} k_{\phantom
0}^{\nicefrac{-3}{2}})( \nn k_{\phantom
0}^{\nicefrac{-1}{2}}+1)
\\
&\le& \frac{x_{\ell_2-1}}{6}+c ( \nn^{\nicefrac{1}{3}}\vee1)
n_0^{\nicefrac{-1}{2}} + c \sum_{\ell=\ell_1}^{\ell_2-2}
x_\ell (k_{\phantom 0}^{-2}+  \nn
k_{\phantom 0}^{\nicefrac{-5}{2}}),
\end{eqnarray*}
where the coefficient $\nicefrac16$ is achieved by choosing a
large enough $c_0$. We continue
\begin{eqnarray*}
\zeta_2\;\le   \;\sum_{\ell=\ell_1}^{\ell^*} |
g_{2,\ell}|\;\le\; \frac{c}{n_0} \,
|q|\sum_{\ell=\ell_1}^{\ell^*} x_\ell
 \;\le\; c'  \sum_{\ell=\ell_1}^{\ell^*} x_\ell\,  \nn
 k_{\phantom 0}^{\nicefrac{-3}{2}}.
\end{eqnarray*}
The term $\zeta_3$ is handled by Lemma \ref{l_blackboxnew}
with $g_j=g_{2,j}$. Standard bounds on  $q, q'$ and Lemma
\ref{l_diszno} give
\begin{eqnarray*}
|g_\ell|\le c k_{\phantom 0}^{-1}x_\ell,\quad
|g_\ell-g_{\ell+1}|\le c k_{\phantom 0}^{-2}x_\ell +c
n_0^{\nicefrac{-1}{2}} k_{\phantom 0}^{\nicefrac{-3}{2}},
\end{eqnarray*}
hence from Lemma \ref{l_blackboxnew} we get
\begin{eqnarray*}
\zeta_3&\le& c( \nn
(n_0-\ell_1)^{\nicefrac{-1}{2}}+1)k_{\phantom 0}^{-1}
x_{\ell_2-1} + c \sum_{\ell=\ell^*+1}^{\ell_2-2} (x_\ell
k_{\phantom 0}^{-2} +n_0^{\nicefrac{-1}{2}} k_{\phantom
0}^{\nicefrac{-3}{2}})( \nn k_{\phantom
0}^{\nicefrac{-1}{2}}+1)
\\
&\le& \frac{x_{\ell_2-1}}{6}+c ( \nn^{\nicefrac{1}{3}}\vee1)
n_0^{\nicefrac{-1}{2}} + c \sum_{\ell=\ell^*+1}^{\ell_2-2}
x_\ell (k_{\phantom 0}^{-2}+  \nn
k_{\phantom 0}^{\nicefrac{-5}{2}})
\end{eqnarray*}
if $c_0$ is chosen sufficiently large. The claim follows.
\end{proof}

\subsection{Why does the right boundary condition disappear?}
\label{s_unif}

The goal of this section is to show that the phase evolution picks
up sufficient randomness that will neutralize the right boundary
condition of the discrete equations.

\begin{proposition}\label{p_unif}
Let $m=\lfloor n- \nnn- \kappa (\nn^{\nicefrac{2}{3}}\vee 1)\rfloor$
and suppose that $\kappa\to \infty$ with $n_0^{-1} \kappa
(\nn^{\nicefrac{2}{3}}\vee 1)\to 0$.  Then $\varphi_{m,0}$ modulo
$2\pi$ converges in distribution to Uniform$(0,2\pi)$.
\end{proposition}

\begin{proof} We will show that given $\eps>0$, every subsequence of indices has a
 further subsequence along which $\varphi_{m,0}$ modulo $2\pi$ is
eventually $\eps$-close to uniform distribution. So we first pick
an integer $\tau=\tau(\eps)$ and show that along a suitable
subsequence, the conditional distribution given $\mathcal
F_{m-\tau\xi}$ of $\varphi_{m,0}-\varphi_{m-\tau \xi,0}$ converges
to Gaussian with variance tending to $\infty$ with $\tau$. Here
the scaling factor is $\xi=\lfloor \kappa
(\nn^{\nicefrac{2}{3}}\vee 1)\rfloor$. Since a constant plus a
Gaussian with large variance is close to uniform modulo $2\pi$,
the claim follows if we let $\tau$ go to $\infty$.

To show the distributional convergence, we apply the SDE limit
Theorem \ref{thm_sdelim} to the evolution of $\varphi$ from time
$m-\tau\xi$ on. To adapt to the setup of the theorem we introduce
the new parameters $$\breve n=n-m+\tau \xi, \qquad
 \mu_{\breve n}= \nn,\qquad \breve n_0=\breve n- \mu^2_{\breve n}/4-1/2,\qquad
\breve \varphi_{\ell,\gl}=\varphi_{\ell+m-\tau \xi,\gl} .$$

By assumption, we have $\breve n_0 \mu_{\breve
 n}^{\nicefrac{-2}{3}}\to \infty$. We pass to a subsequence
so that  $\breve n_0/\breve n$ has a limit
$1/(1+\nu)\in[0,\infty]$, so Theorem \ref{thm_sdelim} (trivially
modified to allow general initial conditions) applies. The result
is that $\breve\varphi_{\lfloor t \breve n_0 \rfloor, 0 }$ has an
SDE limit given by \eqref{sdelim2} with $\lambda=0$. Thus
$\varphi_{m,0}-\varphi_{m-\tau \xi,0}$ converges to a normal
random variable which does not depend on the initial value $\breve
\varphi(0)$. Its variance is given by integrating the  sum of the
squares of the  independent diffusion coefficients on
the corresponding scaled time interval:
\begin{eqnarray*}
\int_0^{1-(1+\tau)^{-1}}  \frac{6\nu+2-2t}{\beta (1-t)(\nu+1-t)}
dt\ge\frac2\beta \log (\tau+1),
\end{eqnarray*}
which goes to $\infty$ with $\tau$, as required.
\end{proof}

\subsection{The uneventful ending}\label{s_end}

This section is about the last part of the recursion, from
$$m_2=\lfloor n- \nnn- \kappa (\nn^{\nicefrac{2}{3}}\vee 1)
\rfloor$$  to $n$ where $\kappa>0$ is a constant. The goal is to
show that nothing interesting happens on this stretch. More
precisely, we show
\begin{lemma}\label{l_short}
For every  $\gl\in \RR$ and $\kappa>0$ as  $n\to \infty$ we have
$\tph_{m_2,\gl}-\tph_{m_2,0}\to 0$ in probability.
\end{lemma}
Fix $\kappa$ and $\gl$. We will show the convergence by showing
that any subsequence has a further subsequence with the
desired limit. Because of this, we may assume that the limit of
$ \nn$ exists. We will consider two cases: $\lim  \nn<\infty$
and $\lim  \nn=\infty.$

\begin{proof}[Proof of Lemma \ref{l_short} in the case when $\lim  \nn$ is
finite.] \ \\ In this case we can assume that $n-m_2$ is eventually equal to some
integer $\xi$. Also, $\rho_{m_2}$ converges to a unit complex
number $\rho$ with $\Im \rho>0$. By (\ref{e_1stepfi})
 we have
\begin{equation}\label{e_prod}
\tph_{n-\xi,\gl}\lstar \varupsl{Q}_{n-\xi-1}^{-1}=0\lstar
\varupsl{R}^{-1}_{n-1}\cdots \varupsl{R}^{-1}_{n-\xi},
\end{equation}
where
$
\varupsl{R}_{n-j,\gl}^{-1}=\varupsl{W}_{n-j}^{-1}
 \tilde{\varupsl{L}}_{n-j,\gl}^{-1} \varupsl{Q}(\pi)^{-1}.
$
%
Consider the components of the product on the right-hand
side of (\ref{e_prod}). The elements $\tilde{\varupsl{L}}_{n-j,\gl}$ are
deterministic (see (\ref{e_evoevo})) and
as functions on $\RR'$ -- the lifted unit
circle -- they converge uniformly to non-degenerate limits
that do not depend on $\lambda$. (Here we also used  $s_{n-j}=\sqrt{j-\nicefrac12}$.)
In the same sense, we also
have
 $
 \varupsl{T}_{n-\xi}\to \varupsl{A}(\Im (\rho)^{-1},-\Re \rho)$.
 Because of the  moment bounds (\ref{e_mom}) we may find a subsequence along which $X_{n-1},  \dots, X_{n-\xi}$ and $Y_{n-1}, \dots, Y_{n-\xi}$ all converge. Then (using the definition (\ref{d_QW})) it follows that the random elements $\varupsl{W}_{n-j}$ converge as functions for $j=1,\ldots,\xi$.

Since all of these limits are non-degenerate and the dependence on
$\lambda$ disappears, we have
\begin{equation*}
\left|\tph_{n-\xi,\gl} -\tph_{n-\xi,0}
\right|=\left|\tph_{n-\xi,\gl}\lstar
\varupsl{Q}_{n-\xi-1}^{-1}-\tph_{n-\xi,0}\lstar \varupsl{Q}_{n-\xi-1}^{-1}\right|\to
0.\qedhere
\end{equation*}
\end{proof}

\begin{remark}\label{r_RRV}
For the second case, we review some of the results of \RRVlong,
henceforth denoted RRV, about the eigenvalues of the stochastic
Airy operator. The paper considers
 the   eigenvalue process $\Lambda_n$ of the random matrix $M$ (see
(\ref{e_originalmatrix})) under the edge scaling
$n^{\nicefrac{1}{6}}(\Lambda_n-2 \sqrt{n}).$ By Theorem 1.1 of
RRV, the limit is a point process $\Xi$ given by the eigenvalues
of the  {\bf stochastic Airy operator}, the random Schr\"odinger
operator
$$\mathcal H_\beta=-\frac{d^2}{dx^2}+x+\frac{2}{\sqrt{\beta}} b'_x$$ on the
positive half-line. Here $b'$ is white noise and the initial
condition for the eigenfunction $f$ is $f(0)=0, f'(0)=1$.  By RRV,
Proposition 3.5 and the discussion preceding RRV, Proposition
3.7,
\begin{equation}\label{simplepp}
  \mbox{$\Xi$ is a.s.\ simple, and for every $x\in \R,$ we have }
  \pr(x\in \Xi )=0.
\end{equation}
The proof is based on the observation that after appropriate
rescaling the matrix $M$ acts on vectors as a discrete
approximation of $\mathcal H_\beta$. The initial condition
$f(0)=0, f'(0)=1$ comes from the fact that the discrete eigenvalue
equation for an eigenvalue $\Lambda=2\sqrt{n}+n^{-1/6} \nu$ is
equivalent to a three-term recursion for the vector entries
$w_{\ell,\nu}$ (c.f. (\ref{e_vrec}) and Remark \ref{r_oldmatrix})
with the initial condition  $w_{0,\nu}=0$ and $w_{1,\nu}\neq 0$.

By RRV, Remark 3.8, the results of RRV extend to solutions of the
same three-term recursion with  more general initial conditions.
We say that a value of $\nu$ is an eigenvalue for a family of
recursions parameterized by $\nu$ if the corresponding recursion
reaches $0$ in its last step.
%
Suppose that for given $\zeta\in
[-\infty,\infty]$ the initial condition for the three-term recursion equation
satisfies
\begin{equation}\nonumber
\frac{w_{0,\nu}}{n^{\nicefrac{1}{3}}
(w_{1,\nu}-w_{0,\nu})}=n^{\nicefrac{-1}{3}} (p_n-1)^{-1} \cp \zeta,
\end{equation}
where $p_n:=w_{1,\nu}/w_{0,\nu}$  does not depend on $\nu$. Here
the factor $n^{\nicefrac{1}{3}}$ is the spatial scaling for the
problem (RRV, Section 5).   Then the eigenvalues of this family of recursions
converge to those of the stochastic Airy operator with initial
condition $f(0)/f'(0)=\zeta$. The corresponding point process
$\Xi_\zeta$ will also satisfy \eqref{simplepp}, see RRV, Remark
3.8.
\end{remark}

Now we are ready to complete the proof of Lemma \ref{l_short}.

\begin{proof}[Proof of Lemma \ref{l_short} in the case when $\lim
 \nn=\infty$.]\ \\
Without loss of generality we assume $\gl>0$.
Fix a  $\theta\in\R'$ and let $B$ denote the event that
\begin{equation}\label{arcs}
\textup{$x\lstar \varupsl{Q}_{m_2-1}\not\equiv \theta $ mod $2\pi$ for $x\in
[\tph_{m_2,\gl},\tph_{m_2,0}]$}.
\end{equation}
%
%
 It suffices to show that
$\pr(B)\to 1$.  Indeed, by considering a subdivision of the unit
circle into arcs of length at most $ \eps$ at points
$e^{i\theta_j}$,  if the event  \eqref{arcs} holds for each
$\theta_j$ then
 $$\left|\tph_{m_2,\gl}\lstar \varupsl{Q}_{m_2-1}^{-1}-\tph_{m_2,0}\lstar \varupsl{Q}_{m_2-1}^{-1}\right|=\left|\tph_{m_2,\gl}-\tph_{m_2,0}\right|$$
cannot be greater than $\eps$. Taking $\eps\to 0$
completes the proof.

Equation  (\ref{e_target}) translates $B$ to an event about
$\tphh_{\ell,\gl}$. More specifically, by Proposition \ref{p_wild}
$B$ is the event that the one-parameter family of recursions
parameterized by $\nu$
$$
\hat \varphi_{\ell+1,\nu}=\hat \varphi_{\ell,\nu}\lstar \varupsl{R}_{\ell,\nu}, \qquad \ell\ge m_2
$$
with initial condition
\begin{equation}\label{icond}
\hat \varphi_{m_2,\nu}=\theta \lstar \varupsl{T}_{m_2}^{-1}
\end{equation}
does not have an eigenvalue in the interval
 $[0, \gl]$.
 This recursion is determined by the bottom right $n_2\times n_2$
submatrix of $M(n)^{D(n)}$ \eqref{e_matrix}, where $n_2=n-m_2$.
Thus the recursion is in fact the discrete eigenvalue equation for
$M(n_2)^{D(n_2)}$ with a generalized initial condition. This can
be transformed back to the discrete eigenvalue equation for
$M(n_2)$ with the corresponding initial condition. Let
$u=\varupsl{U}^{-1}(e^{i \theta})\in \R$ be the point corresponding to
$\theta\in \RR'$. Then \eqref{icond} translates to the initial
condition
  \be r_{m_2,\nu}=u.\varupsl{T}_{m_2}^{-1}=\Im
(\rho_{m_2}) u+\Re(\rho_{m_2}),
\nonumber
 \ee for the
eigenvalue equation of $M^D$ (see \eqref{e_rrec}) and by
Remark \ref{r_oldmatrix} to the initial condition
\be p_{m_2,\nu}=r_{m_2,\nu}\,
\frac{\chi_{(n-m_2-1)\beta}}{\sqrt{\beta}\,
s_{m_2+1}}=\frac{\chi_{(n-m_2-1)\beta}}{\sqrt{\beta
(n-m_2-\nicefrac12)} }\, (\Im (\rho_{m_2})
u+\Re(\rho_{m_2})) \label{e_init1} \ee for the eigenvalue
equation of $M$. As $ \nn\to \infty$, we have
\begin{equation}\label{e_asym}
n_2\to\infty, \quad  \nnn=n_2- \kappa
n_2^{\nicefrac{1}{3}}+o(n_2^{\nicefrac{1}{3}}),\quad \textup{and}
\quad \rho_{m_2}=1+i \sqrt{\kappa}
n_2^{\nicefrac{-1}{3}}+o(n_2^{\nicefrac{-1}{3}}).\end{equation}
Since $\ell^{\nicefrac{-1}{2}} \,\chi_{\ell}$ converges to $1$ in probability as $\ell \to \infty$,  (\ref{e_init1}) and \eqref{e_asym} imply
\[
n_2^{\nicefrac{-1}{3}} \left(p_{m_2,
\nu
}-1\right)^{-1}\cp \kappa^{\nicefrac{-1}{2}}
u^{-1}=:\zeta.
\]
This means that the limit of $\pr(B)$ can be related to the
limit point process $\Xi_\zeta$. The interval $[0,\gl]$
corresponds to $2  \nn+[0, \gl
n_0^{\nicefrac{-1}{2}}/2]$ in our scaling
(\ref{e_scalingnu}). In the edge scaling corresponding to
$n_2$, the length of the remaining stretch, this turns into
the interval
\begin{equation}
n_2^{\nicefrac{1}{6}} (2  \nn-2
n_2^{\nicefrac{1}{2}})+[0, \gl n_0^{\nicefrac{-1}{2}}/2
]\to -\kappa+[0,0] \nonumber
\end{equation}
where the convergence follows from \eqref{e_asym}.

For $\delta>0$ let $B_\delta$ be the event that the discrete
eigenvalue equation for $M(n_2)$ with initial condition
\eqref{e_init1} does not have an eigenvalue in the interval $$2
n_2^{\nicefrac{1}{2}}+n_2^{\nicefrac{-1}{6}}\,
(-\kappa-\delta,-\kappa+\delta).$$ By  Remark \ref{r_RRV}, for any
fixed $\delta$ we have
$$\limsup_{n\to \infty} \pr(B)\le \limsup_{n\to \infty}
\pr(B_\delta)\le \pr(\Xi_\zeta \textup{ doesn't have a
point in } [\kappa-\delta,\kappa+\delta]).$$ Since this
holds for all $\delta>0$, the fact  \eqref{simplepp} gives
$\lim\pr(B)=1$, as required.
\end{proof}

\appendix

\section{Tools}

\subsection{Angular shift bounds}\label{a_ash}

The objective of this section is to prove Lemma \ref{l_ash}, which
relies on Fact \ref{f_ashid}.

\begin{proof}[Proof of Fact \ref{f_ashid}]
The general form of such a transformation is given by $w\lcirc
T=e^{i\alpha}(w-\sigma)/(1-\bar\sigma w)$, where $\sigma$ is the
pre-image of $0$. We may assume $\alpha=0$ since post-composing
$T$ with a rotation does not change the quantities in question.
Using the definition of $\ash(\varupsl{T},v,w)$ and the fact that
$|w|=|v|=1$ we have
$$
\ash(\varupsl{T},v,w)=\Arg_{[0,2\pi)}\left(\frac{w-\sigma}{v-\sigma}
\frac{\bar v-\bar \sigma}{\bar w-\bar \sigma} \frac{v}{w}
\right)-\Arg_{[0,2\pi)}(w/v).
$$
The additivity of $\Arg$ mod $2\pi$ proves \eqref{f_ashidentity}
mod $2\pi$. By definition, $\ash$ is continuous in $T$ and so also in
$\sigma$. Since $|\sigma|<1$, the right-hand side of (\ref{f_ashidentity}) is continuous in
$\sigma$. As equality holds for $\sigma=0$, the proof is complete.
\end{proof}

\begin{proof}[Proof of Lemma \ref{l_ash}]
Recall that $r.\varupsl{U}=(i-r)/(i+r)$ maps the upper half plane to the
unit disk, sending $i$ to $0$. By Fact \ref{f_ashid} we have
$$
 \ash(\varupsl{T},v,w)
 =2 \Arg\left(\frac{1-\left((i+z).\varupsl{U}\right) \bar w}{1-\left((i+z).\varupsl{U}\right) \bar v}   \right)
 = 2\Arg \left(1+x\right), \qquad
 x=\frac{z(\bar w - \bar v)}{2i+z(1+\bar v)}.
$$
If  $|z|\le \nicefrac13$ then we have $|x|\le \nicefrac12$
so we can write $ \ash(\varupsl{T},v,w)=\Re h_{v,w}(z)$ with
$$
h_{v,w}(z)=\frac{2}{i}\log \left(1+\frac{z(\bar w - \bar
v)}{2i+z(1+\bar v)}\right)=(\bar w -\bar v)\left(-z-\frac{i(2+\bar
v + \bar w)}{4}z^2+\eta_{v,w}(z) \right).
$$
Here we use the standard branch of the logarithm defined outside
the negative real axis. The second equality is Taylor expansion in
$z$. To bound the error term, we write
$$
h_{v,w}'''(z)= \frac{(\bar w-\bar v)\,p(z,\bar v,\bar
w)}{(2i+z(1+\bar v))^3(2i+z(1+\bar w))^3},
$$
where $p$ is some (explicitly computable) polynomial, so the Taylor error term satisfies
$$|\eta_{v,w}(z)|\le \frac{|z^3|}{3!}\sup_{|z|\le1/3, |v|=|w|=1}
\frac{|h_{v,w}'''(z)|}{|w-v|}<c|z|^3.
$$
This proves the quadratic approximation of the angular
shift for $|z|\le\nicefrac13$, and the other two estimates
of (\ref{e_ash}) follow easily.

For the case $|z|>\nicefrac13, v=-1$ we use the fact that
$|\Arg(1+x)| \le \pi |x|$ to conclude that
$$
 |\ash(\varupsl{T},v,w)|\le
 4\pi  \left|\frac{z(\bar w - \bar v)}{i+z(1+\bar v)/2}\right|
= 4\pi  \left|z(\bar w - \bar v) \right|\le 4\pi \,3^{d-1}   \left|z^d (\bar w - \bar v) \right|
$$
for any $d\ge1$. Using $|z|>\nicefrac13$ we get that the main terms on the right-hand side of  (\ref{e_ash}) may also be bounded by $c_d |w-v| |z|^d$ and from this we get  upper bounds of (\ref{e_asherr})
 as well.
\end{proof}

\subsection{Oscillatory sums}\label{a_osc}

Recall from the definition (\ref{d_eta}) that $\eta_\ell$ is a unit complex number with a rapidly oscillating angle. Lemma \ref{l_blackboxnew} below will show that this oscillation has an averaging effect in sums. In order to prove that we first need the following  harmonic
analysis lemma.

\begin{lemma}\label{l_vdcorp}
Suppose that $2\pi>\theta_1>\theta_2>\ldots>\theta_m>0$ and
let $s_\ell=\sum_{j=1}^\ell \theta_j$. Then
\[
\max_{1\le \ell \le m} |\sum_{j=1}^\ell e^{i s_j}|\le
c(\theta_m^{-1}+(2\pi-\theta_1)^{-1}).
\]
\end{lemma}

\begin{proof}
We first consider the case when $\theta_1\le \pi$. Using second order interpolation 
we can construct a differentiable
function $s(x)$ on $[1,m]$ with $s(\ell)=s_\ell$ for $1\le
\ell \le m$ for which the derivative $s'(x)$ is monotone
decreasing derivative with $-\pi \le s'(x) \le -\theta_1/2$.

Our proof is based on the following lemmas of van der Corput (see
\cite{Hille} for the first and  \cite{Stein}, Chapter VIII,
Proposition 2  for the second):
\begin{enumerate}
\item If $s(x)$ has a monotone derivative with
$|s'(x)|\le \pi$ for $x\in[a,b]$ (with $a,b\in \Z$)
then the difference of $\sum_{\ell=a}^{b} e^{i s(l)}$ and  $\int_{a}^{b} e^{i\,s(x)}
dx$ is at most 3.

\item If $s'(x)$ is monotone and
$|s'(x)|>p$ on an interval $[a,b]$ then
$
|\int_a^{b} e^{i s(x)} dx|\le 3 p^{-1}.
$
\end{enumerate}
 Since for our function $\pi>|s'(x)|>\theta_m/2$ for $x\in [1,m]$ we may apply these lemmas to get the bound
$|\sum_{j=1}^\ell e^{i s_j}|\le c\, \theta_m^{-1}$.

Consider now the case  $2\pi>\theta_1>\pi$. Let $\ell^*$ be the
largest index with $\theta_{\ell^*}>\pi$, then
\be
\Big|\sum_{j=1}^\ell \exp\Big[{i \sum_{u=1}^j
\theta_u}\Big]\Big|\le \Big|\sum_{j=1}^{\ell\wedge \ell^*}
\exp\Big[{i \sum_{u=1}^j
\theta_u}\Big]\Big|+\Big|\sum_{j=\ell^*+1}^\ell \exp\Big[{i
\sum_{u=\ell^*+1}^j \theta_u}\Big]\Big|\label{e_osc}. \ee The
second sum  can be bounded by $c \,\theta_m^{-1}$ using the first
half of our proof. To bound the first sum, note that
$$
\Big|\sum_{j=1}^{\ell} \exp\Big[{i \sum_{u=1}^j
\theta_u}\Big]\Big|=\Big|\sum_{j=1}^{\ell} \exp\Big[{i \sum_{u=1}^j
\tilde \theta_{u}}\Big]\Big|, \qquad \tilde\theta_u=2\pi-\theta_{\ell+1-u}
$$
and for $\ell\le \ell^*$ we have
$$\pi>2\pi-\theta_{\ell^*}\ge \tilde \theta_1>\tilde \theta_2>\ldots>\tilde \theta_\ell\ge 2\pi-\theta_1>0.$$
Thus  the first half of the proof can be
applied again to get the bound $c (2\pi-\theta_1)^{-1}$.
\end{proof}

The following lemma describes the averaging effects of the oscillating  unit complex numbers  $\eta_\ell$.

\begin{lemma}\label{l_blackboxnew}
Let $g_\ell\in \CC$ for $\ell\in \N$ and $\ell_0<\ell_1\le n_0$. Then
\begin{eqnarray*}
\big|\Re\sum_{\ell=\ell_0}^{\ell_1} g_\ell \eta_\ell
\big|&\!\le\!& c \left( \nn
k_1^{-\nicefrac{1}{2}}+1\right) \big| g_{\ell_1}\big|+c
\sum_{\ell=\ell_0}^{\ell_1-1} \left( \nn
k_{\phantom{1}}^{-\nicefrac{1}{2}}+1\right)| g_{\ell+1}- g_\ell|
\\
\big|\Re\sum_{\ell=\ell_0}^{\ell_1} g_\ell \eta_j^2 \big|&\!\le\!&
c \left( \nn
k_1^{-\nicefrac{1}{2}}+\nn^{-1}
k_0^{\nicefrac{1}{2}}\right) \left| g_{\ell_1}\right|+c
\sum_{\ell=\ell_0}^{\ell_1-1} \left( \nn
k_{\phantom{1}}^{-\nicefrac{1}{2}}+\nn^{-1}
k_0^{\nicefrac{1}{2}}\right)| g_{\ell+1}- g_\ell|
\end{eqnarray*}
(We used the shorthanded notation $k=n_0-\ell$,
$k_1=n_0-\ell_1$ and $k_0=n_0-\ell_0$.)
\end{lemma}
\begin{proof}
For $d=1,2$ we introduce $F_{d,j}=\sum_{m=\ell_0}^j \eta^d_m$ with
$F_{d,\ell_0-1}=0$. By partial summation
\be\label{e_abel}
\sum_{j=\ell_0}^{\ell_1}  g_j \eta_j=F_{d,\ell_1}
g_{\ell_1}+\sum_{j=\ell_0}^{\ell_1-1} F_{d,j} ( g_j- g_{j+1}).
\ee
From (\ref{e_rhorho}) we get the estimates
$$\Arg{\rho_{\ell}}\le \nn^{-1} k_{\phantom{1}}^{\nicefrac{1}{2}},
\quad\textup{
and}\quad \pi/2-\Arg \rho_{\ell}\le   \nn k_{\phantom{1}}^{-\nicefrac{1}{2}}.$$
Together with  (\ref{d_eta}) this means that we can use
Lemma \ref{l_vdcorp} to get
$$
\big |F_{1,\ell} \big|\le c \left( \nn
k_1^{-\nicefrac{1}{2}}+1\right) \qquad \big| F_{2,\ell}\big|\le c
\left( \nn
k_1^{-\nicefrac{1}{2}}+\nn^{-1}
k_0^{\nicefrac{1}{2}}\right). $$ This with (\ref{e_abel}) implies
the lemma.
\end{proof}

\subsection{A limit theorem for random difference equations }\label{a_sdetools}

\begin{proof}[Proof of Proposition \ref{p_turboEK}]
Let $\|\cdot \|_\infty$ denote supremum norm on $[0,T]$.
For a two-parameter function $f$ and $x\in \RR$ let  $\IO$
denote the integral $\IO_{f,x}(t)= \int_0^t f(s,x)\,ds$. We
recycle this notation for a function $X:[0,T]\to \RR$ to
write  $\IO_{f,X}(t)= \int_0^t f(s,X(s))\,ds$.
%

The proof of this proposition is based on  Theorem 7.4.1 of
\cite{EthierKurtz}, as well as Corollary 7.4.2 and its proof. (See
also \cite{SV}.) It states that  if the limiting SDE has unique
distribution (i.e. the martingale problem is well-posed) and also
\begin{eqnarray}\label{e kuka}
\|\IO_{b^n,X^n}-\IO_{b,X^n}\|_\infty &\lip& 0,\\
\|\IO_{a^n,X^n}-\IO_{a,X^n}\|_\infty &\lip& 0,\nonumber
\end{eqnarray}
\begin{equation}
\textup{for every $\eps>0$}\qquad    \sup_{x,\ell}
\pr(|Y^n_\ell(x)|\ge \eps)\longrightarrow 0\label{e kuka3},
\end{equation}
then $X^n\cd X$. The theorem there only deals with the case
of time-independent coefficients, but adding time as an
extra coordinate extends  the results to the general case.

Because of our assumptions on $a$ and $b$ the well-posedness of
the martingale problem follows from Theorem 5.3.7 of
\cite{EthierKurtz}  (see especially the remarks following the
proof), and even pathwise uniqueness holds. Condititon  (\ref{e
kuka3}) follows  from the uniform third absolute
moment bounds (\ref{e 3m}) and Markov's inequality. Thus we only need to show (\ref{e
kuka}) as well as the analogous statement for $a$, for which the
proof is identical. We do this by bounding the successive
uniform-norm distances between
$$
\IO_{b^n,X^n},\quad \IO_{b^n,X^{n,L}}, \quad
\IO_{b,X^{n,L}},\quad \IO_{b,X^n},
$$
where $X^{n,L}_\ell= X^n_{K\lfloor \ell/K\rfloor} $ with $K=\lceil
nT/L \rceil$, and $X^{n,L}(t)=X^{n,L}_{\lfloor nt \rfloor}$. In
words, we divide $[0,\lfloor nT \rfloor]$ into $L$ roughly equal
intervals and then set $X_\ell^{n,L}$ to be constant on each
interval and equal to the first value of $X^{n}_\ell$ occurring
there.

If a function $f$ takes countably many values $f_i$, then
for any $h$ we have
$$
\|\IO_{h,f}\|_\infty\le \sum_{i} \|\IO_{h,f_i}\|_\infty\
$$
Since $X^{n,L}$ takes at most $L$ values, we have
$$
 \|\IO_{b^n,X^{n,L}}-\IO_{b,X^{n,L}}\|_\infty=\|\IO_{b^n-b,X^{n,L}}\|_\infty\le
 L \sup_x \|\IO_{b^n-b,x}\|_\infty=Lo(1)
$$
by \eqref{e fia} where $o(1)$ is uniform in $L$ and refers to
$n\to \infty$.
From (\ref{e lip}), the other terms
satisfy
\begin{eqnarray*}
\|\IO_{b^n,X^{n,L}}-\IO_{b^n,X^{n}}\|_\infty&\le&
T\|b^n(\cdot,X^{n,L}(\cdot))-b^n(\cdot,X^{n}(\cdot))\|_\infty \\
&\le& c T \|X^n-X^{n,L}\|_\infty + o(1)
\end{eqnarray*}
The same  holds with $b$ replacing $b^n$. It now
suffices to show that
\begin{equation}
\ev \|X^{n,L}-X^n\|_\infty = \ev \sup_\ell
|X^{n,L}_\ell-X^n_\ell|\le f(L)\label{e_L1}
\end{equation}
uniformly in $n$ where $f(L)\to 0$ as $L\to\infty$. The
left-hand side of (\ref{e_L1}) is bounded by
$$
\ev \sup_\ell |X^n_\ell-\frac1n\sum_{k={\lfloor
\ell/K\rfloor K}}^{\ell-1} b_n(X^n_\ell)-X^{n,L}_\ell|+ \ev
\sup_\ell |\frac1n\sum_{k={\lfloor \ell/K\rfloor K}}^\ell
b(X_\ell)|
$$
and the second quantity is bounded by $T\sup_{\ell,x}
|b^n_\ell(x)|/L$. The first quantity can be written as $\ev
M^*$ where
$$
M^*=\max_{i=0,\ldots,L-1} M_i^*, \qquad
M_i^*=\max_{\ell=0,\ldots, K-1} |M_{i,\ell}|,\qquad
 M_{i,\ell}=
X_{iK+\ell}-X_{iK}- \frac1n\sum_{k=0}^{\ell-1}
b^n(X_{iK+k}).
$$
Note that for each $i$, $M_{i,\ell}$ is a martingale. For any martingale
with $M_0=0$ we have
$$
\ev \max_{k\le n}|M_k|^3 \le c \ev \Big|
\sum_{k\le n} \ev[(M_k-M_{k-1})^2|\mathcal F_{k-1}]
\Big| ^{3/2}
\le c n^{3/2} \max_{k\le
n} \ev[|M_k-M_{k-1}|^3|\mathcal F_{k-1}].
$$
The first step is the  Burkholder-Davis-Gundy inequality   (see
\cite{Kallenberg}, Theorem 26.12) and the second step follows from
Jensen's inequality. Therefore \eqref{e 3m} implies
$$
\ev [|M_i^*|^3| \mathcal F_{iL}] \le c (n/L)^{3/2}
n^{\nicefrac{-3}{2}} =c L^{\nicefrac{-3}{2}},
$$
which gives the desired conclusion
$$
(\ev M^*)^3\le \ev (M^*)^3\le \ev \sum_{i=0}^{L-1}
(M_i^*)^3 \le c L^{\nicefrac{-1}{2}}.
$$
Letting first $n\to\infty $ and then $L\to\infty$ gives \eqref{e_L1} and
\eqref{e kuka}.
\end{proof}

{\bf Acknowledgments.} This research is supported by the
Sloan and Connaught grants, the NSERC discovery grant
program, and the Canada Research Chair program (Vir\'ag).
Valk\'o is partially supported by the Hungarian Scientific
Research Fund grant K60708. We thank Yuval Peres for
comments simplifying the proof of Proposition
\ref{p_turboEK}, and also Mu Cai, Laure Dumaz, Peter
Forrester and Brian Sutton for helpful comments. We are
indebted to the anonymous referees for their extensive
comments and suggestions.
\bibliography{sse}

\end{document}